\documentclass[11pt]{article}
\usepackage{smile}

\usepackage{fullpage}
\usepackage{lscape}
\usepackage{bigints}
\usepackage{framed}
\usepackage{mdframed}
\usepackage{enumerate}
\usepackage[inline]{enumitem}
\usepackage[T1]{fontenc}
\usepackage{moresize}
\usepackage{bm}
\usepackage{bbm}
\usepackage{dsfont}
\usepackage{amsmath}
\usepackage{amssymb}
\usepackage{amsthm}
\usepackage{amsfonts}
\usepackage{stmaryrd}
\usepackage{array}
\usepackage{mathrsfs}
\usepackage{mathtools} 
\usepackage{extarrows}
\usepackage{stackrel}
\usepackage{relsize,exscale}
\usepackage{scalerel}
\usepackage[nodisplayskipstretch]{setspace}
\usepackage{color}
\usepackage[usenames,dvipsnames]{xcolor}
\usepackage{cancel}
\usepackage{soul}
\usepackage{undertilde}
\usepackage{xfrac}
\usepackage{siunitx}
\usepackage{graphicx}
\usepackage{float}
\usepackage{rotating}
\usepackage{subcaption}
\usepackage{overpic}
\usepackage[all]{xy}
\usepackage{tikz}
\usetikzlibrary{arrows,matrix,positioning,calc,automata,patterns}
\usepackage{booktabs}
\usepackage{dcolumn}
\usepackage{multirow}
\usepackage{diagbox}
\usepackage{tabularx}
\usepackage{verbatim}
\usepackage{listings}
\usepackage[ruled,vlined]{algorithm2e}
\usepackage{fancyvrb}
\usepackage{hyperref}
\usepackage[round]{natbib}
\usepackage{sectsty}

\hypersetup{
    bookmarks=true,         
    unicode=false,          
    pdftoolbar=true,        
    pdfmenubar=true,        
    pdffitwindow=false,     
    pdfstartview={FitH},    
    pdftitle={My title},    
    pdfauthor={Author},     
    pdfsubject={Subject},   
    pdfcreator={Creator},   
    pdfproducer={Producer}, 
    pdfkeywords={key1, key2}, 
    pdfnewwindow=true,      
    colorlinks=true,        
    linkcolor=blue,         
    citecolor=blue,         
    filecolor=blue,         
    urlcolor=cyan           
}

\usepackage{stackengine}
\stackMath
\newcommand\tenq[2][1]{%
\def\useanchorwidth{T}%
\ifnum#1>1%
\stackunder[0pt]{\tenq[\numexpr#1-1\relax]{#2}}{\!\scriptscriptstyle\thicksim}%
\else%
\stackunder[1pt]{#2}{\!\scriptstyle\thicksim}%
\fi%
}

\makeatletter
\DeclareRobustCommand\widecheck[1]{{\mathpalette\@widecheck{#1}}}
\def\@widecheck#1#2{%
    \setbox\z@\hbox{\m@th$#1#2$}%
    \setbox\tw@\hbox{\m@th$#1%
       \widehat{%
          \vrule\@width\z@\@height\ht\z@
          \vrule\@height\z@\@width\wd\z@}$}%
    \dp\tw@-\ht\z@
    \@tempdima\ht\z@ \advance\@tempdima2\ht\tw@ \divide\@tempdima\thr@@
    \setbox\tw@\hbox{%
       \raise\@tempdima\hbox{\scalebox{1}[-1]{\lower\@tempdima\box
\tw@}}}%
    {\ooalign{\box\tw@ \cr \box\z@}}}
\makeatother

\def\given{\,|\,}
\def\biggiven{\,\big{|}\,}
\def\Biggiven{\,\Big{|}\,}
\def\tr{\mathop{\text{tr}}\kern.2ex}
\def\tZ{{\tilde Z}}
\def\tX{{\tilde X}}

\def\tV{{\tilde V}}
\def\tW{{\tilde W}}
\def\P{{\mathrm P}}

\def\E{{\mathrm E}}

\def\d{{\mathrm d}}

\newcommand{\zahl}[1]{\llbracket #1\rrbracket}
\newcommand\yestag{\addtocounter{equation}{1}\tag{\theequation}}
\newcolumntype{L}[1]{>{\raggedright\let\newline\\\arraybackslash\hspace{0pt}}m{#1}}
\newcolumntype{C}[1]{>{  \centering\let\newline\\\arraybackslash\hspace{0pt}}m{#1}}
\newcolumntype{R}[1]{>{ \raggedleft\let\newline\\\arraybackslash\hspace{0pt}}m{#1}}
\newcolumntype{d}[1]{D{.}{.}{#1}}
\newcolumntype{H}{>{\setbox0=\hbox\bgroup}c<{\egroup}@{}}
\newcolumntype{Z}{>{\setbox0=\hbox\bgroup}c<{\egroup}@{\hspace*{-\tabcolsep}}}
\newcolumntype{b}{X}
\newcolumntype{s}{>{\hsize=.5\hsize}X}

\numberwithin{equation}{section}

\newtheorem{theorem}{Theorem}[section]
\newtheorem{lemma}{Lemma}[section]
\newtheorem{proposition}{Proposition}[section]
\newtheorem{assumption}{Assumption}[section]
\newtheorem{corollary}{Corollary}[section]
\newtheorem{innercustomgeneric}{\customgenericname}
\providecommand{\customgenericname}{}
\newcommand{\newcustomtheorem}[2]{%
  \newenvironment{#1}[1]
  {%
   \renewcommand\customgenericname{#2}%
   \renewcommand\theinnercustomgeneric{##1}%
   \innercustomgeneric
  }
  {\endinnercustomgeneric}
}
\newcustomtheorem{customdefinition}{Definition}
\newcustomtheorem{customdefinitions}{Definitions}
\newcustomtheorem{customtheorem}{Theorem}
\newcustomtheorem{customassumption}{Assumption}
\newcustomtheorem{customlemma}{Lemma}
\newcustomtheorem{customexample}{Example}
\theoremstyle{definition}
\newtheorem{definition}{Definition}[section]

\newtheorem{remark}{Remark}[section]

\usepackage{enumitem}
\makeatletter
\newcommand{\mylabel}[2]{#2\def\@currentlabel{#2}\label{#1}}
\makeatother

\graphicspath{{./fig3/}}

\allowdisplaybreaks

\begin{document}

\setlength{\abovedisplayskip}{5pt}
\setlength{\belowdisplayskip}{5pt}
\setlength{\abovedisplayshortskip}{5pt}
\setlength{\belowdisplayshortskip}{5pt}
\hypersetup{colorlinks,breaklinks,urlcolor=blue,linkcolor=blue}

\title{\LARGE Estimation based on nearest neighbor matching: from density ratio to average treatment effect}

\author{Zhexiao Lin\thanks{Department of Statistics, University of Washington, Seattle, WA 98195, USA; e-mail: {\tt zxlin@uw.edu}}, ~~Peng Ding\thanks{Department of Statistics, University of California, Berkeley, CA 94720, USA; email: {\tt pengdingpku@berkeley.edu}}, ~and~Fang Han\thanks{Department of Statistics, University of Washington, Seattle, WA 98195, USA; e-mail: {\tt fanghan@uw.edu}}
}

\date{\today}

\maketitle

\vspace{-1em}

\begin{abstract}
Nearest neighbor (NN) matching as a tool to align data sampled from different groups is both conceptually natural and practically well-used. In a landmark paper, \cite{abadie2006large} provided the first large-sample analysis of NN matching under, however, a crucial assumption that the number of NNs, $M$, is fixed. This manuscript reveals something new out of their study and shows that, once allowing $M$ to diverge with the sample size, an intrinsic statistic in their analysis actually constitutes a consistent estimator of the density ratio. Furthermore, through selecting a suitable $M$, this statistic can attain the minimax lower bound of estimation over a Lipschitz density function class. Consequently, with a diverging $M$, the NN matching provably yields a doubly robust estimator of the average treatment effect and is semiparametrically efficient if the density functions are sufficiently smooth and the outcome model is appropriately specified. It can thus be viewed as a precursor of double machine learning estimators.
\end{abstract}



\section{Introduction}\label{sec:intro}

With observations from different groups, matching methods \citep{greenwood1945experimental, chapin1947experimental} aim to balance them through minimizing group differences in observed covariates. Such methods have proven their usefulness for causal inference in various disciplines,  including economics \citep{imbens2004nonparametric}, epidemiology \citep{brookhart2006variable}, political science \citep{ho2007matching,sekhon2008multivariate}, sociology \citep{morgan2006matching}, and statistics \citep{cochran1973controlling,rubin2006matched,rosenbaum2010design}. 

Among all the matching methods, nearest neighbor (NN) matching  \citep{rubin1973matching} is likely the most well-used and easiest to implement approach. In addition, it is computationally attractive as the time complexity of locating NNs is low. In the simplest treatment-control study, NN matching assigns each treatment (control) individual to $M$ control (treatment) individuals with the smallest distance to it. In this regard, two natural questions arise. First, how do we select the number of matches, $M$? This is referred to in the literature as ratio matching, and is both important and delicate, well-known to be related to the bias-variance tradeoff in nonparametic statistics \citep{smith19976,rubin2000combining}. Second, how do we perform large-sample statistical inference for NN matching methods? Such an analysis is usually nonstandard and thus believed to be mathematically challenging. Indeed, it was long-lacking in the literature until \cite{abadie2006large}.

To answer the above two questions, in a series of ingenious papers,  \cite{abadie2006large,abadie2008failure,abadie2011bias,abadie2012martingale} established large-sample properties of $M$-NN matching for estimating the average treatment effect (ATE). These results are, however, only valid under a crucial assumption that, in ratio matching, $M$ is fixed. The according message is then mixed. As a matter of fact, 
\cite{abadie2006large} argued --- which we quote here --- that the ATE estimator based on $M$-NN matching with a fixed $M$ is both asymptotically biased and statistically inefficient, namely, it ``does not achieve the semiparametric efficiency bound as calculated by \cite{hahn1998role}''. While bias correction is now feasible to alleviate the first issue \citep{abadie2011bias}, the lack of efficiency seems fundamental. 

This manuscript revisits the study of \cite{abadie2006large} from a new perspective, bridging $M$-NN matching to density ratio estimation \citep{nguyen2010estimating,sugiyama2012density} as well as double robustness \citep{scharfstein1999adjusting,bang2005doubly}. To this end, our analysis stresses, in ratio matching, the importance of forcing $M$ to diverge with the sample size $n$ in order to achieve statistical efficiency. Our claim is thus aligned with similar ones in various other random graph-based inference problems \citep{MR2083,MR532236,MR947577,MR1212489,MR1701112,MR3909934,MR3961499}, which also include a series of results made by some of the authors in this paper \citep{shi2021ac,shi2020power,lin2021boosting}.

The contributions of this manuscript are mainly two-fold. First, we show that an intrinsic statistic that plays a central role in the analysis of \cite{abadie2006large}, $K_M(x)$ (\citet[Page 240]{abadie2006large}; to be defined in \eqref{eq:KM} of Section \ref{sec:method}), actually gives rise to a consistent density ratio estimator in the two-sample setting. 
Even more interestingly, from the angle of density ratio estimation, this NN matching-based estimator is to our knowledge the first one that simultaneously satisfies the following three properties.
\begin{enumerate}
\item[\mylabel{itm:1}{(P1)}]  Conceptually {\it one-step}: it directly estimates the density ratio with no need to estimate individual densities, and is thus in line with Vapnik's rule that ``[w]hen solving a problem of interest, one should not solve a more general problem as an intermediate step'' \citep[on the top of Page 477]{vapnik2006estimation}.
\item[\mylabel{itm:2}{(P2)}]  Computationally {\it of low complexity}: it is of a sub-quadratic (and nearly linear when $M$ is small) time complexity  via a careful algorithmic formulation based on $k$-d trees (cf. Algorithms \ref{alg:1}-\ref{alg:2} and Theorem \ref{prop:complexity} ahead), and thus in many scientific applications is computationally more attractive than its optimization-based alternatives \citep{lima2008estimating,kremer2015nearest,borgeaud2021improving}.
\item[\mylabel{itm:3}{(P3)}] Statistically {\it rate-optimal}: it is information-theoretically efficient in terms of achieving an upper bound of estimation accuracy that matches the corresponding minimax lower bound over a class of Lipschitz density functions (cf. the set of results in Section \ref{sec:theory} ahead).
\end{enumerate} 
This estimator itself is accordingly an appealing alternative to existing density ratio estimators. Moreover, it is potentially useful in many data analysis problems (e.g., $f$-divergence estimation, classification, importance sampling, and etc.) where density ratio estimation plays a pivotal role. 

Getting back to the original ATE estimation problem, our second contribution is to bridge the above insights to the bias-corrected matching-based estimator proposed in \cite{abadie2011bias} as well as the double robustness and double machine learning framework introduced in \cite{scharfstein1999adjusting}, \cite{bang2005doubly}, and \cite{chernozhukov2018double}. In fact, their bias-corrected estimator can be formulated as 
\[
\hat{\tau}_M^{\rm bc} = \hat{\tau}^{\rm reg} + \frac{1}{n} \Big[ \sum_{i=1,D_i = 1}^n \Big(1 + \frac{K^1_M(i)}{M}\Big) \hat{R}_i - \sum_{i=1,D_i = 0}^n \Big(1 + \frac{K^0_M(i)}{M}\Big) \hat{R}_i \Big]
\]
(notation to be introduced in Section \ref{sec:matching}), where $1 + K^1_M(\cdot)/M$ and $1 + K^0_M(\cdot)/M$ approach the inverse of the propensity score $e(x)$ and $1-e(x)$, respectively. One could then leverage the general double robustness and double machine learning theory to validate the following two claims of (a double machine learning version of) $\hat{\tau}_M^{\rm bc}$.
\begin{itemize}
\item[(1)] {\it Consistency:} as long as either the density (propensity score) functions satisfy certain conditions or the outcome (regression) model is correctly specified, $M\log n/n\to 0$, and $M\to \infty$ as $n\to \infty$, $\hat{\tau}_M^{\rm bc}$ converges in probability to the population ATE, denoted as $\tau$.
\item[(2)] {\it Semiparametric efficiency:} if the density functions are sufficiently smooth, the outcome model is appropriately specified, and $M$ scales with $n$ at an proper rate, then (a sample-splitting and cross-fitting version of) $\hat{\tau}_M^{\rm bc}$ is an asymptotically normal estimator of $\tau$ with the asymptotic variance attaining the semiparametric efficiency lower bound \citep{hahn1998role}. Furthermore, a simple consistent estimator of the asymptotic variance is available.
\end{itemize}
Our results thus complement those made in \cite{abadie2006large,abadie2011bias}, rendering necessary confidence for practitioners to implement NN matching for inferring the ATE. In addition, although \citet[Theorem 5]{abadie2006large} hints at the necessity of allowing $M$ to diverge for gaining efficiency, we provide rigorous theory for their conjecture.

Technically speaking, our analysis hinges on those $M$'s that grow with $n$. Existing results in NN matching literature, including \cite{abadie2006large}, \cite{abadie2008failure}, \cite{abadie2011bias}, and \cite{abadie2012martingale}, are then limited as they are all focused on a fixed $M$. Instead, we take a different route to establish nonasymptotic moment bounds of $K_M(x)$, where there is more room for $M$ to move around; of note, similar ideas were also pursued in \cite{lin2021boosting} by the authors to analyze rank-based statistics. Detailed explanation of our theoretical analysis, however, has to be left to latter sections.

\paragraph*{Paper organization.} The rest of this manuscript is organized as follows. In Sections \ref{sec:method}-\ref{sec:theory} we introduce the method, computation, and theory of density ratio estimation via NN matching. In detail, Section \ref{sec:method} introduces the statistical setup and the matching-based estimator of density ratio. Section \ref{sec:computation} introduces the algorithms to implement the matching-based estimator constructed on the $k$-d tree structure. Section \ref{sec:theory} delivers the main theory, quantifying both pointwise and global approximation accuracy. Built on the previous three sections, Section \ref{sec:matching} formally elaborates on the double robustness and semiparametric efficiency of the bias-corrected NN matching-based estimator of the ATE. More applications of the proposed matching-based estimator will be covered in Section \ref{sec:discussion}, with proofs of the main results relegated to Section \ref{sec:main-proof} and the rest put in the appendix. Additional results about NN matching that cannot be incorporated in the double robustness and double machine learning framework are exhibited in the supplement. 

\paragraph*{Notation.}
For any integers $n,d\ge 1$, let $\zahl{n}:= \{1,2,\ldots,n\}$, $n!$ be the factorial of $n$, and $\bR^d$ be the $d$-dimensional real space. A set consisting of distinct elements $x_1,\dots,x_n$ is written as either $\{x_1,\dots,x_n\}$ or $\{x_i\}_{i=1}^{n}$, and its cardinality is written by $\lvert \{x_i\}_{i=1}^n \rvert$. The corresponding sequence is denoted $[x_1,\dots,x_n]$ or $[x_i]_{i=1}^{n}$.
The notation $\ind(\cdot)$ is saved for the indicator function.
For any $a,b \in \bR$, write $a \vee b = \max\{a,b\}$ and $a \wedge b = \min\{a,b\}$.
For any two real sequences $\{a_n\}$ and $\{b_n\}$, write $a_n \lesssim b_n$ (or equivalently, $b_n \gtrsim a_n$ or $a_n=O(b_n)$) if there exists a universal constant $C>0$ such that $a_n/b_n \le C$ for all sufficiently large $n$, and write $a_n \prec b_n$ (or equivalently, $b_n \succ a_n$ and $a_n=o(b_n)$) if $a_n/b_n \to 0$ as $n$ goes to infinity. We write $a_n\asymp b_n$ if both $a_n\lesssim b_n$ and $b_n\lesssim a_n$ holds. We use  $\stackrel{\sf d}{\longrightarrow}$ and $\stackrel{\sf p}{\longrightarrow}$ to denote convergence in distribution and in probability, respectively. For any sequence of random variables $\{X_n\}$, write $X_n = o_\P(1)$ if $X_n \stackrel{\sf p}{\longrightarrow} 0$. For any random variable $Z$, $\P_Z$ represents its law. Denote the closed ball in $\bR^d$ centered at $x$ with radius $\delta$ by $B_{x,\delta}$. In the sequel, let $c,C,C',C'',C''',...$ be generic positive constants whose actual values may change at different locations. 



\section{Density ratio estimation I: method}\label{sec:method}

From this section to Section \ref{sec:theory}, we consider $X,Z$ to be two general random vectors in $\bR^d$ that are defined on the same probability space, with $d$ to be a fixed positive integer. 

Let $\nu_0$ and $\nu_1$ represent the probability measures of $X$ and $Z$, respectively. Assume $\nu_0$ and $\nu_1$ are absolutely continuous with respect to the Lebesgue measure $\lambda$ on $\bR^d$ equipped with the Euclidean norm $\lVert \cdot \rVert$; denote the corresponding densities (Radon-Nikodym derivatives) by $f_0$ and $f_1$. Assume further that $\nu_1$ is absolutely continuous with respect to $\nu_0$ and write the corresponding density ratio, $f_1/f_0$, as $r$; we set $0/0=0$ by default.

Assume $X_1,\ldots,X_{N_0}$ are $N_0$ independent copies of $X$, $Z_1,\ldots,Z_{N_1}$ are $N_1$ independent copies of $Z$, and $[X_i]_{i=1}^{N_0}$ and $[Z_j]_{j=1}^{N_1}$ are mutually independent. The problem of estimating the density ratio $r$ based on $\{X_1,\ldots,X_{N_0},Z_1,\ldots,Z_{N_1}\}$ is fundamental in economics \citep{cunningham2021causal},   information theory \citep{cover1999elements}, machine learning \citep{sugiyama2012density}, statistics \citep{imbens2015causal}, and other fields. 

In density ratio estimation, NN-based estimators are advocated before due to its computational efficiency; cf. \cite{lima2008estimating}, \cite{poczos2011estimation}, \cite{kremer2015nearest}, \cite{noshad2017direct}, \cite{MR3909934}, \cite{zhao2020minimax}, among many others. 
In this manuscript, based on \cite{abadie2006large,abadie2008failure,abadie2011bias,abadie2012martingale}'s NN matching framework, we unveil a new density ratio estimator based on NN matching. To this end, some necessary notation is introduced first.

\begin{definition}[NN matching]\label{def:NN-matching}
 For any $x, z \in \bR^d$ and $M \in \zahl{N_0}$, 
\begin{enumerate}[itemsep=-.5ex,label=(\roman*)]
 \item let $\cX_{(M)}(\cdot): \bR^d \to \{X_i\}_{i=1}^{N_0}$ 
  be the mapping 
  that returns the value of the input $z$'s $M$-th NN in $\{X_i\}_{i=1}^{N_0}$, i.e., the value of $x \in \{X_i\}_{i=1}^{N_0}$ such that
  \begin{align}\label{eq:cX}
    \sum_{i=1}^{N_0} \ind\Big(\lVert X_i - z \rVert \le \lVert x - z \rVert\Big) = M;
  \end{align}
  \item\label{def:NN-matching-2}  let $K_M(\cdot): \bR^d\to \zahl{N_1}$ be the mapping  that returns the number of matched times of $x$, i.e., 
  \begin{align}\label{eq:KM}
    K_M(x) = K_M\big(x;\{X_i\}_{i=1}^{N_0},\{Z_j\}_{j=1}^{N_1}\big) := \sum_{j=1}^{N_1} \ind\Big(\lVert x - Z_j \rVert \le \lVert \cX_{(M)}(Z_j) - Z_j \rVert\Big);
  \end{align}
  \item\label{def:NN-matching-3} let $A_M(\cdot): \bR^d\to \cB(\bR^d)$ be the corresponding mapping from $\bR^d$ to the class of all Borel sets in $\bR^d$ so that
    \begin{align}\label{eq:AM}
    A_M(x) = A_M\big(x;\{X_i\}_{i=1}^{N_0}\big) := \Big\{z \in \bR^d : \lVert x - z \rVert \le \lVert \cX_{(M)}(z) - z \rVert\Big\},
  \end{align}
   returns the catchment area of $x$ in the setting of 
   \ref{def:NN-matching-2};
   \item for $i \in \zahl{N_0}$, let $K_M(i)$ and $A_M(i)$ be shorthands of $K_M(X_i)$ and $A_M(X_i)$, respectively. 
\end{enumerate}
\end{definition}

\begin{remark}
Of note, since $\nu_0$ is absolutely continuous with respect to the Lebesgue measure, a solution to \eqref{eq:cX} exists and is unique. \citet[Pages 240 and 260]{abadie2006large} introduced the terms $K_M(\cdot)$ and $A_M(\cdot)$ to analyze the asymptotic behavior of their NN matching-based ATE estimator. We also adopt the terminology ``catchment area'' in Definition \ref{def:NN-matching}\ref{def:NN-matching-3} to align with them.  
\end{remark}

The following proposition formally links $K_M(\cdot)$ to $A_M(\cdot)$, was shown and used in the proof of \citet[Lemma 3]{abadie2006large}, and is stated here for aiding understanding. 
\begin{proposition}[\cite{abadie2006large}]
  For any $x \in \bR^d$, we have
  \[
    K_M(x) = \sum_{j=1}^{N_1} \ind\Big(Z_j \in A_M(x)\Big).
  \]
\end{proposition}

\begin{remark}[Relation between $A_M(i)$'s and Voronoi tessellation when $M=1$]\label{remark:vt}
It is easy to verify that, due to the absolute continuity of $\nu_0$, $[A_1(i)]_{i=1}^{N_0}$ are almost surely disjoint except for a Lebesgue measure zero area, and partition $\bR^d$ into $N_0$ polygons. Furthermore,  one can verify that $\{A_1(i)\}_{i=1}^{N_0}$ are exactly the Voronoi tessellation defined in \cite{voronoi1908nouvelles}, which plays a vital role in stochastic and computational geometry. In this case, each element $A_1(i)$ is the Voronoi cell from the definition of \eqref{eq:AM}.
\end{remark}

With these notation and concepts, we are now ready to introduce the following density ratio estimator based on NN matching.

\begin{definition}[NN matching-based density ratio estimator]
  For any $M\in\zahl{N_0}$ and $x \in \bR^d$, we define the following estimator of $r(x)$,
  \begin{align}\label{eq:rhat}
    \hat{r}_M(x) = \hat{r}_M\Big(x;\{X_i\}_{i=1}^{N_0},\{Z_j\}_{j=1}^{N_1}\Big) := \frac{N_0}{N_1} \frac{K_M(x)}{M}.
  \end{align}
\end{definition}

The estimator $\hat r_M(\cdot)$ is by construction a one-step estimator, and thus automatically satisfies Property \ref{itm:1} in \hyperref[sec:intro]{Introduction}. In the following two sections, we will show that $\hat r_M(\cdot)$ also satisfies Properties \ref{itm:2} and \ref{itm:3}. 

\section{Density ratio estimation II: computation}\label{sec:computation}

This section discusses implementation of and establishes Property \ref{itm:2} for the proposed estimator $\hat r_M(\cdot)$. To this end, we separately discuss two cases:
\begin{itemize}
\item[] {\bf Case I}: estimating only the values of $\hat r_M(\cdot)$ at the observed data points $X_1,\ldots,X_{N_0}$;
\item[] {\bf Case II}: estimating the values of $\hat r_M(\cdot)$ at both the observed data points $X_1,\ldots,X_{N_0}$ and $n$ new points $x_1,\ldots,x_n\in \bR^d$.
\end{itemize} 

{\bf Case I.} In many applications, 
we are only interested in a functional of density ratios at observed sample points, i.e., the values of $\Phi\big(\{r(X_i)\}_{i=1}^{N_0}\big)$ for some given functions $\Phi$ defined on $\bR^{N_0}$. Check, e.g., \eqref{eq:kl} and -- in a slightly different but symmetric form -- \eqref{eq:mbc} ahead for such examples on Kullback-Leibler (KL) divergence  and ATE estimation. To this end, it is natural to consider the plug-in estimator $\Phi\big(\{\hat{r}_M(X_i)\}_{i=1}^{N_0}\big)$, for which it suffices to compute the values of $\{\hat{r}_M(X_i)\}_{i=1}^{N_0}$. 

Built on the  $k$-d tree structure \citep{bentley1975multidimensional} for tracking NNs, Algorithm~\ref{alg:1} below outlines an easy to implement algorithm to simultaneously compute all the values of $\{\hat{r}_M(X_i)\}_{i=1}^{N_0}$. This algorithm could be regarded as a direct extension of the celebrated Friedman-Bentley-Finkel algorithm \citep{10.1145/355744.355745} to the NN matching setting. 

\begin{algorithm}
\caption{Density ratio estimators at sample points.}\label{alg:1}
  \KwIn{$\{X_i\}_{i=1}^{N_0}$, $\{Z_j\}_{j=1}^{N_1}$, and $M$.}
  \KwOut{$\{\hat{r}_M(X_i)\}_{i=1}^{N_0}$.}
  Build a $k$-d tree using $\{X_i\}_{i=1}^{N_0}$\;
  \For{$j = 1:N_1$}{
    Search the $M$-NNs of $Z_j$ in $\{X_i\}_{i=1}^{N_0}$ using the $k$-d tree\;
    Store the indices of the $M$-NNs of $Z_j$ as $S_j$\;
  }
 Count and store the number of occurrence in $\bigcup_{j=1}^{N_1} S_j$ for each element in $\zahl{N_0}$, which is then $\{K_M(i)\}_{i=1}^{N_0}$\;
  Obtain $\{\hat{r}_M(X_i)\}_{i=1}^{N_0}$ based on \eqref{eq:rhat}.
\end{algorithm}

\vspace{0.3cm}

{\bf Case II.} Suppose we are interested in estimating density ratios at both the observed and $n$ new points in $\bR^d$. A naive algorithm is then to insert each new point into observed points and perform Algorithm~\ref{alg:1} in order. However, this algorithm is not ideal as the corresponding time complexity would be $n$ times the complexity of Algorithm~\ref{alg:1}, which could be computationally heavy with a large number of new points. 

Instead, we develop a more sophisticated implementation. Let the new points be $\{x_i\}_{i=1}^n$. Algorithm~\ref{alg:2} computes all the values of $\{\hat{r}_M(x_i)\}_{i=1}^n$ as well as $\{\hat{r}_M(X_i)\}_{i=1}^{N_0}$. The key message delivered here is that, compared to the aforementioned naive implementation, in Algorithm~\ref{alg:2} we only need to construct one single $k$-d tree;  the matching elements are then categorized to two different sets, corresponding to those with regard to $X_i$'s and $x_i$'s, separately. Such an implementation is thus intuitively much more efficient.

\begin{algorithm}[t]
\caption{Density ratio estimators at both sample and new points.}\label{alg:2}
  \KwIn{$\{X_i\}_{i=1}^{N_0}$, $\{Z_j\}_{j=1}^{N_1}$, $M$, and new points $\{x_i\}_{i=1}^n$.}
  \KwOut{$\{\hat{r}_M(X_i)\}_{i=1}^{N_0}$ and $\{\hat{r}_M(x_i)\}_{i=1}^n$.}
  Build a $k$-d tree using $\{X_i\}_{i=1}^{N_0} \cup \{x_i\}_{i=1}^n$\;
  \For{$j = 1:N_1$}{
    Set $S_j$ and $S'_j$ be two empty sets\;
    $m \gets 1$\;
    \While{$\lvert S_j \rvert < M$}{
      Search the $m$-th NN of $Z_j$ in $\{X_i\}_{i=1}^{N_0} \cup \{x_i\}_{i=1}^n$\;
      \eIf{the $m$-th NN of $Z_j$ is in $\{X_i\}_{i=1}^{N_0}$}{
        add the index into $S_j$\;
      }{
        add the index into $S'_j$\;
      }
      $m \gets m+1$\;
    }
    Store the indices sets $S_j$ and $S'_j$;
  }
  Count and store the number of occurrence in $\bigcup_{j=1}^{N_1} S_j$ for each element in $\zahl{N_0}$, which is then $\{K_M(i)\}_{i=1}^{N_0}$. Count and store the number of occurrence in $\bigcup_{j=1}^{N_1} S'_j$ for each element in $\zahl{n}$, which is then $\{K_M(x_i)\}_{i=1}^n$\;
  Obtain $\{\hat{r}_M(X_i)\}_{i=1}^{N_0}$ and $\{\hat{r}_M(x_i)\}_{i=1}^n$ based on \eqref{eq:rhat}.
\end{algorithm}

The following is the main theorem of this section, elaborating on the computational advantage of the proposed estimator. 


\begin{theorem}\phantomsection \label{prop:complexity}
\begin{itemize}
\item[(1)] The average time complexity of Algorithm~\ref{alg:1} to compute all the values of $\{\hat{r}_M(X_i)\}_{i=1}^{N_0}$  is 
\[
O\Big( (d+ N_1M/N_0)N_0\log N_0\Big).
\]
\item[(2)] Assume $[x_i]_{i=1}^n$ are independent and identically distributed (i.i.d.) following $\nu_0$ and are independent of $[X_i]_{i=1}^{N_0}$. Then the average time complexity of Algorithm~\ref{alg:2} to compute all the values of $\{\hat{r}_M(x_i)\}_{i=1}^n$ and $\{\hat{r}_M(X_i)\}_{i=1}^{N_0}$ is 
\[
O\Big( (d + N_1M/N_0) (N_0 + n) \log(N_0 + n) \Big).
\]
\end{itemize}
\end{theorem}

\begin{remark}[Comparison to non-NN-based estimators]
Assuming $N_0\asymp N_1\asymp N$, it is worth noting that optimization-based methods are commonly of a time complexity $O(N^2)$ if not worse \citep{noshad2017direct}. They are thus less appealing in terms of handling gigantic data as was argued in, e.g., astronomy \citep{lima2008estimating,kremer2015nearest} and big text analysis \citep{borgeaud2021improving} applications.
\end{remark}

\begin{remark}[Comparison to the two-step NN-based density ratio estimator]
Regarding Case I, a direct calculation yields that the time complexity of the simple two-step NN-based method, which separately estimates $f_1$ and $f_0$ based on individual $M$-NN density estimators,  is 
\[
O(d N_0 \log N_0 + d N_1 \log N_1 + N_0 M \log N_0 + N_0 M \log N_1).
\]
It is thus of the same order as Algorithm~\ref{alg:1} when $N_1 \asymp N_0$,  while computationally heavier when $N_1 \prec N_0$. Regarding Case II, the time complexity of the simple two-step NN-based method is 
\[
O(d N_0 \log N_0 + d N_1 \log N_1 + (N_0+n)M\log N_0 + (N_0+n)M\log N_1). 
\]
Thus, if $n$ is of less or equal order of $N_0$, it is of the same order when $N_1 \asymp N_0$, while computationally heavier than Algorithm~\ref{alg:2} when $N_1 \prec N_0$.
\end{remark}

\begin{remark}[Comparison to the one-step NN-based density ratio estimator in \cite{noshad2017direct}]\label{remark:hero} It is worth noting that, in order to estimate $f$-divergence measures, \cite{noshad2017direct} constructed another one-step NN-based estimator admitting the following simple form,
\[
\hat r_M^{\prime}(x)=\frac{N_0}{N_1}\frac{\cM_i}{\cN_i+1},
\]
where $\cN_i$ and $\cM_i$ are the numbers of points in $\{X_i\}_{i=1}^{N_0}$ and $\{Z_i\}_{i=1}^{N_1}$ among the $M$ NNs of $x$; cf. \citet[Equ. (20)]{noshad2017direct}. For Case I, its time complexity is
\[
O(d (N_0+N_1) \log (N_0+N_1) + N_0 M \log (N_0+N_1));
\]
while for Case II it is
\[
O(d (N_0+N_1) \log (N_0+N_1) + (N_0+n) M \log (N_0+N_1)).
\]
Both are at the same order as the naive NN-based one, but unlike the naive approach, this estimator is indeed one-step. However, it is still theoretically unclear if this estimator is statistically efficient; see Remark \ref{remark:hero2} ahead for more details.
\end{remark}

\section{Density ratio estimation III: theory}\label{sec:theory}

This section introduces the theory for density ratio estimation based on NN matching. To this end, before establishing detailed theoretical properties (e.g., consistency and the rate of convergence) for $\hat r_M(\cdot)$, we first exhibit a lemma elaborating on the (asymptotic) $L^p$ moments of $\nu_1(A_M(x))$, the $\nu_1$-measure of the catchment area. This novel result did not appear in Abadie and Imbens's analysis. It is also of independent interest in stochastic and computational geometry in light of Remark~\ref{remark:vt}.


\begin{lemma}[Asymptotic $L^p$ moments of catchment areas's $\nu_1$-measure]
\label{lemma:moment,catch} Assuming $M\log N_0/N_0 \to 0$ as $N_0 \to \infty$, we have
\[
    \lim_{N_0\to\infty} \frac{N_0}{M} \E\big[\nu_1\big(A_M(x)\big)\big] = r(x)
\]
holds for $\nu_0$-almost all $x$. If we further assume $M \to \infty$, then for any positive integer $p$, we have
  \[
    \lim_{N_0\to\infty} \Big(\frac{N_0}{M}\Big)^p \E\big[\nu_1^p\big(A_M(x)\big)\big] = \big[r(x)\big]^p
  \]
  holds for $\nu_0$-almost all $x$.
\end{lemma}

\begin{remark}[Relation to the measure of Voronoi cells]
  When $M=1$ and $\nu_0 = \nu_1$, the measure of catchment areas reduces to the measure of Voronoi cells as pointed out in Remark~\ref{remark:vt}. Interestingly, in the stochastic geometry literature, \cite{devroye2017measure} studied a related problem of bounding the moments of the measure of Voronoi cells (cf. Theorem 2.1 therein). Setting $M=1$ and $\nu_0 = \nu_1$ in the first part of Lemma~\ref{lemma:moment,catch} and recalling Remark \ref{remark:vt}, we can derive their Theorem 2.1(i). On the other hand, \citet[Theorem 2.1(ii)]{devroye2017measure} showed that, as $\nu_0 = \nu_1$, $p=2$, and $d\leq 3$, unlike $(M^{-1}N_0)^2 \E[\nu_1^2\big(A_M(x)\big)]$, $N_0^2\E[\nu_1^2(A_1(x))]$ does not converge to 1; cf. \citet[Section 4.2]{devroye2017measure}. This supports the necessity of forcing $M \to \infty$ for stabilizing the moments of $\hat r_M(\cdot)$. 
\end{remark}

\subsection{Consistency}\label{subsec:cons}

We first establish the pointwise consistency of the estimator $\hat r_M(x)$ for estimating $r(x)$. This requires nearly no assumption on $\nu_0,\nu_1$ except for those made at the beginning of Section \ref{sec:method}, in line with similar observations made in NN-based density estimation \citep[Theorem 3.1]{MR3445317}. 


\begin{theorem}[Pointwise consistency]\label{thm:cons,lp} Assume $M\log N_0/N_0 \to 0$.
\begin{enumerate}[itemsep=-.5ex,label=(\roman*)]
\item\label{thm:cons,lp1} (Asymptotic unbiasedness) For $\nu_0$-almost all $x$, we have
  \[
    \lim_{N_0\to\infty} \E\big[\hat{r}_M(x)\big] = r(x).
  \]
\item\label{thm:cons,lp2} (Pointwise $L_p$ consistency) Let $p$ be any positive integer and assume $MN_1/N_0 \to \infty$ and $M \to \infty$ as $N_0 \to \infty$. Then for $\nu_0$-almost all $x$, we have
  \[
    \lim_{N_0\to\infty} \E \big[ \lvert \hat{r}_M(x) - r(x) \rvert^p\big] = 0.
  \]
\end{enumerate}
\end{theorem}

For evaluating the global consistency of the estimator, on the other hand, it is necessary to introduce the following (global) $L_p$ risk:
\begin{align}\label{eq:risk}
 L_p \text{ risk}:=& \E \Big[ \Big\lvert \hat{r}_M(X) - r(X) \Big\rvert^p ~\Big |~X_1,\ldots,X_{N_0},Z_1,\ldots,Z_{N_1}\Big] \notag\\
  =& \int_{\bR^d} \Big\lvert \hat{r}_M(x) - r(x) \Big\rvert^p f_0(x) \d x,
\end{align}
where $X$ is a copy drawn from $\nu_0$ that is independent of the data. For the $L_p$ risk consistency of the estimator, we impose conditions on $\nu_0$ and $\nu_1$ further as follows.

Denote the supports of $\nu_0$ and $\nu_1$ by $S_0$ and $S_1$, respectively. For any set $S\subset \bR^d$, denote the diameter of $S$ by 
\[
  {\rm diam}(S):= \sup_{x,z \in S} \Big\lVert x-z \Big\rVert.
\]

\begin{assumption} \phantomsection \label{asp:risk}
  \begin{enumerate}[itemsep=-.5ex,label=(\roman*)]
      \item $\nu_0, \nu_1$ are two probability measures on $\bR^d$, both are absolutely continuous with respect to $\lambda$, and $\nu_1$ is absolutely continuous with respect to $\nu_0$.
    \item There exists a constant $R>0$ such that ${\rm diam}(S_0) \le R$.
    \item There exist two constants $f_L,f_U>0$ such that for any $x \in S_0$ and $z \in S_1$, $f_L \le f_0(x) \le f_U$ and $f_1(z) \le f_U$.
    \item There exists a constant $a \in (0,1)$ such that for any $\delta \in (0,{\rm diam}(S_0)]$ and $z \in S_1$,
    \[
      \lambda(B_{z,\delta} \cap S_0) \ge a \lambda(B_{z,\delta}),
    \]
    recalling that $B_{z,\delta}$ represents the closed ball in $\bR^d$ with center at $z$ and radius $\delta$.
  \end{enumerate}
\end{assumption}

\begin{remark}
Assumption \ref{asp:risk} is standard in the literature for establishing the global consistency of density ratio estimators. The regularity conditions on the support ensure that the angle of the support is not too sharp, which trivially hold for any $d$-dimensional cube. These conditions were also enforced in \citet[Theorem 1]{nguyen2010estimating}, \citet[Assumption 1]{sugiyama2008direct}, \citet[Definition 1]{kpotufe2017lipschitz}, among many others.
\end{remark}

We then establish the $L_p$ risk consistency of the estimator via the Hardy–Littlewood maximal inequality \citep{stein2016singular}; cf. Lemma \ref{lemma:HL} ahead. Of note, this inequality was used in \cite{han2020optimal} in a relative manner in order to study the information-theoretical limit of entropy estimation.

\begin{theorem}[$L_p$ risk consistency]\label{thm:risk,lp}
Assume the pair of $\nu_0, \nu_1$ satisfies Assumption~\ref{asp:risk}. Let $p$ be any positive integer. Assume further that $M\log N_0/N_0 \to 0$, $MN_1/N_0 \to \infty$, and $M \to \infty$ as $N_0 \to \infty$. We then have
  \[
    \lim_{N_0\to\infty} \E \Big[ \int_{\bR^d} \Big\lvert \hat{r}_M(x) - r(x) \Big\rvert^p f_0(x) \d x \Big] = 0.
  \]
\end{theorem}


As a direct corollary of Theorem~\ref{thm:risk,lp}, one can obtain the limit of any finite moment of $\nu_1(A_M(\cdot))$ with a random center. This could be regarded as a global extension of Lemma \ref{lemma:moment,catch}.



\begin{corollary}\label{crl:moment,catch,random}
Assume the same conditions as in Theorem \ref{thm:risk,lp}. 
We then have
  \[
    \lim_{N_0\to\infty} \Big(\frac{N_0}{M}\Big)^p \E\big[\nu_1^p\big(A_M(W)\big)\big] = \E \big(\big[r(W)\big]^p\big),
  \]
where $W$ follows an arbitrary distribution that is absolutely continuous with respect to $\nu_0$ and has density bounded upper and below by two positive constants. In particular, it holds when $W$ is drawn from $\nu_0$.
\end{corollary}


\subsection{Rates of Convergence}\label{subsec:rate}

In this section we establish the rates of convergence for $\hat r(x)$ under both pointwise and global measures. We first consider the pointwise mean square error (MSE) convergence rate and show that $\hat r_M(\cdot)$ is minimax optimal in that regard. In the sequel, we fix an $x \in \bR^d$ and consider the following local assumption on $(\nu_0,\nu_1)$.

\begin{assumption}[Local assumption]  \phantomsection \label{asp:ptw} 
  \begin{enumerate}[itemsep=-.5ex,label=(\roman*)]
    \item $\nu_0, \nu_1$ are two probability measures on $\bR^d$, both are absolutely continuous with respect to $\lambda$, and $\nu_1$ is absolutely continuous with respect to $\nu_0$.
    \item There exist two constants $f_L,f_U>0$ such that $f_0(x) \ge f_L$ and $f_1(x) \le f_U$.
    \item There exists a constant $\delta>0$ such that for any $z \in B_{x,\delta}$,
    \[
      \lvert f_0(x) - f_0(z) \rvert \vee \lvert f_1(x) - f_1(z) \rvert \le L \lVert x-z \rVert,
    \]
    for some constants $L>0$.
  \end{enumerate}
\end{assumption}

Define the following probability class 
\[
\cP_{x,\rm p}(f_L,f_U,L,d,\delta):=\Big\{(\nu_0,\nu_1): \text{Assumption \ref{asp:ptw} holds} \Big\}.
\]
The following theorem establishes the uniform pointwise convergence rate of $\hat r_M(\cdot)$.


\begin{theorem}[Pointwise rates of convergence]\label{thm:pw-rate} Assume $M\log N_0/N_0 \to 0$ and $M/\log N_0 \to \infty$, and consider a sufficiently large $N_0$.
\begin{enumerate}[itemsep=-.5ex,label=(\roman*)]
\item\label{thm:pw-rate1} Asymptotic bias: 
  \[
    \sup_{(\nu_0,\nu_1) \in \cP_{x,\rm p}(f_L,f_U,L,d,\delta)}\Big\lvert \E\big[\hat{r}_M(x)\big] - r(x) \Big\rvert \le C \Big(\frac{M}{N_0}\Big)^{1/d},
  \]
  where $C>0$ is a constant only depending on $f_L,f_U,L,d$.
\end{enumerate}  
Further assume $MN_1/N_0 \to \infty$.  
\begin{enumerate}[resume*]
  \item\label{thm:pw-rate2} Asymptotic variance:
  \[
    \sup_{(\nu_0,\nu_1) \in \cP_{x,\rm p}(f_L,f_U,L,d,\delta)}\Var\big[\hat{r}_M(x)\big] \le C' \Big[ \Big(\frac{1}{M}\Big) + \Big(\frac{N_0}{MN_1}\Big)\Big],
  \]
  where $C'>0$ is a constant only depending on $f_L,f_U$.
  \item Asymptotic MSE:   \[
    \sup_{(\nu_0,\nu_1) \in \cP_{x,\rm p}(f_L,f_U,L,d,\delta)} \E\big[\hat{r}_M(x) - r(x)\big]^2 \le C'' \Big[ \Big(\frac{M}{N_0}\Big)^{2/d} + \Big(\frac{1}{M}\Big) + \Big(\frac{N_0}{MN_1}\Big)\Big],
  \]
  where $C''>0$ is a constant only depending on $f_L,f_U,L,d$.
\end{enumerate}
Further assume $N_1^{-\frac{d}{2+d}}\log N_0 \to 0$. 
\begin{enumerate}[resume*]
 \item  Fix $\alpha > 0$ and take $M = \alpha \cdot\{N_0^{\frac{2}{2+d}} \vee (N_0 N_1^{-\frac{d}{2+d}})\}$. We have
  \begin{align}\label{eq:final-pw-rate}
    \sup_{(\nu_0,\nu_1) \in \cP_{x,\rm p}(f_L,f_U,L,d,\delta)} \E\big[\hat{r}_M(x) - r(x)\big]^2 \le C''' (N_0 \wedge N_1)^{-\frac{2}{2+d}},
  \end{align}
  where $C'''>0$ is a constant only depending on $f_L,f_U,L,d,\alpha$.
 \end{enumerate}
\end{theorem}

The final rate of convergence in Theorem \ref{thm:pw-rate}, Equ. \eqref{eq:final-pw-rate} matches the established minimax lower bound in Lipschitz density function estimation \citep[Section 2]{MR2724359}. By some simple manipulation, the argument in \citet[Exercise 2.8]{MR2724359} directly extends to density ratio as the latter is a harder statistical problem \citep[Remark 3]{kpotufe2017lipschitz}. This is formally stated in the following proposition.

\begin{proposition}[Pointwise MSE minimax lower bound]\label{thm:minimax,mse}
  For all sufficiently large $N_0$,
  \[
    \inf_{\tilde{r}} \sup_{(\nu_0,\nu_1) \in \cP_{x,\rm p}(f_L,f_U,L,d,\delta)} \E\big[\tilde{r}(x) - r(x)\big]^2 \ge c (N_0 \wedge N_1)^{-\frac{2}{2+d}},
  \]
  where $c>0$ is a constant only depending on $f_L,f_U,L,d$ and the infimum is taken over all measurable functions.
\end{proposition}

We then move on to the case of a global risk and study the rates of convergence in this regard. To this end, a global assumption on $(\nu_0,\nu_1)$ is given below.

\begin{assumption}[Global assumption]  \phantomsection \label{asp:global} 
  \begin{enumerate}[itemsep=-.5ex,label=(\roman*)]
    \item $\nu_0, \nu_1$ are two probability measures on $\bR^d$, both are absolutely continuous with respect to $\lambda$, and $\nu_1$ is absolutely continuous with respect to $\nu_0$.
    \item There exists a constant $R>0$ such that ${\rm diam}(S_0) \le R$.
    \item There exist two constants $f_L,f_U>0$ such that for any $x \in S_0$ and $z \in S_1$, $f_L \le f_0(x) \le f_U$ and $f_1(z) \le f_U$.
    \item There exists a constant $a \in (0,1)$ such that for any $\delta \in (0,{\rm diam}(S_0)]$ and any $z \in S_1$,
    \[
      \lambda(B_{z,\delta} \cap S_0) \ge a \lambda(B_{z,\delta}).
    \]
    \item There exists a constant $H>0$ such that the surface area (Hausdorff measure,  \citet[Section 3.3]{evans2018measure})
    of $S_1$ is bounded by $H$.
    \item There exists a constant $L>0$ such that for any $x,z \in S_1$,
    \[
      \lvert f_0(x) - f_0(z) \rvert \vee \lvert f_1(x) - f_1(z) \rvert \le L \lVert x-z \rVert.
    \]
  \end{enumerate}
\end{assumption}

\begin{remark}
Assumption \ref{asp:global} is standard in the literature for establishing the global risk of density ratio estimators; similar assumptions were made in \citet[Assumption 1]{zhao2020analysis} and \citet[Assumption 1]{zhao2020minimax}. Note that the regularity conditions on the support automatically hold for $d$-dimensional cubes, and the restriction on the surface area is added to control the boundary effect on NN-based methods. 
\end{remark}

Define the following probability class
\begin{align}\label{eq:prob-class-g}
\cP_{\rm g}(f_L,f_U,L,d,a,H,R):=\Big\{(\nu_0,\nu_1): \text{Assumption~\ref{asp:global} holds} \Big\}.
\end{align}
The next theorem establishes the uniform rate of convergence of $\hat r(\cdot)$ within the above probability class and under the $L_1$ risk. This rate is further matched by a minimax lower bound derived in Theorem 1 of \cite{zhao2020analysis} using similar arguments as in the pointwise case. 

\begin{theorem}[Global rates of convergence under the $L_1$ risk]\label{thm:rate,risk} Assume $M\log N_0/N_0 \to 0$, $M/\log N_0 \to \infty$, $MN_1/N_0 \to \infty$, and consider a sufficiently large $N_0$.
  \begin{enumerate}[itemsep=-.5ex,label=(\roman*)]
\item[(i)] We have the following uniform upper bound,
  \[
    \sup_{(\nu_0,\nu_1) \in \cP_{\rm g}(f_L,f_U,L,d,a,H,R)} \E \Big[ \int_{\bR^d} \Big\lvert \hat{r}_M(x) - r(x) \Big\rvert f_0(x) \d x \Big] \le C \Big[ \Big(\frac{M}{N_0}\Big)^{1/d} + \Big(\frac{1}{M}\Big)^{1/2} + \Big(\frac{N_0}{MN_1}\Big)^{1/2}\Big] ,
  \]
  where $C>0$ is a constant only depending on $f_L,f_U,a,H,L,d$.
\item[(ii)] Further assume $N_1^{-\frac{d}{2+d}}\log N_0 \to 0$, fix $\alpha > 0$, and take $M = \alpha\cdot\{ N_0^{\frac{2}{2+d}} \vee (N_0 N_1^{-\frac{d}{2+d}})\}$. We then have
  \[
    \sup_{(\nu_0,\nu_1) \in \cP_{\rm g}(f_L,f_U,L,d,a,H,R)} \E \Big[ \int_{\bR^d} \Big\lvert \hat{r}_M(x) - r(x) \Big\rvert f_0(x) \d x \Big] \le C' (N_0 \wedge N_1)^{-\frac{1}{2+d}},
  \]
  where $C'>0$ is a constant only depending on $f_L,f_U,a,H,L,d,\alpha$.
 \end{enumerate}
\end{theorem}

\begin{proposition}[Global minimax lower bound under the $L_1$ risk]\label{thm:minimax,risk}
  If $a$ is sufficiently small and $H,R$ are sufficiently large, then for all sufficiently large $N_0$,
  \[
    \inf_{\tilde{r}} \sup_{(\nu_0,\nu_1) \in \cP_{\rm g}(f_L,f_U,L,d,a,H,R)} \E \Big[ \int_{\bR^d} \Big\lvert \tilde{r}(x) - r(x) \Big\rvert f_0(x) \d x \Big] \ge c (N_0 \wedge N_1)^{-\frac{1}{2+d}},
  \]
  where $c>0$ is a constant only depending on $f_L,f_U,L,d$ and the infimum is taken over all measurable functions.
\end{proposition}

\begin{remark}[Comparison to the one-step estimator in \cite{noshad2017direct}]\label{remark:hero2}
The estimator introduced in Remark \ref{remark:hero} by \cite{noshad2017direct} is, to our knowledge, the only alternative density ratio estimator in the literature that is able to attain both the properties \ref{itm:1} and \ref{itm:2}. However, 
the arguments in \citet[Section III]{noshad2017direct} can only yield the  bound 
\[
  \E\big[\hat r_M^{\prime}(x) - r(x)\big]^2 \lesssim \Big(\frac{M}{N_0}\Big)^{1/d} + \Big(\frac{1}{M}\Big)
\]
for $(\nu_0,\nu_1) \in \cP_{x,\rm p}(f_L,f_U,L,d,\delta)$. This is via Equ. (21) therein, de-poissonizing the estimator, and further assuming $N_1/N_0$ converges to a positive constant. The above bound is strictly looser than the bound $(M/N_0)^{2/d}+M^{-1}$ for $\hat r_M(\cdot)$ shown in Theorem \ref{thm:pw-rate}. However, it seems mathematically challenging to improve their analysis and accordingly, unlike $\hat r_M(\cdot)$, it is still theoretically unclear if the estimator $\hat{r}_M^{\prime}(x)$ is a statistically efficient density ratio estimator. 
\end{remark}


\section{Revisiting the matching-based estimator of the ATE}\label{sec:matching}

This section studies the bias-corrected NN matching-based estimator of the ATE, proposed in \cite{abadie2011bias} to correct the asymptotic bias of the original matching-based estimator derived by \cite{abadie2006large}. To this end, we leverage the new insights obtained in the last three sections, and bridge the study to both the classic double robustness framework \citep{scharfstein1999adjusting,bang2005doubly} and the modern double machine learning framework \citep{chernozhukov2018double}.

We first introduce the setup for the NN matching-based estimator and its bias-corrected version. Following \cite{abadie2006large}, let $[(X_i,D_i,Y_i)]_{i=1}^n$ be $n$ independent independent copies of $(X,D,Y)$, where $D\in\{0,1\}$ is a binary variable, $X \in \bR^d$ represents the individual covariates, assumed to be absolute continuous admitting a density $f_X$, and $Y \in \bR$ stands for the outcome variable. 

For each unit $i \in \zahl{n}$, we observe $D_i=1$ if in the treated group and $D_i=0$ if in the control group. Let $n_0:=\sum_{i=1}^n (1-D_i)$ and $n_1:=\sum_{i=1}^n D_i$ be the numbers of control and treated units, respectively. Under the potential outcome framework \citep{rubin1974estimating}, the unit $i$ has two potential outcomes, $Y_i(1)$ and $Y_i(0)$, but we observe only one of them:
\[
    Y_i = \begin{cases}
        Y_i(0), & \mbox{ if } D_i=0,\\
        Y_i(1), & \mbox{ if } D_i=1.
    \end{cases}
\]

The goal is to estimate the following population ATE, 
\begin{align*}
    \tau := \E\Big[Y_i(1)-Y_i(0)\Big],
\end{align*}
based on the observations $\{(X_i, D_i, Y_i)\}_{i=1}^n$. To estimate ATE, we consider its empirical counterpart
\begin{align*}
    \hat{\tau}_M := \frac{1}{n} \sum_{i=1}^n \Big[\hat{Y}_i(1) -\hat{Y}_i(0)\Big], 
\end{align*}
where $\hat{Y}_i(0)$ and $\hat{Y}_i(1)$ are the imputed outcomes of $Y_i(0)$ and $Y_i(1)$. With a fixed $M$, the matching-based estimator in \cite{abadie2006large} imputes the missing potential outcomes by
\begin{align*}
    \hat{Y}_i(0) := \begin{cases}
        Y_i, & \mbox{ if } D_i=0,\\
        \frac{1}{M} \sum_{j \in \mathcal{J}^0_M(i)} Y_j, & \mbox{ if } D_i=1,     
    \end{cases}
~~~{\rm and}~~~    \hat{Y}_i(1) := \begin{cases}
        \frac{1}{M} \sum_{j \in \mathcal{J}^1_M(i)} Y_j, & \mbox{ if } D_i=0,\\   
        Y_i, & \mbox{ if } D_i=1.
    \end{cases}    
\end{align*}
Here for $\omega\in\{0,1\}$, $\mathcal{J}^\omega_M(i)$ represents the index set of $M$-NNs of $X_i$ in $\{X_j:D_j=\omega\}_{j=1}^n$, i.e., the set of all indices $j \in \zahl{n}$ such that $D_j=\omega$ and 
\[
    \sum_{\ell=1, D_\ell=\omega}^n \ind\Big(\lVert X_\ell -X_i \rVert \le \lVert X_j - X_i \rVert\Big) \le M.
\]
Let $K^\omega_M(i)$ be the number of matched times for unit $i$ such that $D_i=\omega$, i.e.,
\[
  K^\omega_M(i) := \sum_{j=1, D_j = 1-\omega}^n \ind\Big(i \in \cJ^\omega_M(j)\Big). 
\]
With the above notation and concepts, the matching-based estimator in \cite{abadie2006large} can be written as
\begin{align}\label{eq:m}
  \hat{\tau}_M = \frac{1}{n} \Big[ \sum_{i=1,D_i = 1}^n \Big(1 + \frac{K^1_M(i)}{M}\Big) Y_i - \sum_{i=1,D_i = 0}^n \Big(1 + \frac{K^0_M(i)}{M}\Big) Y_i\Big].
\end{align}

However, when $d>1$, the bias term of $\hat{\tau}_M$ is asymptotically non-negligible \citep{abadie2006large}. To fix this, \cite{abadie2011bias} proposed the following bias-corrected version for $\hat\tau_M$. In detail, let $\hat{\mu}_0(x)$ and $\hat{\mu}_1(x)$ be mappings from $\bR^d$ to $\bR$ that estimate the conditional means of the outcomes
\[
\mu_0(x) := \E [Y \given X=x,D=0]~~ {\rm and}~~ \mu_1(x) := \E [Y \given X=x,D=1], 
\]
respectively, with the corresponding residuals
\[
\hat{R}_i := Y_i - \hat{\mu}_{D_i}(X_i), ~~i\in\zahl{n}.
\]
The estimator based on the outcomes regression is
\[
\hat{\tau}^{\rm reg}:= n^{-1} \sum_{i=1}^n \Big[\hat{\mu}_1(X_i) - \hat{\mu}_0(X_i)\Big].
\]
One is then ready to check that the bias-corrected matching-based estimator in \cite{abadie2011bias} has the following equivalent form, summarized as a lemma.
\begin{lemma}\label{lemma:mbc}  The bias-corrected matching-based estimator in \cite{abadie2011bias} can be rewritten as
\begin{align}\label{eq:mbc}
  \hat{\tau}_M^{\rm bc} = \hat{\tau}^{\rm reg} + \frac{1}{n} \Big[ \sum_{i=1,D_i = 1}^n \Big(1 + \frac{K^1_M(i)}{M}\Big) \hat{R}_i - \sum_{i=1,D_i = 0}^n \Big(1 + \frac{K^0_M(i)}{M}\Big) \hat{R}_i \Big].
\end{align}
\end{lemma}

Equ. \eqref{eq:mbc} is related to doubly robust estimators. To compare them, we review the general double robustness framework. In detail, we first have some outcome models and residuals defined in the same way as above, and then let 
\[
\hat{e}(x):\bR^d \to \bR 
\]
be a genetic estimator of the propensity score
\[
e(x) := \P(D=1 \given X =x). 
\]
The doubly robust estimator in \cite{scharfstein1999adjusting} and \cite{bang2005doubly} could then be formulated as
\begin{align}\label{eq:dr}
  \hat{\tau}^{\rm dr} = \hat{\tau}^{\rm reg} + \frac{1}{n} \Big[ \sum_{i=1,D_i = 1}^n \frac{1}{\hat{e}(X_i)} \hat{R}_i - \sum_{i=1,D_i = 0}^n \frac{1}{1-\hat{e}(X_i)} \hat{R}_i \Big].
\end{align}

Notice that conditional on $\mD := (D_1, \ldots, D_n)$, $[X_i:D_i=\omega]_{i=1}^n$ are $n_{\omega}$ i.i.d. random variables sampled from the distribution of $X \given D=\omega$,  
and the two groups of sample points, 
\[
[X_i:D_i=0]_{i=1}^n ~~~{\rm and}~~~ [X_i:D_i=1]_{i=1}^n, 
\]
are mutually independent. Denote the density of $X \given D=\omega$ by $f_{X \given D=\omega}$. 
From the construction of $K_M^0(i), K_M^1(i)$ and results in Section \ref{sec:theory}, under the density assumptions and an appropriate choice of $M$, conditional on $\mD$, 
\[
\frac{n_0}{n_1} \frac{K^0_M(i)}{M}~~~{\rm and}~~~ \frac{n_1}{n_0} \frac{K^1_M(i)}{M}
\]
are consistent estimators of 
\[
\frac{f_{X \given D=1}(X_i)}{f_{X \given D=0}(X_i)} ~~~{\rm and}~~~\frac{f_{X \given D=0}(X_i)}{f_{X \given D=1}(X_i)},
\]
respectively. Noting further that $n_1/n_0$ converges almost surely to $\P(D=1)/\P(D=0)$ by the law of large numbers, the following two statistics
\[
\frac{K^0_M(i)}{M}~~~{\rm and }~~~\frac{K^1_M(i)}{M}
\]
 are then consistent estimators of 
 \[
  \frac{e(X_i)}{1-e(X_i)}~~~{\rm and }~~~\frac{1-e(X_i)}{e(X_i)},
 \]
respectively. Thus, in view of \eqref{eq:dr}, the bias-corrected matching-based estimator $\hat{\tau}_M^{\rm bc}$ in \eqref{eq:mbc} is actually a doubly robust estimator of $\tau$, and accordingly, should also enjoy all the desirable properties of doubly robust estimators. This novel insight into \cite{abadie2011bias}'s bias-corrected matching estimator allows us to establish its asymptotic properties with a diverging $M$, which is the main topic in the rest of this section. It is worth mentioning here that Abadie and Imbens never pointed out the relation between matching-based ATE estimators and the double robustness framework, and there is no study in the double robustness literature that has analyzed NN matching-based estimators.


To formally state the properties, we first leverage the results of 
\cite{chernozhukov2018double}. In the sequel, let $U_\omega := Y(\omega) - \mu_{\omega}(X)$ for $\omega \in \{0,1\}$ and $\bX$ be the support of $X$.

\begin{assumption}  \phantomsection \label{asp:dml1} 
  \begin{enumerate}[itemsep=-.5ex,label=(\roman*)]
    \item\label{asp:dml1-1} For almost all $x \in \bX$, $D$ is independent of $(Y(0),Y(1))$ conditional on $X=x$, and there exists some constant $\eta > 0$ such that $\eta < \P(D=1 \given X=x) < 1-\eta$.
    \item $[(X_i,D_i,Y_i)]_{i=1}^n$ are i.i.d. following the joint distribution of $(X,D,Y)$.
    \item $\E [U^2_\omega \given X=x] $ is uniformly bounded for almost all $x \in \bX$ and $\omega \in \{0,1\}$.
    \item $\E [\mu^2_\omega]$ is bounded for $\omega \in \{0,1\}$.
  \end{enumerate}
\end{assumption}

\begin{assumption}  \phantomsection \label{asp:dml2} 
  \begin{enumerate}[itemsep=-.5ex,label=(\roman*)]
    \item $\E [U^2_\omega]$ is bounded away from zero for $\omega \in \{0,1\}$.
    \item There exists some constants $\kappa>0$ such that $\E [\lvert Y \rvert^{2+\kappa}]$ is bounded.
  \end{enumerate}
\end{assumption}

\begin{assumption}\label{asp:dml3''}
  For $\omega \in \{0,1\}$, there exists a deterministic function $\bar{\mu}_\omega(\cdot):\bR^d \to \bR$  such that $\E [\bar{\mu}^2_\omega(X)]$ is bounded and the estimator $\hat{\mu}_\omega(x)$ satisfies
\[
\lVert \hat{\mu}_\omega - \bar{\mu}_\omega \rVert_\infty = o_\P(1),
\]
where $\lVert \cdot \rVert_\infty$ denotes the function $L_\infty$ norm.
\end{assumption}

\begin{assumption}\label{asp:dml3'}
  For $\omega \in \{0,1\}$, the estimator $\hat{\mu}_\omega(x)$ satisfies
\[
\lVert \hat{\mu}_\omega - \mu_\omega \rVert_\infty = o_\P(1).
\]
\end{assumption}

\begin{assumption}\label{asp:dml3}
   For $\omega \in \{0,1\}$, the estimator $\hat{\mu}_\omega(x)$ satisfies
\[
\lVert \hat{\mu}_\omega - \mu_\omega \rVert_\infty = o_\P(n^{-d/(4+2d)}).
\]
\end{assumption}

\begin{remark}
  Assumption~\ref{asp:dml1}\ref{asp:dml1-1} is the unconfoundedness and overlap assumptions, and is often referred to as the strong ignorability condition \citep{rosenbaum1983central}. Assumption~\ref{asp:dml2} corresponds to Assumption 5.1 in \cite{chernozhukov2018double}, and are similar to Assumption 4 in \cite{abadie2006large}. Assumption \ref{asp:dml3''} allows for the misspecification of the outcome models; for example, if $\hat{\mu}_\omega = \bar{\mu}_\omega = 0$,  $\hat{\tau}_M^{\rm bc}$ then reduces to $\hat{\tau}_M$. Assumption \ref{asp:dml3'} assumes that the outcome models are correctly specified.  Assumption~\ref{asp:dml3} assumes approximation accuracy of the outcome model. \cite{abadie2011bias} uses the power series approximation \citep{newey1997convergence} to estimate the outcome model, which under some classic nonparametric statistics assumptions automatically satisfy Assumption \ref{asp:dml3} (cf. Lemma A.1 in \cite{abadie2011bias}).
\end{remark}

\begin{remark}
 In Assumption~\ref{asp:dml3}, we assume an approximation rate under $L_\infty$ norm. This is different from the $L_2$ norm put in Assumption 5.1 in \cite{chernozhukov2018double}, but can be handled with some trivial modification to the proof of \citet[Theorem 5.1]{chernozhukov2018double} since one can replace the Cauchy–Schwarz inequality by the $L_1$-$L_\infty$ Hölder's inequality. Theorem~\ref{thm:rate,risk} can then be applied directly. 
\end{remark}

Lastly, we review the semiparametric efficiency lower bound for estimating ATE \citep{hahn1998role}:
\[
  \sigma^2:= \E \Big[\mu_1(X) - \mu_0(X) + \frac{D(Y-\mu_1(X))}{e(X)} - \frac{(1-D)(Y-\mu_0(X))}{1-e(X)} - \tau \Big]^2.
\]
Theorem~\ref{thm:risk,lp} above and standard results on the double machine learning estimators (cf. \citet[Theorem 5.1]{chernozhukov2018double}) then imply the following theorem, showing that, once allowing $M$ to diverge at an appropriate rate, $\hat{\tau}_M^{\rm bc}$ in \eqref{eq:mbc} constitutes a doubly robust estimator of $\tau$. In addition, the counterpart of $\hat{\tau}_M^{\rm bc}$ via  sample splitting and cross fitting \citep[Definition 3.1]{chernozhukov2018double}, denoted by $\tilde{\tau}_{M,K}^{\rm bc}$ with $K \ge 2$ representing a fixed number of partitions, is semiparametrically efficient. Check the beginning of Theorem \ref{thm:dml}'s proof for a formal definition of $\tilde{\tau}_{M,K}^{\rm bc}$.


\begin{theorem}  \phantomsection \label{thm:dml} 
\begin{enumerate}[itemsep=-.5ex,label=(\roman*)]
 \item\label{thm:dml1} (Double robustness) On one hand, if the distribution of $(X,D,Y)$ satisfies Assumptions~\ref{asp:dml1}, \ref{asp:dml3''}, and either $(\P_{X\given D=0}, \P_{X\given D=1})$ or  $(\P_{X\given D=1}, \P_{X\given D=0})$ satisfies Assumption~\ref{asp:risk}, then if $M\log n/n \to 0$ and $M \to \infty$ as $n \to \infty$, 
  \begin{align*}
    \hat{\tau}_M^{\rm bc} - \tau \stackrel{\sf p}{\longrightarrow} 0~~~{\rm and}~~~ \tilde{\tau}_{M,K}^{\rm bc} - \tau \stackrel{\sf p}{\longrightarrow} 0.
  \end{align*}
  On the other hand, if the distribution of $(X,D,Y)$ satisfies Assumptions~\ref{asp:dml1} and \ref{asp:dml3'}, then
  \begin{align*}
    \hat{\tau}_M^{\rm bc} - \tau \stackrel{\sf p}{\longrightarrow} 0~~~{\rm and}~~~ \tilde{\tau}_{M,K}^{\rm bc} - \tau \stackrel{\sf p}{\longrightarrow} 0.
  \end{align*}
  \item\label{thm:dml2} (Semiparametric efficiency of $\tilde{\tau}_{M,K}^{\rm bc}$) Assume the distribution of $(X,D,Y)$ satisfies Assumptions~\ref{asp:dml1}, \ref{asp:dml2}, \ref{asp:dml3} and either $(\P_{X\given D=0}, \P_{X\given D=1})$ or  $(\P_{X\given D=1}, \P_{X\given D=0})$ satisfies Assumption~\ref{asp:global}. Then if we pick $M = \alpha n^{\frac{2}{2+d}}$ for some constant $\alpha>0$, 
  \begin{align*}
   \sqrt{n} (\tilde{\tau}_{M,K}^{\rm bc} - \tau) \stackrel{\sf d}{\longrightarrow} N(0,\sigma^2).
  \end{align*}
  In addition, letting 
  \begin{align*}
  \hat{\sigma}^2:= \frac{1}{n} \sum_{i=1}^n \Big[\hat{\mu}_1(X_i) - \hat{\mu}_0(X_i) + D_i\Big(1 + \frac{K^1_M(i)}{M}\Big)\hat R_i -
   (1-D_i)\Big(1 + \frac{K^0_M(i)}{M}\Big)\hat R_i - \hat{\tau}_M^{\rm bc} \Big]^2,
\end{align*}
we have $\hat{\sigma}^2\stackrel{\sf p}{\longrightarrow}\sigma^2$.
\end{enumerate}
\end{theorem}

\begin{remark}\label{remark:dml}
Both $\hat\tau_{M}^{\rm bc}$ and $\tilde{\tau}_{M,K}^{\rm bc}$ directly estimate $1/e(X)$ and $1/(1-e(X))$ using $1+K_M^1(X)/M$ and $1+K_M^0(X)/M$, respectively. This is slightly different from the setup in \cite{chernozhukov2018double}, which considered plug-in estimators based on an estimate of $e(X)$. Accordingly, some minor modifications are needed for employing their Theorem 5.1, and are completed in the proof of Theorem \ref{thm:dml}\ref{thm:dml2}.
\end{remark}

\begin{remark}
  To be in line with the double robustness terminology, we can call Assumptions \ref{asp:risk} and  \ref{asp:global} in Theorem \ref{thm:dml}  the ``density (or propensity) model assumptions'' and Assumptions \ref{asp:dml3'} and \ref{asp:dml3} the ``outcome (or regression) model assumptions''. Note also that in Theorem \ref{thm:dml} we require additional smoothness (Lipschitz) conditions on the density functions. This is stronger than the corresponding conditions made in \cite{abadie2011bias} but is needed for us to prove the semiparametric efficiency of $\tilde{\tau}_{M,K}^{\rm bc}$'s based on the double machine learning framework. On the other hand, this requirement could be removed for $\hat{\tau}_M^{\rm bc}$ via a more careful treatment based on the particular structure of the matching-based estimators; cf. Theorem \ref{thm:mbc} and Remark \ref{remark:comp-to-dml} ahead. 
\end{remark}

The analysis of $\hat\tau_M^{\rm bc}$ itself, on the other hand, cannot be directly incorporated in the double machine learning and its analysis has to be done independently. This is given in the next theorem based on the following two sets of assumptions. 

\begin{assumption}  \phantomsection \label{asp:mbc1} 
  \begin{enumerate}[itemsep=-.5ex,label=(\roman*)]
    \item $\E [U^2_\omega \given X=x]$ is uniformly bounded away from zero for almost all $x \in \bX$ and $\omega \in \{0,1\}$.
    \item There exists some constants $\kappa>0$ such that $\E [\lvert U_\omega \rvert ^{2+\kappa} \given X=x]$ is uniformly bounded for almost all $x \in \bX$ and $\omega \in \{0,1\}$.
    \item\label{asp:mbc1-3} $\max_{t \in \Lambda_{\lfloor d/2 \rfloor + 1}} \lVert \partial^t \mu_{\omega} \rVert_\infty$ is bounded, where for any positive integer $k$, $\Lambda_k$ is the set of all $d$-dimensional vectors of nonnegative integers $t=(t_1,\ldots,t_d)$ such that $\sum_{i=1}^d t_i = k$ and $\lfloor \cdot \rfloor$ stands for the floor function.
  \end{enumerate}
\end{assumption}

\begin{assumption}  \phantomsection \label{asp:mbc2} 
  There exists some constant $\varepsilon>0$ such that for $\omega \in \{0,1\}$, the estimator $\hat{\mu}_\omega(x)$ satisfies
    \[
      \max_{t \in \Lambda_{\lfloor d/2 \rfloor + 1}} \lVert \partial^t \hat{\mu}_{\omega} \rVert_\infty = O_\P(1)~~~{\rm and}~~~ \max_{\ell \in \zahl{\lfloor d/2 \rfloor}} \max_{t \in \Lambda_\ell} \lVert \partial^t \hat{\mu}_{\omega} - \partial^t \mu_{\omega} \rVert_\infty = O_\P(n^{-\gamma_\ell}),
    \]
    with some constants $\gamma_\ell$'s satisfying $\gamma_\ell > \frac{1}{2} - \frac{\ell}{d} + \varepsilon$ for $\ell=1,2,\ldots,\lfloor d/2 \rfloor$.
\end{assumption}

\begin{theorem}[Semiparametric efficiency of $\hat\tau_M^{\rm bc}$]\label{thm:mbc} 
  Assume the distribution of $(X,D,Y)$ satisfies Assumptions~\ref{asp:dml1}, \ref{asp:mbc1}, \ref{asp:mbc2} and either $(\P_{X\given D=0}, \P_{X\given D=1})$ or  $(\P_{X\given D=1}, \P_{X\given D=0})$ satisfies Assumption~\ref{asp:risk}. Define
  \[
  \gamma := \min_{\ell \in \zahl{\lfloor d/2 \rfloor} }\Big\{ \Big[1-\Big(\frac{1}{2} - \gamma_\ell + \varepsilon\Big)\frac{d}{\ell}\Big] \wedge \Big[1- \Big(\frac{1}{2} + \varepsilon\Big)\frac{d}{\lfloor d/2 \rfloor + 1}\Big]\Big\};
  \]
 recall that $\gamma_\ell$'s and the constant $\varepsilon$ were introduced in Assumption \ref{asp:mbc2}. Then, if $M \to \infty$ as $n \to \infty$ and $M \lesssim n^\gamma$, 
  \begin{align*}
   \sqrt{n} (\hat\tau_M^{\rm bc} - \tau) \stackrel{\sf d}{\longrightarrow} N(0,\sigma^2).
  \end{align*}
If in addition Assumption~\ref{asp:dml3'} holds, we have $\hat{\sigma}^2\stackrel{\sf p}{\longrightarrow}\sigma^2$.
\end{theorem}

\begin{remark}
  Assumption~\ref{asp:mbc1} is comparable to Assumption A.4 and the assumptions of Theorem 2 in \cite{abadie2011bias}. Compared to the assumptions of Theorem 2 in \cite{abadie2011bias}, Assumption~\ref{asp:mbc1}\ref{asp:mbc1-3} is weaker in the sense that we only require a finite order of smoothness. Assumption~\ref{asp:mbc2} again assumes the approximation accuracy of the outcome model, with lower convergence rates required for higher order derivatives of the outcome model. We note that under some smoothness conditions on the outcome model as made in \cite{abadie2011bias}, Assumption~\ref{asp:mbc2} is satisfied using power series approximation \citep[Lemma A.1]{abadie2011bias}. 
\end{remark}

\begin{remark}\label{remark:comp-to-dml}
Due to the intrinsic structure of bias-corrected matching-based estimators, unlike a direct use of Hölder's inequality and assuming Assumption~\ref{asp:dml3} for double machine learning estimators, we instead assume Assumption~\ref{asp:mbc2} on the approximation accuracy of the derivatives of the outcome models. There are additionally two differences between Theorem \ref{thm:dml}\ref{thm:dml2} and Theorem \ref{thm:mbc}. First,  unlike in Theorem~\ref{thm:dml}\ref{thm:dml2} where $M$ needs to grow polynomially fast with $n$, in Theorem~\ref{thm:mbc} we only require $M$ to (i) diverge not so fast for controlling the difference of matching units; and (ii) diverge to infinity (no matter how slowly it is) for achieving semiparametric efficiency. The sets of assumptions in Theorems \ref{thm:dml}\ref{thm:dml2} and \ref{thm:mbc} both render semiparametric efficiency for bias-corrected matching-based estimators. Second, in Theorem \ref{thm:mbc} we only require Assumption~\ref{asp:risk} to hold for the density model. This is again weaker than the Lipschitz-type conditions (Assumption \ref{asp:global}) assumed in Theorem \ref{thm:dml}\ref{thm:dml2} but is in line with the observations made in \cite{abadie2006large} and \cite{abadie2011bias}.
\end{remark}

As $d = 1$, by picking $\hat{\mu}_\omega = 0$ for $\omega \in \{0,1\}$, Assumption~\ref{asp:mbc2} is automatically satisfied and the bias-corrected estimator $\hat\tau_M^{\rm bc}$ reduces to the original estimator $\hat\tau_M$ studied in \cite{abadie2006large}. Theorem \ref{thm:mbc} then directly implies the following corollary, which corresponds to \citet[Corollary 1]{abadie2006large} but with one key difference that $M\to\infty$ here. 

\begin{corollary}[Semiparametric efficiency of $\hat\tau_M$ when $d=1$]\label{crl:m} 
  Assume $d=1$, the distribution of $(X,D,Y)$ satisfies Assumptions~\ref{asp:dml1}, \ref{asp:mbc1}, and either $(\P_{X\given D=0}, \P_{X\given D=1})$ or  $(\P_{X\given D=1}, \P_{X\given D=0})$ satisfies Assumption~\ref{asp:risk}. If $M \to \infty$ as $n \to \infty$ and $M \lesssim n^{\frac{1}{2} - \varepsilon}$ for some $\varepsilon>0$, then we have
  \begin{align*}
   \sqrt{n} (\hat\tau_M - \tau) \stackrel{\sf d}{\longrightarrow} N(0,\sigma^2).
  \end{align*}
\end{corollary}

\begin{remark}
By picking $\hat{\mu}_\omega = 0$ for $\hat\tau_M$, Assumption~\ref{asp:dml3'} is no longer satisfied. Accordingly, in Corollary \ref{crl:m}, $\hat{\sigma}^2$ may not be a consistent estimator of $\sigma^2$ without additional assumptions. However, by decomposing $\sigma^2$ into the form of Theorem 1 in \cite{hahn1998role}, one could still estimate $\sigma^2$ via a similar and direct way as what is outlined in Section 4 in \cite{abadie2006large}. We do not pursue this track in detail here as the case of $d=1$ without Assumption~\ref{asp:dml3'} is beyond the main scope of this manuscript.
\end{remark}

\begin{remark}

There are three additional problems in \cite{abadie2006large, abadie2012martingale}. First, estimation of the average treatment effect on the treated (ATT) can be incorporated in the double robustness and double machine learning framework (Theorem~\ref{thm:dml}) and matching framework (Theorem~\ref{thm:mbc}) in the same way. Second, asymptotic normality (with an additional asymptotic bias term) of $\hat{\tau}_M$ in general $d$ can be established as Theorem~\ref{thm:mbc}. Third, unbalanced designs with $n_0 \succ n_1$ cannot be incorporated in the double robustness and double machine learning framework. In the supplement, we shall study the above unbalanced design case as $M$ is forced to diverge with $n$ and establish the corresponding properties.

\end{remark}


\section{Discussion}\label{sec:discussion}

The success of NN matching in estimating ATE raises natural question about possible extension to other functional estimation problems. In what follows we present some partial results along this direction for the KL divergence estimation, which is a topic of interest in many applications; cf. the references at the beginning of \cite{zhao2020minimax}. 

In detail, let's consider the same two-sample setting as Sections \ref{sec:method}-\ref{sec:theory}. Denote $\phi(x) := x\log x$ with the convention that $\phi(0) = 0$. We consider estimating the KL divergence between $\nu_1$ and $\nu_0$, written as 
\begin{align*}
  D_\phi(\nu_1 \Vert \nu_0) := \int \phi(r(x)) f_0(x) \d x,
\end{align*}
based on the data points $\{X_i\}_{i=1}^{N_0},\{Z_j\}_{j=1}^{N_1}$. To estimate $D_\phi(\nu_1 \Vert \nu_0)$, it is natural to consider the following plug-in estimator,
\begin{align}\label{eq:kl}
  \hat{D}_\phi = \hat{D}_\phi\Big(\{X_i\}_{i=1}^{N_0},\{Z_j\}_{j=1}^{N_1}\Big) := \frac{1}{N_0} \sum_{i=1}^{N_0} \phi\Big(\hat{r}_M(X_i)\Big),
\end{align}
where $\hat r_M(\cdot)$ is the NN matching-based density ratio estimator introduced in Section \ref{sec:method}.

Define the following probability class
\begin{align}\label{eq:prob-class-kl}
\cP_{\rm KL}(f_L,f_U,L,d,a,H,R,f_U'):=&\Big\{(\nu_0,\nu_1): \text{Assumption~\ref{asp:global} holds} \notag\\ 
&~~~~~~\text{and } \E_{X\sim \nu_0}\{[\log(f_1(X))]^4\} \le f_U' \Big\}.
\end{align}
The following results establish the MSE convergence rate of $\hat{D}_\phi$ and show that the estimator is minimax optimal over $\cP_{\rm KL}(f_L,f_U,L,d,a,H,R,f_U')$. Under the condition that $N_0\gtrsim N_1$, this rate is further (up to some $\log n$-terms) minimax optimal in view of \citet[Theorem 7]{han2020optimal} and similar arguments as used in Section \ref{sec:theory}.


\begin{theorem}[Rates of convergence, KL divergence estimation] \phantomsection \label{thm:rate,kl}
  \begin{enumerate}[itemsep=-.5ex,label=(\roman*)]
\item[(i)]  Assume $M\log N_0/N_0 \to 0$, $M/\log N_0 \to \infty$, $MN_1/(N_0\log^2N_1) \to \infty$, $N_1^{-\frac{d}{1+d}}\log N_0 \to 0$. Then for all sufficiently large $N_0$,
  \[
    \sup_{(\nu_0,\nu_1) \in \cP_{\rm KL}(f_L,f_U,L,d,a,H,R,f_U')} \E\Big[\hat{D}_\phi - D_\phi(\nu_1 \Vert \nu_0) \Big]^2 \le C \Big[ \frac{1}{N_0} + \frac{1}{N_1} + \Big(\frac{M}{N_0}\Big)^{2/d} + \Big(\frac{1}{M}\Big)^2 + \Big(\frac{N_0}{MN_1}\Big)^2\Big] ,
  \]
  where $C>0$ is a constant only depending on $f_L,f_U,a,H,L,d,f_U'$.
\item[(ii)]  Fix $\alpha > 0$, and take $M = \alpha (N_0^{\frac{1}{1+d}}) \vee (N_0 N_1^{-\frac{d}{1+d}})$. Then for all sufficiently large $N_0$,
  \[
    \sup_{(\nu_0,\nu_1) \in \cP_{\rm KL}(f_L,f_U,L,d,R,a,H,f_U')} \E\Big[\hat{D}_\phi - D_\phi(\nu_1 \Vert \nu_0) \Big]^2 \le C' (N_0 \wedge N_1)^{-\frac{2}{1+d}},
  \]
  where $C'>0$ is a constant only depending on $f_L,f_U,a,H,L,d,f_U',\alpha$.
\end{enumerate}
\end{theorem}

\begin{proposition}[MSE minimax rate, KL divergence estimation]\label{thm:minimax,kl}
  If $a$ is sufficiently small and $L,R,H,f_U'$ are sufficiently large, then for all sufficiently large $N_0$,
  \[
    \inf_{\tilde{D}_\phi} \sup_{(\nu_0,\nu_1) \in \cP_{\rm KL}(f_L,f_U,L,d,a,H,R,f_U')} \E\Big[\tilde{D}_\phi - D_\phi(\nu_1 \Vert \nu_0) \Big]^2 \ge c \Big[ \Big(\frac{1}{N_1}\Big)^{\frac{2}{1+d}} \Big(\frac{1}{\log N_1}\Big)^{\frac{2+4d}{1+d}} + \frac{1}{N_1} \Big],
  \]
  where $c>0$ is a constant only depending on $f_L,f_U,d$ and the infimum is taken over all measurable functions.
\end{proposition}






\section{Proofs of the main results}\label{sec:main-proof}

\paragraph*{Additional notation.} We use $\mX$ and $\mZ$ to represent $(X_1,X_2,\ldots,X_{N_0})$ and $(Z_1,Z_2,\ldots,Z_{N_1})$, respectively. Let $U(0,1)$ denote the uniform distribution on $[0,1]$. In the sequel, let $U \sim U(0,1)$ and $U_{(M)}$ be the $M$-th order statistic of $N_0$ independent random variables from $U(0,1)$, assumed to be mutually independent and both independent of $(\mX,\mZ)$.
It is well known that $U_{(M)}$ follows the beta distribution ${\rm Beta}(M, N_0+1-M)$. Let ${\rm Bin}(\cdot,\cdot)$ denote the binomial distribution. Let $L_1(\bR^d)$ denote the space of all functions $f:\bR^d\to\bR$ such that $\int|f(x)|\d x<\infty$. For any $x \in \bR^d$ and function $f:\bR^d \to \bR$, we say $x$ is a Lebesgue point of $f$ if
\[
  \lim_{\delta \to 0^+} \frac{1}{\lambda(B_{x,\delta})} \int_{B_{x,\delta}} \lvert f(x) - f(y) \rvert \d y = 0.
\]

\subsection{Proof of Lemma~\ref{lemma:moment,catch}}

\begin{proof}[Proof of Lemma~\ref{lemma:moment,catch}]

From the Lebesgue differentiation theorem, for any $f \in L_1(\bR^d)$, for $\lambda$-almost all $x$, $x$ is a Lebesgue point of $f$. Then for $\nu_0$-almost all $x$, we have $f_0(x)>0$ and $x$ is a Lebesgue point of $f_0$ and $f_1$ from the absolute continuity of $\nu_0$ and $\nu_1$. We then only need to consider those $x \in \bR^d$ such that $f_0(x)>0$ and $x$ is a Lebesgue point of $f_0$ and $f_1$.

We first introduce a lemma about the Lebesgue point.

\begin{lemma}\label{lemma:leb,p}
  Let $\nu$ be a probability measure on $\bR^d$ admitting a density $f$ with respect to the Lebesgue measure. Let $x \in \bR^d$ be a Lebesgue point of $f$. We then have, for any $\epsilon \in (0,1)$, there exists $\delta = \delta_x>0$ such that for any $z \in \bR^d$ such that $\lVert z - x \rVert \le \delta$, we have
  \begin{align*}
    \Big\lvert \frac{\nu(B_{x,\lVert z-x \rVert})}{\lambda(B_{x,\lVert z-x \rVert})} - f(x) \Big\rvert \le \epsilon,~~ \Big\lvert \frac{\nu(B_{z,\lVert z-x \rVert})}{\lambda(B_{z,\lVert z-x \rVert})} - f(x) \Big\rvert \le \epsilon.
  \end{align*}
\end{lemma}

{\bf Part I.} This part proves the first claim. We separate the proof of Part I into two cases based on the value of $f_1(x)$.

{\bf Case I.1.} $f_1(x)>0$.
Since $x$ is a Lebesgue point of $\nu_0$ and $\nu_1$, by Lemma \ref{lemma:leb,p}, for any $\epsilon \in (0,1)$, there exists some $\delta = \delta_x>0$ such that for any $z \in \bR^d$ with $\lVert z - x \rVert \le \delta$, we have
\begin{align*}
  & \Big\lvert \frac{\nu_0(B_{x,\lVert z-x \rVert})}{\lambda(B_{x,\lVert z-x \rVert})} - f_0(x) \Big\rvert \le \epsilon f_0(x),~~ \Big\lvert \frac{\nu_0(B_{z,\lVert z-x \rVert})}{\lambda(B_{z,\lVert z-x \rVert})} - f_0(x) \Big\rvert \le \epsilon f_0(x),\\
  & \Big\lvert \frac{\nu_1(B_{x,\lVert z-x \rVert})}{\lambda(B_{x,\lVert z-x \rVert})} - f_1(x) \Big\rvert \le \epsilon f_1(x),~~ \Big\lvert \frac{\nu_1(B_{z,\lVert z-x \rVert})}{\lambda(B_{z,\lVert z-x \rVert})} - f_1(x) \Big\rvert \le \epsilon f_1(x).
\end{align*}
Accordingly, if $\lVert z - x \rVert \le \delta$, we have
\[
  \frac{1-\epsilon}{1+\epsilon} \frac{f_0(x)}{f_1(x)} \le \frac{\nu_0(B_{z,\lVert x-z \rVert})}{\lambda(B_{z,\lVert x-z \rVert})} \frac{\lambda(B_{x,\lVert x-z \rVert})}{\nu_1(B_{x,\lVert x-z \rVert})} \le \frac{1+\epsilon}{1-\epsilon} \frac{f_0(x)}{f_1(x)}.
\]
Since $\lambda(B_{z,\lVert x-z \rVert}) = \lambda(B_{x,\lVert x-z \rVert})$, we then have
\begin{align}\label{eq:meancatch2}
  \frac{1-\epsilon}{1+\epsilon} \frac{f_0(x)}{f_1(x)} \le \frac{\nu_0(B_{z,\lVert x-z \rVert})}{\nu_1(B_{x,\lVert x-z \rVert})} \le \frac{1+\epsilon}{1-\epsilon} \frac{f_0(x)}{f_1(x)}.
\end{align}

On the other hand, for any $z \in \bR^d$ such that $\lVert z-x \rVert > \delta$,
\[
  \nu_0(B_{z,\lVert z-x \rVert}) \ge \nu_0(B_{z^*,\delta}) \ge (1-\epsilon)f_0(x) \lambda(B_{z^*,\delta}) = (1-\epsilon)f_0(x) \lambda(B_{0,\delta}),
\]
where $z^*$ is the intersection point of the surface of $B_{x,\delta}$ and the line connecting $z$ and $x$.

Let $\eta_N = 4\log(N_0/M)$. Since $M \log N_0 /N_0 \to 0$, we can take $N_0$ large enough so that
\[
  \eta_N \frac{M}{N_0} = 4\frac{M}{N_0} \log\Big(\frac{N_0}{M}\Big) < (1-\epsilon)f_0(x) \lambda(B_{0,\delta}).
\]
Then for any $z \in \bR^d$ such that $\nu_0(B_{z,\lVert z-x \rVert}) \le \eta_N M/N_0$, we have $\lVert z-x \rVert \le \delta$ since otherwise it would contradict the selection of $N_0$.

Let $Z$ be a copy from $\nu_1$ independent of the data. Then
\begin{align*}
  & \E\big[\nu_1\big(A_M(x)\big)\big] = \P \Big(Z \in A_M(x)\Big)\\
  = & \P\Big(\lVert x - Z \rVert \le \lVert \cX_{(M)}(Z) - Z \rVert\Big) = \P\Big( \nu_0(B_{Z,\lVert x - Z \rVert}) \le \nu_0(B_{Z,\lVert \cX_{(M)}(Z) - Z \rVert}) \Big).
  \yestag\label{eq:meancatch0}
\end{align*}
For any given $z \in \bR^d$, $[\nu_0(B_{z,\lVert X_i - z \rVert})]_{i=1}^{N_0}$ are i.i.d. from $U(0,1)$ since $[X_i]_{i=1}^{N_0}$ are i.i.d. from $\nu_0$ and we use the probability integral transform. Then $\nu_0(B_{Z,\lVert \cX_{(M)}(Z) - Z \rVert})$ has the same distribution as $U_{(M)}$ and is independent of $Z$.

{\bf Upper bound.} With a slight abuse of notation, we note $W = \nu_0(B_{Z,\lVert x-Z \rVert})$. We then have, from \eqref{eq:meancatch2} and \eqref{eq:meancatch0},
\begin{align*}
  & \E\big[\nu_1\big(A_M(x)\big)\big] = \P \Big(W \le \nu_0(B_{Z,\lVert \cX_{(M)}(Z) - Z \rVert})\Big) \\
  \le& \P \Big(W \le \nu_0(B_{Z,\lVert \cX_{(M)}(Z) - Z \rVert}), \nu_0(B_{Z,\lVert \cX_{(M)}(Z) - Z \rVert}) \le \eta_N \frac{M}{N_0}\Big) + \P\Big(\nu_0(B_{Z,\lVert \cX_{(M)}(Z) - Z \rVert}) > \eta_N \frac{M}{N_0}\Big)\\
  =&\P \Big(W \le \nu_0(B_{Z,\lVert \cX_{(M)}(Z) - Z \rVert}), \nu_0(B_{Z,\lVert \cX_{(M)}(Z) - Z \rVert}) \le \eta_N \frac{M}{N_0}, \lVert Z - x \rVert \le \delta\Big) + \P\Big(U_{(M)} > \eta_N \frac{M}{N_0}\Big)\\
  \le& \P \Big(\nu_0(B_{Z,\lVert x-Z \rVert}) \le \nu_0(B_{Z,\lVert \cX_{(M)}(Z) - Z \rVert}), \lVert Z - x \rVert \le \delta\Big) + \P\Big(U_{(M)} > \eta_N \frac{M}{N_0}\Big)\\
  \le& \P \Big(\frac{1-\epsilon}{1+\epsilon} \frac{f_0(x)}{f_1(x)}  \nu_1(B_{x,\lVert x-Z \rVert}) \le \nu_0(B_{Z,\lVert \cX_{(M)}(Z) - Z \rVert}), \lVert Z - x \rVert \le \delta\Big) + \P\Big(U_{(M)} > \eta_N \frac{M}{N_0}\Big)\\
  \le&\P \Big(\frac{1-\epsilon}{1+\epsilon} \frac{f_0(x)}{f_1(x)}  \nu_1(B_{x,\lVert x-Z \rVert}) \le \nu_0(B_{Z,\lVert \cX_{(M)}(Z) - Z \rVert})\Big) + \P\Big(U_{(M)} > \eta_N \frac{M}{N_0}\Big)\\
  =& \P \Big(\frac{1-\epsilon}{1+\epsilon} \frac{f_0(x)}{f_1(x)}   U \le U_{(M)}\Big) + \P\Big(U_{(M)} > \eta_N \frac{M}{N_0}\Big).
  \yestag\label{eq:meancatch1}
\end{align*}

For the second term in \eqref{eq:meancatch1}, notice that $\eta_N \to \infty$ as $N_0 \to \infty$. Then from the Chernoff bound and for $N_0$ sufficiently large, we have
\begin{align*}
  & \frac{N_0}{M} \P\Big(U_{(M)} > \eta_N \frac{M}{N_0}\Big) = \frac{N_0}{M} \P\Big({\rm Bin}\Big(N_0,\eta_N \frac{M}{N_0}\Big) \le M \Big)\\
  \le & \frac{N_0}{M}\exp\Big((1+\log \eta_N - \eta_N) M\Big) \le \frac{N_0}{M} \exp\Big(-\frac{1}{2} \eta_N M \Big) = \Big(\frac{N_0}{M}\Big)^{1-2M}.
\end{align*}
Since $M /N_0 \to 0$ and $M \ge 1$, we then obtain
\[
  \lim_{N_0 \to \infty} \frac{N_0}{M} \P\Big(U_{(M)} > \eta_N \frac{M}{N_0}\Big) = 0.
  \yestag\label{eq:meancatch3}
\]

For the first term in \eqref{eq:meancatch1}, we have
\begin{align*}
  & \frac{N_0}{M} \P \Big(\frac{1-\epsilon}{1+\epsilon} \frac{f_0(x)}{f_1(x)}   U \le U_{(M)}\Big) = \frac{N_0}{M} \int_0^1 \P \Big( U_{(M)} \ge \frac{1-\epsilon}{1+\epsilon} \frac{f_0(x)}{f_1(x)}  t \Big)\d t\\
  =& \frac{1+\epsilon}{1-\epsilon} \frac{f_1(x)}{f_0(x)} \int_0^{\frac{1-\epsilon}{1+\epsilon} \frac{f_0(x)}{f_1(x)} \frac{N_0}{M}} \P \Big( U_{(M)} \ge \frac{M}{N_0}t \Big) \d t \le  \frac{1+\epsilon}{1-\epsilon} \frac{f_1(x)}{f_0(x)} \int_0^{\infty} \P \Big( \frac{N_0}{M}U_{(M)} \ge t \Big) \d t\\
  =& \frac{1+\epsilon}{1-\epsilon} \frac{f_1(x)}{f_0(x)} \frac{N_0}{M} \E[U_{(M)}] = \frac{1+\epsilon}{1-\epsilon} \frac{f_1(x)}{f_0(x)} \frac{N_0}{N_0+1}.
  \yestag\label{eq:meancatch8}
\end{align*}
We then obtain
\[
  \limsup_{N_0 \to \infty}  \frac{N_0}{M} \P \Big(\frac{1-\epsilon}{1+\epsilon} \frac{f_0(x)}{f_1(x)}   U \le U_{(M)}\Big) \le \frac{1+\epsilon}{1-\epsilon} \frac{f_1(x)}{f_0(x)}.
  \yestag\label{eq:meancatch4}
\]

Plugging \eqref{eq:meancatch3} and \eqref{eq:meancatch4} to \eqref{eq:meancatch1} then yields
\begin{align}\label{eq:meancatch6}
  \limsup_{N_0 \to \infty} \frac{N_0}{M} \E\big[\nu_1\big(A_M(x)\big)\big] \le \frac{1+\epsilon}{1-\epsilon} \frac{f_1(x)}{f_0(x)}.
\end{align}

{\bf Lower bound.}
We have, from \eqref{eq:meancatch2} and \eqref{eq:meancatch0},
\begin{align*}
  & \E\big[\nu_1\big(A_M(x)\big)\big] = \P \Big(W \le \nu_0(B_{Z,\lVert \cX_{(M)}(Z) - Z \rVert})\Big) \\
  \ge& \P \Big(W \le \nu_0(B_{Z,\lVert \cX_{(M)}(Z) - Z \rVert}), \nu_0(B_{Z,\lVert \cX_{(M)}(Z) - Z \rVert}) \le \eta_N \frac{M}{N_0}\Big)\\
  =& \P \Big(W \le \nu_0(B_{Z,\lVert \cX_{(M)}(Z) - Z \rVert}), \nu_0(B_{Z,\lVert \cX_{(M)}(Z) - Z \rVert}) \le \eta_N \frac{M}{N_0}, \lVert Z - x \rVert \le \delta\Big)\\
  \ge& \P \Big(\frac{1+\epsilon}{1-\epsilon} \frac{f_0(x)}{f_1(x)}  \nu_1(B_{x,\lVert x-Z \rVert}) \le \nu_0(B_{Z,\lVert \cX_{(M)}(Z) - Z \rVert}), \nu_0(B_{Z,\lVert \cX_{(M)}(Z) - Z \rVert}) \le \eta_N \frac{M}{N_0}, \lVert Z - x \rVert \le \delta\Big)\\
  =& \P \Big(\frac{1+\epsilon}{1-\epsilon} \frac{f_0(x)}{f_1(x)}  \nu_1(B_{x,\lVert x-Z \rVert}) \le \nu_0(B_{Z,\lVert \cX_{(M)}(Z) - Z \rVert}), \nu_0(B_{Z,\lVert \cX_{(M)}(Z) - Z \rVert}) \le \eta_N \frac{M}{N_0}\Big)\\
  \ge& \P \Big(\frac{1+\epsilon}{1-\epsilon} \frac{f_0(x)}{f_1(x)}  \nu_1(B_{x,\lVert x-Z \rVert}) \le \nu_0(B_{Z,\lVert \cX_{(M)}(Z) - Z \rVert})\Big) - \P\Big(\nu_0(B_{Z,\lVert \cX_{(M)}(Z) - Z \rVert}) > \eta_N \frac{M}{N_0}\Big)\\
  =& \P \Big(\frac{1+\epsilon}{1-\epsilon} \frac{f_0(x)}{f_1(x)} U \le U_{(M)}\Big) - \P\Big(U_{(M)} > \eta_N \frac{M}{N_0}\Big).
  \yestag\label{eq:meancatch5}
\end{align*}

The second last equality is from the fact that for $\lVert Z-x \rVert > \delta$,
\begin{align*}
  &\frac{1+\epsilon}{1-\epsilon} \frac{f_0(x)}{f_1(x)}  \nu_1(B_{x,\lVert x-Z \rVert}) \ge \frac{1+\epsilon}{1-\epsilon} \frac{f_0(x)}{f_1(x)}  \nu_1(B_{x,\delta})\\
  \ge& \frac{1+\epsilon}{1-\epsilon} \frac{f_0(x)}{f_1(x)} f_1(x)(1-\epsilon) \lambda(B_{0,\delta}) = (1+\epsilon)f_0(x) \lambda(B_{0,\delta}) > \eta_N \frac{M}{N_0},
\end{align*}
and then that $\frac{1+\epsilon}{1-\epsilon} \frac{f_0(x)}{f_1(x)}  \nu_1(B_{x,\lVert x-Z \rVert}) \le \eta_N \frac{M}{N_0}$ implies $\lVert Z-x \rVert \le \delta$.

For the first term in \eqref{eq:meancatch5}, we have
\begin{align*}
  \frac{N_0}{M} \P \Big(\frac{1+\epsilon}{1-\epsilon} \frac{f_0(x)}{f_1(x)}   U \le U_{(M)}\Big) &= \frac{N_0}{M} \int_0^1 \P \Big( U_{(M)} \ge \frac{1+\epsilon}{1-\epsilon} \frac{f_0(x)}{f_1(x)}  t \Big)\d t\\
  &= \frac{1-\epsilon}{1+\epsilon} \frac{f_1(x)}{f_0(x)} \int_0^{\frac{1+\epsilon}{1-\epsilon} \frac{f_0(x)}{f_1(x)} \frac{N_0}{M}} \P \Big( U_{(M)} \ge \frac{M}{N_0}t \Big) \d t.
\end{align*}
If $\frac{1+\epsilon}{1-\epsilon} \frac{f_0(x)}{f_1(x)} \ge 1$, since $U_{(M)} \in [0,1]$, we have
\begin{align*}
  \frac{N_0}{M} \P \Big(\frac{1+\epsilon}{1-\epsilon} \frac{f_0(x)}{f_1(x)}   U \le U_{(M)}\Big) = \frac{1-\epsilon}{1+\epsilon} \frac{f_1(x)}{f_0(x)} \frac{N_0}{M} \E[U_{(M)}] = \frac{1-\epsilon}{1+\epsilon} \frac{f_1(x)}{f_0(x)} \frac{N_0}{N_0+1}.
\end{align*}
Then
\begin{align*}
  \lim_{N_0 \to \infty} \frac{N_0}{M} \P \Big(\frac{1+\epsilon}{1-\epsilon} \frac{f_0(x)}{f_1(x)}   U \le U_{(M)}\Big) = \frac{1-\epsilon}{1+\epsilon} \frac{f_1(x)}{f_0(x)}.
\end{align*}
If $\frac{1+\epsilon}{1-\epsilon} \frac{f_0(x)}{f_1(x)} < 1$, from the Chernoff bound,
\begin{align*}
  & \int_{\frac{1+\epsilon}{1-\epsilon} \frac{f_0(x)}{f_1(x)} \frac{N_0}{M}}^{\frac{N_0}{M}} \P \Big( U_{(M)} \ge \frac{M}{N_0}t \Big) \d t \\
  \le& \Big[1-\frac{1+\epsilon}{1-\epsilon} \frac{f_0(x)}{f_1(x)}\Big] \frac{N_0}{M} \P \Big( U_{(M)} \ge \frac{1+\epsilon}{1-\epsilon} \frac{f_0(x)}{f_1(x)} \Big)\\
  \le&  \Big[1-\frac{1+\epsilon}{1-\epsilon} \frac{f_0(x)}{f_1(x)}\Big] \frac{N_0}{M} \exp \Big[ M - \frac{1+\epsilon}{1-\epsilon} \frac{f_0(x)}{f_1(x)} N_0 - M\log M + M \log\Big(\frac{1+\epsilon}{1-\epsilon} \frac{f_0(x)}{f_1(x)} N_0\Big)\Big].
\end{align*}
Since $f_0(x)>0$ and $M\log N_0 /N_0 \to 0$, we obtain
\[
  \lim_{N_0 \to \infty} \int_{\frac{1+\epsilon}{1-\epsilon} \frac{f_0(x)}{f_1(x)} \frac{N_0}{M}}^{\frac{N_0}{M}} \P \Big( U_{(M)} \ge \frac{M}{N_0}t \Big) \d t = 0.
\]
Then
\begin{align*}
  \lim_{N_0 \to \infty} \frac{N_0}{M} \P \Big(\frac{1+\epsilon}{1-\epsilon} \frac{f_0(x)}{f_1(x)}   U \le U_{(M)}\Big) = \frac{1-\epsilon}{1+\epsilon} \frac{f_1(x)}{f_0(x)}.
\end{align*}

Plugging the above argument along with \eqref{eq:meancatch3} to \eqref{eq:meancatch5} yields
\begin{align}\label{eq:meancatch7}
  \liminf_{N_0 \to \infty} \frac{N_0}{M} \E\big[\nu_1\big(A_M(x)\big)\big] \ge \frac{1-\epsilon}{1+\epsilon} \frac{f_1(x)}{f_0(x)}.
\end{align}

Lastly, combining \eqref{eq:meancatch6} with \eqref{eq:meancatch7} and noticing that $\epsilon$ is arbitrary, we obtain
\begin{align}\label{eq:meancatch9}
  \lim_{N_0 \to \infty} \frac{N_0}{M} \E\big[\nu_1\big(A_M(x)\big)\big] = \frac{f_1(x)}{f_0(x)} = r(x).
\end{align}

{\bf Case I.2.} $f_1(x)=0$. Again, for any $\epsilon \in (0,1)$, by Lemma \ref{lemma:leb,p}, there exists some $\delta = \delta_x>0$ such that for any $z \in \bR^d$ with $\lVert z - x \rVert \le \delta$, we have
\begin{align*}
  \Big\lvert \frac{\nu_0(B_{z,\lVert z-x \rVert})}{\lambda(B_{z,\lVert z-x \rVert})} - f_0(x) \Big\rvert  \le \epsilon f_0(x),~~ \Big\lvert \frac{\nu_1(B_{x,\lVert z-x \rVert})}{\lambda(B_{x,\lVert z-x \rVert})} \Big\rvert \le \epsilon.
\end{align*}

Recall that $W = \nu_0(B_{Z,\lVert x-Z \rVert})$. Then if $\lVert Z - x \rVert \le \delta$, we have
\begin{align*}
  Z \ge (1-\epsilon) f_0(x) \lambda(B_{Z,\lVert x-Z \rVert}) = (1-\epsilon) f_0(x) \lambda(B_{x,\lVert x-Z \rVert}) \ge \frac{1-\epsilon}{\epsilon} f_0(x) \nu_1(B_{x,\lVert x-Z \rVert}).
\end{align*}

Proceeding in the same way as \eqref{eq:meancatch1}, we obtain
\begin{align*}
  & \E\big[\nu_1\big(A_M(x)\big)\big] \\
  \le & \P \Big(W \le \nu_0(B_{Z,\lVert \cX_{(M)}(Z) - Z \rVert}), \nu_0(B_{Z,\lVert \cX_{(M)}(Z) - Z \rVert}) \le \eta_N \frac{M}{N_0}, \lVert Z - x \rVert \le \delta\Big) + \P\Big(U_{(M)} > \eta_N \frac{M}{N_0}\Big)\\
  \le& \P \Big(\frac{1-\epsilon}{\epsilon} f_0(x) U \le U_{(M)}\Big) + \P\Big(U_{(M)} > \eta_N \frac{M}{N_0}\Big).
\end{align*}

For the first term above,
\begin{align*}
  & \frac{N_0}{M} \P \Big(\frac{1-\epsilon}{\epsilon} f_0(x) U \le U_{(M)}\Big) = \frac{N_0}{M} \int_0^1 \P \Big( U_{(M)} \ge \frac{1-\epsilon}{\epsilon} f_0(x) t \Big)\d t\\
  =& \frac{\epsilon}{1-\epsilon} \frac{1}{f_0(x)} \int_0^{\frac{1-\epsilon}{\epsilon} f_0(x) \frac{N_0}{M}} \P \Big( U_{(M)} \ge \frac{M}{N_0}t \Big) \d t \le  \frac{\epsilon}{1-\epsilon} \frac{1}{f_0(x)} \int_0^{\infty} \P \Big( \frac{N_0}{M}U_{(M)} \ge t \Big) \d t\\
  =& \frac{\epsilon}{1-\epsilon} \frac{1}{f_0(x)} \frac{N_0}{M} \E[U_{(M)}] = \frac{\epsilon}{1-\epsilon} \frac{1}{f_0(x)} \frac{N_0}{N_0+1}.
\end{align*}

By \eqref{eq:meancatch3} and noticing $\epsilon$ is arbitrary, one has
\begin{align}\label{eq:meancatch10}
  \lim_{N_0 \to \infty} \frac{N_0}{M} \E\big[\nu_1\big(A_M(x)\big)\big] = 0 = r(x).
\end{align}

Combining \eqref{eq:meancatch9} and \eqref{eq:meancatch10} completes the proof of the first claim.

\vspace{0.5cm}

{\bf Part II.} This part proves the second claim. We also separate the proof of Part II into two cases based on the value of $f_1(x)$.

{\bf Case II.1.} $f_1(x)>0$. Again, for any $\epsilon \in (0,1)$, we take $\delta$ in the same way as in Case I.1. Let $\eta_N = \eta_{N,p} = 4p\log(N_0/M)$. We also take $N_0$ sufficiently large so that
\[
  \eta_N \frac{M}{N_0} = 4p\frac{M}{N_0} \log\Big(\frac{N_0}{M}\Big) < (1-\epsilon)f_0(x) \lambda(B_{0,\delta}).
\]

Let $\tZ_1,\ldots,\tZ_p$ be $p$ independent copies that are drawn from $\nu_1$ independent of the data. Then
\begin{align*}
  & \E\big[\nu_1^p(A_M(x))\big] = \P \Big(\tZ_1,\ldots,\tZ_p \in A_M(x)\Big)\\
  =& \P\Big(\lVert x - \tZ_1 \rVert \le \lVert \cX_{(M)}(\tZ_1) - \tZ_1 \rVert, \ldots, \lVert x - \tZ_p \rVert \le \lVert \cX_{(M)}(\tZ_p) - \tZ_p \rVert\Big)\\
  =& \P\Big(\nu_0(B_{\tZ_1,\lVert x - \tZ_1 \rVert}) \le \nu_0(B_{\tZ_1,\lVert \cX_{(M)}(\tZ_1) - \tZ_1 \rVert}), \ldots, \nu_0(B_{\tZ_p,\lVert x - \tZ_p \rVert}) \le \nu_0(B_{\tZ_p,\lVert \cX_{(M)}(\tZ_p) - \tZ_p \rVert})\Big).
\end{align*}

Let $W_k = \nu_0(B_{\tZ_k,\lVert x - \tZ_k \rVert})$ and $V_k = \nu_0(B_{\tZ_k,\lVert \cX_{(M)}(\tZ_k) - \tZ_k \rVert})$ for any $k \in \zahl{p}$. Then $[W_k]_{k=1}^p$ are i.i.d. since $[\tZ_k]_{k=1}^p$ are i.i.d. For any $k \in \zahl{p}$ and $\tZ_k \in \bR^d$ given, $V_k \given \tZ_k $ has the same distribution as $U_{(M)}$. Then for any $k \in \zahl{p}$, $V_k$ has the same distribution as $U_{(M)}$, and $V_k$ is independent of $\tZ_k$.

Let $W_{\max} = \max_{k \in \zahl{p}} W_k$ and $V_{\max} = \max_{k \in \zahl{p}} V_k$. Then
\begin{align*}
  \E\big[\nu_1^p(A_M(x))\big] &= \P\Big(W_1 \le V_1, \ldots, W_p \le V_p\Big) \le \P\Big(W_{\max} \le V_{\max}\Big)\\
  &\le \P\Big(W_{\max} \le V_{\max}, V_{\max} \le \eta_N \frac{M}{N_0} \Big) + \P\Big( V_{\max} > \eta_N \frac{M}{N_0} \Big)
  \yestag\label{eq:momentcatch1}
\end{align*}

For the second term in \eqref{eq:momentcatch1},
\[
  \P\Big( V_{\max} > \eta_N \frac{M}{N_0}\Big) \le \sum_{k=1}^p \P\Big( V_k > \eta_N \frac{M}{N_0}\Big) = p \P\Big( U_{(M)} > \eta_N \frac{M}{N_0}\Big).
\]
Proceeding as \eqref{eq:meancatch3},
\begin{align*}
  \Big(\frac{N_0}{M}\Big)^p \P\Big(U_{(M)} > \eta_N \frac{M}{N_0}\Big) \le \Big(\frac{N_0}{M}\Big)^p \exp\Big(-\frac{1}{2} \eta_N M \Big) = \Big(\frac{N_0}{M}\Big)^{p(1-2M)}.
\end{align*}
We then obtain
\begin{align}\label{eq:momentcatch2}
  \lim_{N_0 \to \infty} \Big(\frac{N_0}{M}\Big)^p \P\Big( V_{\max} > \eta_N \frac{M}{N_0}\Big) = 0.
\end{align}

For the first term in \eqref{eq:momentcatch1}, notice that $[\nu_1(B_{x,\lVert \tZ_k - x \rVert})]_{k=1}^p$ are i.i.d. from $U(0,1)$ since $[\tZ_k]_{k=1}^p$ are i.i.d. We then have
\begin{align*}
  & \Big(\frac{N_0}{M}\Big)^p \P\Big(W_{\max} \le V_{\max}, V_{\max} \le \eta_N \frac{M}{N_0} \Big) \\
  =& \Big(\frac{N_0}{M}\Big)^p \P\Big(W_{\max} \le V_{\max}, V_{\max} \le \eta_N \frac{M}{N_0} , \max_{k \in \zahl{p}} \lVert \tZ_k - x \rVert \le \delta \Big)\\
  \le & \Big(\frac{N_0}{M}\Big)^p \P\Big(\frac{1-\epsilon}{1+\epsilon} \frac{f_0(x)}{f_1(x)} \max_{k \in \zahl{p}} \nu_1(B_{x,\lVert \tZ_k - x \rVert}) \le V_{\max}, V_{\max} \le \eta_N \frac{M}{N_0} , \max_{k \in \zahl{p}} \lVert \tZ_k - x \rVert \le \delta \Big)\\
  \le & \Big(\frac{N_0}{M}\Big)^p \P\Big(\frac{1-\epsilon}{1+\epsilon} \frac{f_0(x)}{f_1(x)} \max_{k \in \zahl{p}} \nu_1(B_{x,\lVert \tZ_k - x \rVert}) \le V_{\max} \Big)\\
  = & \Big(\frac{N_0}{M}\Big)^p \int_0^1 pt^{p-1} \P\Big( V_{\max} \ge \frac{1-\epsilon}{1+\epsilon} \frac{f_0(x)}{f_1(x)} t \Biggiven \max_{k \in \zahl{p}} \nu_1(B_{x,\lVert \tZ_k - x \rVert}) = t \Big) \d t\\
  = & p \Big(\frac{1+\epsilon}{1-\epsilon} \frac{f_1(x)}{f_0(x)}\Big)^p \int_0^{\frac{1-\epsilon}{1+\epsilon} \frac{f_0(x)}{f_1(x)} \frac{N_0}{M}} t^{p-1} \P\Big( V_{\max} \ge \frac{M}{N_0} t \Biggiven \max_{k \in \zahl{p}} \nu_1(B_{x,\lVert \tZ_k - x \rVert}) = \frac{1+\epsilon}{1-\epsilon} \frac{f_1(x)}{f_0(x)} \frac{M}{N_0}t \Big) \d t\\
  = & p \Big(\frac{1+\epsilon}{1-\epsilon} \frac{f_1(x)}{f_0(x)}\Big)^p \Big[ \int_0^1 t^{p-1} \P\Big( V_{\max} \ge \frac{M}{N_0} t \Biggiven \max_{k \in \zahl{p}} \nu_1(B_{x,\lVert \tZ_k - x \rVert}) = \frac{1+\epsilon}{1-\epsilon} \frac{f_1(x)}{f_0(x)} \frac{M}{N_0}t \Big) \d t \\
  & + \int_1^{\frac{1-\epsilon}{1+\epsilon} \frac{f_0(x)}{f_1(x)} \frac{N_0}{M}} t^{p-1} \P\Big( V_{\max} \ge \frac{M}{N_0} t \Biggiven \max_{k \in \zahl{p}} \nu_1(B_{x,\lVert \tZ_k - x \rVert}) = \frac{1+\epsilon}{1-\epsilon} \frac{f_1(x)}{f_0(x)} \frac{M}{N_0}t \Big) \d t\Big].
\end{align*}

For the first term,
\[
  \int_0^1 t^{p-1} \P\Big( V_{\max} \ge \frac{M}{N_0} t \Biggiven \max_{k \in \zahl{p}} \nu_1(B_{x,\lVert \tZ_k - x \rVert}) = \frac{1+\epsilon}{1-\epsilon} \frac{f_1(x)}{f_0(x)} \frac{M}{N_0}t \Big) \d t \le \int_0^1 t^{p-1} \d t = \frac{1}{p}. 
\]

For the second term, using the Chernoff bound, conditional on $\tilde{\mZ} = (\tZ_1, \ldots, \tZ_p)$,
\begin{align*}
  & \int_1^{\frac{1-\epsilon}{1+\epsilon} \frac{f_0(x)}{f_1(x)} \frac{N_0}{M}} t^{p-1} \P\Big( V_{\max} \ge \frac{M}{N_0} t \Biggiven \tilde{\mZ} \Big) \d t\\
  \le & \int_1^\infty t^{p-1} \P\Big( V_{\max} \ge \frac{M}{N_0} t \Biggiven \tilde{\mZ} \Big) \d t = \int_0^\infty (1+t)^{p-1} \P\Big( V_{\max} \ge \frac{M}{N_0} (1+t) \Biggiven \tilde{\mZ}  \Big) \d t\\
  \le & \int_0^\infty (1+t)^{p-1} \Big[ \sum_{k=1}^p  \P\Big( V_k \ge \frac{M}{N_0} (1+t) \Biggiven \tilde{\mZ} \Big) \Big] \d t = p \int_0^\infty (1+t)^{p-1} \P\Big( U_{(M)} \ge \frac{M}{N_0} (1+t) \Big) \d t\\
  \le & p \int_0^\infty (1+t)^{p-1} (1+t)^M \exp(-tM) \d t = p e^M \int_1^\infty t^{M+p-1} \exp(-tM) \d t\\
  \le & p e^M \int_0^\infty t^{M+p-1} \exp(-tM) \d t = \frac{pe^M}{M^{M+p}} \Gamma(M+p) = \frac{pe^M}{M^{M+p}} (M+1)^{p-1} \Gamma(M+1) (1+o(1)) \\
  = & \frac{pe^M}{M^{M+p}} (M+1)^{p-1} \sqrt{2\pi M} \frac{M^M}{e^M} (1+o(1)) = \sqrt{2\pi} p M^{-1/2} \Big(1+\frac{1}{M}\Big)^{p-1} (1+o(1)),
\end{align*}
where the last three steps are from Stirling's approximation using $M \to \infty$.

Then
\[
  \lim_{N_0 \to \infty}p \Big(\frac{1+\epsilon}{1-\epsilon} \frac{f_1(x)}{f_0(x)}\Big)^p \int_1^{\frac{1-\epsilon}{1+\epsilon} \frac{f_0(x)}{f_1(x)} \frac{N_0}{M}} t^{p-1} \P\Big( V_{\max} \ge \frac{M}{N_0} t \Biggiven \max_{k \in \zahl{p}} \nu_1(B_{x,\lVert \tZ_k - x \rVert}) = \frac{1+\epsilon}{1-\epsilon} \frac{f_1(x)}{f_0(x)} \frac{M}{N_0}t \Big) \d t = 0,
\]
and then we obtain
\begin{align}\label{eq:momentcatch3}
  \limsup_{N_0 \to \infty} \Big(\frac{N_0}{M}\Big)^p \P\Big(W_{\max} \le V_{\max}, V_{\max} \le \eta_N \frac{M}{N_0} \Big) \le \Big(\frac{1+\epsilon}{1-\epsilon} \frac{f_1(x)}{f_0(x)}\Big)^p.
\end{align}

Plugging \eqref{eq:momentcatch2} and \eqref{eq:momentcatch3} into \eqref{eq:momentcatch1} yields
\begin{align}\label{eq:momentcatch4}
  \limsup_{N_0 \to \infty} \Big(\frac{N_0}{M}\Big)^p \E\big[\nu_1^p(A_M(x))\big] \le \Big(\frac{1+\epsilon}{1-\epsilon} \frac{f_1(x)}{f_0(x)}\Big)^p = \Big(\frac{1+\epsilon}{1-\epsilon} r(x) \Big)^p.
\end{align}

Lastly, using Hölder's inequality, 
\[
  \Big(\frac{N_0}{M}\Big)^p \E\big[\nu_1^p(A_M(x))\big] \ge \Big[\frac{N_0}{M} \E\big[\nu_1\big(A_M(x)\big)\big] \Big]^p.
\]
Employing the first claim, we have
\begin{align}\label{eq:momentcatch5}
  \liminf_{N_0 \to \infty} \Big(\frac{N_0}{M}\Big)^p \E\big[\nu_1^p(A_M(x))\big] \ge \big[r(x)\big]^p.
\end{align}
Combining \eqref{eq:momentcatch4} with \eqref{eq:momentcatch5} and noting that $\epsilon$ is arbitrary, we obtain
\begin{align}\label{eq:momentcatch6}
  \lim_{N_0\to\infty} \Big(\frac{N_0}{M}\Big)^p \E\big[\nu_1^p\big(A_M(x)\big)\big] = \big[r(x)\big]^p.
\end{align}

{\bf Case II.2.} $f_1(x)=0$. For any $\epsilon \in (0,1)$, we take $\delta$ in the same way as in the proof of Case I.2 and take $\eta_N$ as in the proof of Case II.1.

By \eqref{eq:momentcatch1},
\begin{align*}
  \Big(\frac{N_0}{M}\Big)^p \E\big[\nu_1^p(A_M(x))\big] \le \Big(\frac{N_0}{M}\Big)^p \P\Big(W_{\max} \le V_{\max}, V_{\max} \le \eta_N \frac{M}{N_0} \Big) + \Big(\frac{N_0}{M}\Big)^p \P\Big( V_{\max} > \eta_N \frac{M}{N_0} \Big).
\end{align*}

For the first term,
\begin{align*}
  & \Big(\frac{N_0}{M}\Big)^p \P\Big(W_{\max} \le V_{\max}, V_{\max} \le \eta_N \frac{M}{N_0} \Big) \\
  =& \Big(\frac{N_0}{M}\Big)^p \P\Big(W_{\max} \le V_{\max}, V_{\max} \le \eta_N \frac{M}{N_0} , \max_{k \in \zahl{p}} \lVert Z_k - x \rVert \le \delta \Big)\\
  \le & \Big(\frac{N_0}{M}\Big)^p \P\Big(\frac{1-\epsilon}{\epsilon} f_0(x)  \max_{k \in \zahl{p}} \nu_1(B_{x,\lVert \tZ_k - x \rVert}) \le V_{\max}, V_{\max} \le \eta_N \frac{M}{N_0} , \max_{k \in \zahl{p}} \lVert Z_k - x \rVert \le \delta \Big)\\
  \le & \Big(\frac{N_0}{M}\Big)^p \P\Big(\frac{1-\epsilon}{\epsilon} f_0(x)  \max_{k \in \zahl{p}} \nu_1(B_{x,\lVert \tZ_k - x \rVert}) \le V_{\max} \Big)\\
  = & \Big(\frac{N_0}{M}\Big)^p \int_0^1 pt^{p-1} \P\Big( V_{\max} \ge \frac{1-\epsilon}{\epsilon} f_0(x) t \Biggiven \max_{k \in \zahl{p}} \nu_1(B_{x,\lVert \tZ_k - x \rVert}) = t \Big) \d t\\
  = & p \Big(\frac{\epsilon}{1-\epsilon} \frac{1}{f_0(x)}\Big)^p \int_0^{\frac{1-\epsilon}{\epsilon} f_0(x) \frac{N_0}{M}} t^{p-1} \P\Big( V_{\max} \ge \frac{M}{N_0} t \Biggiven \max_{k \in \zahl{p}} \nu_1(B_{x,\lVert \tZ_k - x \rVert}) = t \Big) \d t\\
  = & p \Big(\frac{\epsilon}{1-\epsilon} \frac{1}{f_0(x)}\Big)^p \Big[ \int_0^1 t^{p-1} \P\Big( V_{\max} \ge \frac{M}{N_0} t \Biggiven \max_{k \in \zahl{p}} \nu_1(B_{x,\lVert \tZ_k - x \rVert}) = t \Big) \d t \\
  & + \int_1^{\frac{1-\epsilon}{\epsilon} f_0(x) \frac{N_0}{M}} t^{p-1} \P\Big( V_{\max} \ge \frac{M}{N_0} t \Biggiven \max_{k \in \zahl{p}} \nu_1(B_{x,\lVert \tZ_k - x \rVert}) = t \Big) \d t\Big].
\end{align*}
Then proceeding in the same way as \eqref{eq:momentcatch3}, we have
\begin{align*}
  \limsup_{N_0 \to \infty} \Big(\frac{N_0}{M}\Big)^p \P\Big(W_{\max} \le V_{\max}, V_{\max} \le \eta_N \frac{M}{N_0} \Big) \le \Big(\frac{\epsilon}{1-\epsilon} \frac{1}{f_0(x)}\Big)^p.
\end{align*}

Lastly, using \eqref{eq:momentcatch2} and noting again that $\epsilon$ is arbitrary, one obtains
\begin{align}\label{eq:momentcatch7}
  \lim_{N_0\to\infty} \Big(\frac{N_0}{M}\Big)^p \E\big[\nu_1^p\big(A_M(x)\big)\big] = 0 = \big[r(x)\big]^p.
\end{align}

Combining \eqref{eq:momentcatch6} and \eqref{eq:momentcatch7} then completes the proof of the second claim.
\end{proof}

\subsection{Proof of Theorem~\ref{thm:cons,lp}}

\begin{proof}[Proof of Theorem~\ref{thm:cons,lp}\ref{thm:cons,lp1}]

By \eqref{eq:rhat} and that $[Z_j]_{j=1}^{N_1}$ are i.i.d, 
\begin{align*}
  \E\big[\hat{r}_M(x)\big] = \E\Big[\frac{N_0}{N_1} \frac{K_M(x)}{M}\Big] = \frac{N_0}{N_1 M} \E\Big[ \sum_{j=1}^{N_1} \ind\big(Z_j \in A_M(x)\big) \Big] = \frac{N_0}{M} \E\big[\nu_1\big(A_M(x)\big)\big].
\end{align*}

Employing Lemma~\ref{lemma:moment,catch} then completes the proof.
\end{proof}

\begin{proof}[Proof of Theorem~\ref{thm:cons,lp}\ref{thm:cons,lp2}]

By Hölder's inequality, it suffices to consider the case when $p$ is even. Noticing that $x^p$ is convex for $p >1$ and $x>0$, one has
\begin{align}\label{eq:conslp1}
  \E \big[ \lvert \hat{r}_M(x) - r(x) \rvert^p\big] \le 2^{p-1} \Big( \E \Big[ \Big\lvert \hat{r}_M(x) - \E[\hat{r}_M(x) \given \mX] \Big\rvert^p\Big] + \E \Big[ \Big\lvert \E[\hat{r}_M(x) \given \mX] - r(x) \Big\rvert^p\Big] \Big).
\end{align}

For the second term in \eqref{eq:conslp1}, Lemma~\ref{lemma:moment,catch} implies
\begin{align}\label{eq:conslp2}
  \lim_{N_0 \to \infty} \E \Big[ \Big\lvert \E[\hat{r}_M(x) \given \mX] - r(x) \Big\rvert^p\Big] &= \lim_{N_0 \to \infty} \E \Big[ \Big\lvert \frac{N_0}{M} \nu_1\big(A_M(x)\big) - r(x) \Big\rvert^p\Big] = 0
\end{align}
by expanding the product term.

For the first term in \eqref{eq:conslp1}, noticing that $[Z_j]_{j=1}^{N_1}$ are i.i.d, we have $K_M(x) \given \mX \sim {\rm Bin}(N_1,\nu_1(A_M(x)))$. Using Lemma~\ref{lemma:moment,catch} and $MN_1/N_0 \to \infty$, for any positive integers $p$ and $q$, we have
\begin{align*}
  \lim_{N_0 \to \infty} \Big( \frac{N_0}{N_1M}\Big)^p \E[N_1^p \nu_1^p(A_M(x))] = \big[r(x)\big]^p,\\
  \lim_{N_0 \to \infty} \Big( \frac{N_0}{N_1M}\Big)^p \Big( \frac{N_0}{M}\Big)^q \E[N_1^p \nu_1^{p+q}(A_M(x))] = \big[r(x)\big]^p,
\end{align*}
and then $\E[N_1^p \nu_1^p(A_M(x))]$ is the dominated term among $[\E[N_1^k \nu_1^{k+q}(A_M(x))]]_{k \le p, q \ge 0}$.

To complete the proof, for any positive integer $c$ and $Z \sim {\rm Bin}(n,p')$, let $\mu_c := \E[(Z-\E[Z])^c]$ be the $c$-th central moment. According to \cite{romanovsky1923note}, we have
\[
  \mu_{c+1} = p'(1-p')\Big(nc\mu_{c-1}+\frac{\d \mu_c}{\d p'}\Big).
\]
Then for even $p$, we obtain
\[
  \E \Big[ \Big(K_M(x) - N_1\nu_1(A_M(x)) \Big)^p\Big] \lesssim \E[N_1 \nu_1(A_M(x))]^{p/2} \lesssim \Big( \frac{N_1M}{N_0}\Big)^{p/2}.
\]
The first term in \eqref{eq:conslp1} then satisfies
\begin{align*}
  \E \Big[ \Big\lvert \hat{r}_M(x) - \E[\hat{r}_M(x) \given \mX] \Big\rvert^p\Big] &= \Big( \frac{N_0}{N_1M}\Big)^p \E \Big[ \Big(K_M(x) - N_1\nu_1(A_M(x)) \Big)^p\Big] \lesssim \Big( \frac{N_0}{N_1M}\Big)^{p/2}.
\end{align*}
Since $MN_1/N_0 \to \infty$, we obtain
\begin{align}\label{eq:conslp3}
  \lim_{N_0 \to \infty} \E \Big[ \Big\lvert \hat{r}_M(x) - \E[\hat{r}_M(x) \given \mX] \Big\rvert^p\Big] = 0.
\end{align}

Plugging \eqref{eq:conslp2} and \eqref{eq:conslp3} into \eqref{eq:conslp1} then completes the proof.
\end{proof}

\subsection{Proof of Theorem~\ref{thm:risk,lp}}

\begin{proof}[Proof of Theorem~\ref{thm:risk,lp}]

We first cite the Hardy–Littlewood maximal inequality. 
\begin{lemma}[Hardy–Littlewood maximal inequality \citep{stein2016singular}]\label{lemma:HL}
  For any locally integrable function $f:\bR^d \to \bR$, define
  \[
    {\sf M}f(x):= \sup_{\delta>0} \frac{1}{\lambda(B_{x,\delta})} \int_{B_{x,\delta}} \lvert f(z) \rvert \d z.
  \]
  Then for $d \ge 1$, there exists a constant $C_d>0$ only depending on $d$ such that for all $t>0$ and $f \in L_1(\bR^d)$, we have
  \[
    \lambda(\{x:{\sf M}f(x) > t\}) < \frac{C_d}{t} \lVert f \rVert_{L_1},
  \]
  where $\lVert \cdot \rVert_{L_1}$ stands for the function $L_1$ norm.
\end{lemma}

Let $\epsilon>0$ be given. We assume $\epsilon \le f_L$. From Assumption~\ref{asp:risk}, $S_0$ and $S_1$ are bounded, then $\nu_0$ and $\nu_1$ are compactly supported. Since $f_0, f_1 \in L_1$ and the class of continuous functions are dense in the class of compactly supported $L_1$ functions from simple use of Lusin's theorem, we can find $g_0,g_1$ such that $g_0,g_1$ are continuous and $\lVert f_0 - g_0 \rVert_{L_1} \le \epsilon^3$ and $\lVert f_1 - g_1 \rVert_{L_1} \le \epsilon^3$.

Since $g_0, g_1$ are continuous with compact supports, they are uniformly continuous, that is, there exists $\delta > 0$ such that for any $x,z \in \bR^d$ and $\lVert z-x \rVert \le \delta$, we have
\[
  \lvert g_0(x) - g_0(z) \rvert \le \frac{\epsilon^2}{3}~~~{\rm and}~~~ \lvert g_1(x) - g_1(z) \rvert \le \frac{\epsilon^2}{3}.
\]

For any $x \in \bR^d$, we have
\begin{align*}
  & \frac{1}{\lambda(B_{x,\delta})} \int_{B_{x,\delta}} \lvert f_0(x) - f_0(z) \rvert \d z \\
  \le & \frac{1}{\lambda(B_{x,\delta})} \int_{B_{x,\delta}} \lvert f_0(x) - g_0(x) \rvert \d z + \frac{1}{\lambda(B_{x,\delta})} \int_{B_{x,\delta}} \lvert g_0(x) - g_0(z) \rvert \d z + \frac{1}{\lambda(B_{x,\delta})} \int_{B_{x,\delta}} \lvert f_0(z) - g_0(z) \rvert \d z\\
  = & \lvert f_0(x) - g_0(x) \rvert + \frac{1}{\lambda(B_{x,\delta})} \int_{B_{x,\delta}} \lvert g_0(x) - g_0(z) \rvert \d z + \frac{1}{\lambda(B_{x,\delta})} \int_{B_{x,\delta}} \lvert f_0(z) - g_0(z) \rvert \d z.
  \yestag\label{eq:risklp1}
\end{align*}

For the first term in \eqref{eq:risklp1}, using Markov's inequality, one has
\begin{align}\label{eq:risklp2}
  \lambda\Big(\Big\{x: \Big\lvert f_0(x) - g_0(x) \Big\rvert > \frac{\epsilon^2}{3} \Big\}\Big) \le \frac{3}{\epsilon^2} \lVert f_0 - g_0 \rVert_{L_1} \le 3\epsilon.
\end{align}

For the second term in \eqref{eq:risklp1}, by the selection of $\delta$,
\begin{align}\label{eq:risklp3}
  \frac{1}{\lambda(B_{x,\delta})} \int_{B_{x,\delta}} \lvert g_0(x) - g_0(z) \rvert \d z \le \max_{z \in B_{x,\delta}} \lvert g_0(x) - g_0(z) \rvert \le \frac{\epsilon^2}{3}.
\end{align}

For the third term,
\[
  \frac{1}{\lambda(B_{x,\delta})} \int_{B_{x,\delta}} \lvert f_0(z) - g_0(z) \rvert \d z \le \sup_{\delta>0} \frac{1}{\lambda(B_{x,\delta})} \int_{B_{x,\delta}} \lvert f_0(z) - g_0(z) \rvert \d z = {\sf M}(f_0-g_0)(x).
\]
Lemma~\ref{lemma:HL} then yields
\begin{align}\label{eq:risklp4}
  \lambda\Big(\Big\{x:{\sf M}(f_0-g_0)(x) > \frac{\epsilon^2}{3} \Big\}\Big) < \frac{3C_d}{\epsilon^2} \lVert f_0 - g_0 \rVert_{L_1} \le 3C_d\epsilon.
\end{align}
We can establish similar results for $f_1,g_1$.

Let
\begin{align*}
  A_1 := & \Big\{x: \Big\lvert f_0(x) - g_0(x) \Big\rvert > \frac{\epsilon^2}{3} \Big\} \bigcup \Big\{x: \Big\lvert f_1(x) - g_1(x) \Big\rvert > \frac{\epsilon^2}{3} \Big\} \bigcup \\
  & \Big\{x:{\sf M}(f_0-g_0)(x) > \frac{\epsilon^2}{3} \Big\} \bigcup \Big\{x:{\sf M}(f_1-g_1)(x) > \frac{\epsilon^2}{3} \Big\}.
\end{align*}
Plugging \eqref{eq:risklp2}, \eqref{eq:risklp3}, \eqref{eq:risklp4} into \eqref{eq:risklp1}, for any $x \notin A_1$ and $\lVert z - x \rVert \le \delta$, we have
\begin{align*}
  \frac{1}{\lambda(B_{x,\delta})} \int_{B_{x,\delta}} \lvert f_0(x) - f_0(z) \rvert \d z \le \epsilon^2, ~~ \frac{1}{\lambda(B_{x,\delta})} \int_{B_{x,\delta}} \lvert f_1(x) - f_1(z) \rvert \d z \le \epsilon^2,
\end{align*}
and
\[
  \lambda(A_1) \le 6(C_d+1)\epsilon.
\]

Let $A_2 := \Big\{x: f_1(x) \le \epsilon\Big\}$. We then seperate the proof into three cases. In the following, it suffices to consider $f_0(x)>0$ due to the definition of $L_p$ risk in \eqref{eq:risk}.

{\bf Case I.} $x \notin A_1 \cup A_2$.

According to $\epsilon \le f_L$ and by the definition of $A_2$, for any $x \notin A_1 \cup A_2$ and $\lVert z-x \rVert \le \delta$,
\begin{align*}
  & \frac{1}{\lambda(B_{x,\delta})} \int_{B_{x,\delta}} \lvert f_0(x) - f_0(y) \rvert \d z \le \epsilon^2 \le \epsilon f_L \le \epsilon f_0(x),\\
  & \frac{1}{\lambda(B_{x,\delta})} \int_{B_{x,\delta}} \lvert f_1(x) - f_1(y) \rvert \d z \le \epsilon^2 \le \epsilon f_1(x).
\end{align*}
We then obtain
\begin{align*}
  & \Big\lvert \frac{\nu_0(B_{x,\lVert z-x \rVert})}{\lambda(B_{x,\lVert z-x \rVert})} - f_0(x) \Big\rvert \le \epsilon f_0(x), ~~ \Big\lvert \frac{\nu_0(B_{z,\lVert z-x \rVert})}{\lambda(B_{z,\lVert z-x \rVert})} - f_0(x) \Big\rvert \le \epsilon f_0(x),\\
  & \Big\lvert \frac{\nu_1(B_{x,\lVert z-x \rVert})}{\lambda(B_{x,\lVert z-x \rVert})} - f_1(x) \Big\rvert \le \epsilon f_1(x), ~~\Big\lvert \frac{\nu_1(B_{z,\lVert z-x \rVert})}{\lambda(B_{z,\lVert z-x \rVert})} - f_1(x) \Big\rvert \le \epsilon f_1(x).
\end{align*}

Let $\eta_N = \eta_{N,p} = 4p\log(N_0/M)$. We also take $N_0$ large enough so that
\[
  \eta_N \frac{M}{N_0} = 4p\frac{M}{N_0} \log\Big(\frac{N_0}{M}\Big) < (1-\epsilon)f_L \lambda(B_{0,\delta}).
\]
Then for any $x \in \bR^d$ such that $f_0(x)>0$, we have
\[
  \eta_N \frac{M}{N_0} < (1-\epsilon)f_0(x) \lambda(B_{0,\delta}).
\]

Proceeding as in the proof of Case II.1 of Lemma~\ref{lemma:moment,catch} and also Theorem~\ref{thm:cons,lp} by using Fubini's theorem, since $\epsilon$ is arbitrary, we obtain
\begin{align}\label{eq:risklp5}
  \lim_{N_0\to\infty} \E \Big[ \int_{\bR^d} \Big\lvert \hat{r}_M(x) - r(x) \Big\rvert^p f_0(x) \ind(x \notin A_1 \cup A_2)\d x \Big] = 0.
\end{align}

{\bf Case II.} $x \in A_2 \setminus A_1$.

In this case, we have
\begin{align*}
  & \Big\lvert \frac{\nu_0(B_{x,\lVert z-x \rVert})}{\lambda(B_{x,\lVert z-x \rVert})} - f_0(x) \Big\rvert \le \epsilon f_0(x),~~ \Big\lvert \frac{\nu_0(B_{z,\lVert z-x \rVert})}{\lambda(B_{z,\lVert z-x \rVert})} - f_0(x) \Big\rvert \le \epsilon f_0(x),\\
  & \Big\lvert \frac{\nu_1(B_{x,\lVert z-x \rVert})}{\lambda(B_{x,\lVert z-x \rVert})} - f_1(x) \Big\rvert \le \epsilon^2,~~ \Big\lvert \frac{\nu_1(B_{z,\lVert z-x \rVert})}{\lambda(B_{z,\lVert z-x \rVert})} - f_1(x) \Big\rvert \le \epsilon^2.
\end{align*}

Take $\eta_N$ and take $N_0$ sufficiently large as in Case I above. Proceeding as the proof of Case II.2 of Lemma~\ref{lemma:moment,catch} and also Theorem~\ref{thm:cons,lp} by using Fubini's theorem, since $\epsilon$ is arbitrary, we obtain
\begin{align}\label{eq:risklp6}
  \lim_{N_0\to\infty} \E \Big[ \int_{\bR^d} \Big\lvert \hat{r}_M(x) - r(x) \Big\rvert^p f_0(x) \ind(x \in A_2 \setminus A_1)\d x \Big] = 0.
\end{align}

{\bf Case III.} $x \in A_1$.

In this case, for any $x \in A_1$ and $z \in S_1$,
\[
  \nu_0(B_{z,\lVert z-x \rVert}) \ge f_L \lambda(B_{z,\lVert z-x \rVert} \cap S_0) \ge af_L \lambda(B_{z,\lVert z-x \rVert}) \ge \frac{af_L}{f_U} \nu_1(B_{x,\lVert z-x \rVert}).
\]

Then for any $x \in A_1$, from \eqref{eq:momentcatch1} and in the same way as \eqref{eq:momentcatch3},
\begin{align*}
  & \Big(\frac{N_0}{M}\Big)^p \E\big[\nu_1^p\big(A_M(x)\big)\big] = \Big(\frac{N_0}{M}\Big)^p \P\Big(W_1 \le V_1, \ldots, W_p \le V_p\Big) \le \Big(\frac{N_0}{M}\Big)^p \P\Big(W_{\max} \le V_{\max}\Big)\\
  \le& \Big(\frac{N_0}{M}\Big)^p \P\Big(\frac{af_L}{f_U} \max_{k \in \zahl{p}} \nu_1(B_{x,\lVert \tZ_k - x \rVert}) \le V_{\max} \Big) \le \Big(\frac{f_U}{af_L}\Big)^p (1+o(1)) = O(1).
\end{align*}

Proceeding as in the proof of Theorem~\ref{thm:cons,lp}, and due to the boundedness assumptions on $f_0$ and $f_1$, for any $x \in A_1$ and $p$ even,
\begin{align*}
  \E \Big[ \Big\lvert \hat{r}_M(x) - r(x) \Big\rvert^p \Big] \lesssim \E \Big[ \Big\lvert \hat{r}_M(x) - \E[\hat{r}_M(x) \given \mX] \Big\rvert^p \Big] + \E \Big[ \Big( \E[\hat{r}_M(x) \given \mX] \Big)^p \Big] + \Big\lvert r(x) \Big\rvert^p \lesssim 1.
\end{align*}
Then
\[
  \E \Big[ \int_{\bR^d} \Big\lvert \hat{r}_M(x) - r(x) \Big\rvert^p f_0(x) \ind(x \in A_1)\d x \Big] \lesssim f_U \lambda(A_1) \lesssim \epsilon.
\]
Since $\epsilon$ is arbitrary, one has
\begin{align}\label{eq:risklp7}
  \lim_{N_0\to\infty} \E \Big[ \int_{\bR^d} \Big\lvert \hat{r}_M(x) - r(x) \Big\rvert^p f_0(x) \ind(x \in A_1)\d x \Big] = 0.
\end{align}

Combining \eqref{eq:risklp5}, \eqref{eq:risklp6}, and \eqref{eq:risklp7} completes the proof.
\end{proof}

\subsection{Proof of Theorem~\ref{thm:pw-rate}}

We only have to prove the first two claims as the rest are trivial.

\begin{proof}[Proof of Theorem~\ref{thm:pw-rate}\ref{thm:pw-rate1}]

For any $z \in \bR^d$ such that $\lVert z-x \rVert \le \delta/2$, since $B_{z,\lVert z-x \rVert} \subset B_{x,2\lVert z-x \rVert} \subset B_{x,\delta}$, we have
\begin{align*}
  \Big\lvert \frac{\nu_0(B_{z,\lVert z-x \rVert})}{\lambda(B_{z,\lVert z-x \rVert})} - f_0(x) \Big\rvert \le \frac{1}{\lambda(B_{z,\lVert z-x \rVert})} \int_{B_{z,\lVert z-x \rVert}} \lvert f_0(y) - f_0(x) \rvert \d y \le 2L \lVert z-x \rVert,\\
  \Big\lvert \frac{\nu_1(B_{x,\lVert z-x \rVert})}{\lambda(B_{x,\lVert z-x \rVert})} - f_1(x) \Big\rvert \le \frac{1}{\lambda(B_{x,\lVert z-x \rVert})} \int_{B_{x,\lVert z-x \rVert}} \lvert f_1(z) - f_1(x) \rvert \d z \le L \lVert z-x \rVert.
\end{align*}

Consider any $\delta_N>0$ such that $\delta_N \le \delta/2$. If $\lVert z-x \rVert \le \delta_N$ and $f_0(x) > 2L\delta_N$, then
\[
  \frac{f_0(x)-2L\delta_N}{f_1(x)+L\delta_N} \le \frac{\nu_0(B_{z,\lVert x-z \rVert})}{\lambda(B_{z,\lVert x-z \rVert})} \frac{\lambda(B_{x,\lVert x-z \rVert})}{\nu_1(B_{x,\lVert x-z \rVert})}.
\]
If further $f_1(x) > L\delta_N$, then
\[
  \frac{\nu_0(B_{z,\lVert x-z \rVert})}{\lambda(B_{z,\lVert x-z \rVert})} \frac{\lambda(B_{x,\lVert x-z \rVert})}{\nu_1(B_{x,\lVert x-z \rVert})} \le \frac{f_0(x)+2L\delta_N}{f_1(x)-L\delta_N}.
\]
On the other hand, if $\lVert z-x \rVert \ge \delta_N$ and $f_0(x) > 2L\delta_N$,
\[
  \nu_0(B_{z,\lVert z-x \rVert}) \ge (f_0(x)-2L\delta_N) \lambda(B_{0,\delta_N}) = (f_0(x)-2L\delta_N) V_d \delta_N^d,
\]
where $V_d$ is the Lebesgue measure of the unit ball on $\bR^d$.

Let $\delta_N = (\frac{4}{f_L V_d})^{1/d} (\frac{M}{N_0})^{1/d}$. Since $M/N_0 \to 0$, we have $\delta_N \to 0$ as $N_0 \to \infty$. Taking $N_0$ large enough so that $\delta_N < f_L/(4L)$ and $\delta_N \le \delta/2$, then
\[
  2LV_d \delta_N^{d+1} = \frac{M}{N_0} \frac{8L}{f_L} \delta_N < 2\frac{M}{N_0}.
\]
Then for any $(\nu_0,\nu_1) \in \cP_{x,\rm p}(f_L,f_U,L,d,\delta)$,
\[
  (f_0(x)-2L\delta_N) V_d \delta_N^d > 4\frac{f_0(x)}{f_L} \frac{M}{N_0} - 2\frac{M}{N_0} \ge 2\frac{M}{N_0}.
\]
With a slight abuse of notation, let $W := \nu_0(B_{Z,\lVert x-Z \rVert})$. Then $W \le 2\frac{M}{N_0}$ implies that $\lVert Z-x \rVert \le \delta_N$.

Depending on the value of $f_1(x)$, the proof is separated into two cases.

{\bf Case I.} $f_1(x) > L\delta_N$.

{\bf Upper bound.} Proceeding similar to \eqref{eq:meancatch1}, we have
\begin{align*}
  & \E\big[\hat{r}_M(x)\big] = \frac{N_0}{M} \E\big[\nu_1\big(A_M(x)\big)\big] = \frac{N_0}{M} \P \Big(W \le \nu_0(B_{Z,\lVert \cX_{(M)}(Z) - Z \rVert})\Big) \\
  \le& \frac{N_0}{M} \P \Big(W \le \nu_0(B_{Z,\lVert \cX_{(M)}(Z) - Z \rVert}), \nu_0(B_{Z,\lVert \cX_{(M)}(Z) - Z \rVert}) \le 2 \frac{M}{N_0}\Big) + \frac{N_0}{M} \P\Big(U_{(M)} > 2 \frac{M}{N_0}\Big)\\
  =& \frac{N_0}{M} \P \Big(W \le \nu_0(B_{Z,\lVert \cX_{(M)}(Z) - Z \rVert}), \nu_0(B_{Z,\lVert \cX_{(M)}(Z) - Z \rVert}) \le 2 \frac{M}{N_0}, \lVert Z - x \rVert \le \delta_N \Big) + \frac{N_0}{M} \P\Big(U_{(M)} > 2 \frac{M}{N_0}\Big)\\
  \le& \frac{N_0}{M} \P \Big(W \le \nu_0(B_{Z,\lVert \cX_{(M)}(Z) - Z \rVert}), \lVert Z - x \rVert \le \delta_N\Big) + \frac{N_0}{M} \P\Big(U_{(M)} > 2 \frac{M}{N_0}\Big)\\
  \le& \frac{N_0}{M} \P \Big(\frac{f_0(x)-2L\delta_N}{f_1(x)+L\delta_N} \nu_1(B_{x,\lVert x-Z \rVert}) \le \nu_0(B_{Z,\lVert \cX_{(M)}(Z) - Z \rVert}), \lVert Z - x \rVert \le \delta_N\Big) + \frac{N_0}{M} \P\Big(U_{(M)} > 2 \frac{M}{N_0}\Big)\\
  \le& \frac{N_0}{M} \P \Big(\frac{f_0(x)-2L\delta_N}{f_1(x)+L\delta_N} \nu_1(B_{x,\lVert x-Z \rVert}) \le \nu_0(B_{Z,\lVert \cX_{(M)}(Z) - Z \rVert})\Big) + \frac{N_0}{M} \P\Big(U_{(M)} > 2 \frac{M}{N_0}\Big)\\
  =& \frac{N_0}{M} \P \Big(\frac{f_0(x)-2L\delta_N}{f_1(x)+L\delta_N} U \le U_{(M)}\Big) + \frac{N_0}{M} \P\Big(U_{(M)} > 2 \frac{M}{N_0}\Big).
  \yestag\label{eq:ratebias1}  
\end{align*}

For the second term in \eqref{eq:ratebias1}, since $M/\log N_0 \to \infty$, for any $\gamma>0$,
\begin{align}\label{eq:ratebias-new1}
  \frac{N_0}{M} \P\Big(U_{(M)} > 2 \frac{M}{N_0}\Big) &= \frac{N_0}{M} \P\Big({\rm Bin}\Big(N_0,2 \frac{M}{N_0}\Big) \le M \Big)\notag\\
  \le& \frac{N_0}{M}\exp\Big((\log2 -1) M\Big) = \frac{N_0}{M} N_0^{-(1-\log 2)M/\log N_0} \prec N_0^{-\gamma}.
\end{align}

For the first term in \eqref{eq:ratebias1}, proceeding as \eqref{eq:meancatch8}, we obtain
\begin{align*}
  \frac{N_0}{M} \P \Big(\frac{f_0(x)-2L\delta_N}{f_1(x)+L\delta_N} U \le U_{(M)}\Big) \le \frac{f_1(x)+L\delta_N}{f_0(x)-2L\delta_N} \frac{N_0}{N_0+1}.
\end{align*}

Then we obtain
\begin{align}\label{eq:ratebias2}
  \E\big[\hat{r}_M(x)\big] \le \frac{f_1(x)+L\delta_N}{f_0(x)-2L\delta_N} \frac{N_0}{N_0+1} + o(N_0^{-\gamma}).
\end{align}

{\bf Lower bound.} Proceeding similar to \eqref{eq:meancatch5}, we have
\begin{align*}
  & \E\big[\hat{r}_M(x)\big] = \frac{N_0}{M} \E\big[\nu_1\big(A_M(x)\big)\big] = \frac{N_0}{M} \P \Big(W \le \nu_0(B_{Z,\lVert \cX_{(M)}(Z) - Z \rVert})\Big) \\
  \ge& \frac{N_0}{M} \P \Big(W \le \nu_0(B_{Z,\lVert \cX_{(M)}(Z) - Z \rVert}), \nu_0(B_{Z,\lVert \cX_{(M)}(Z) - Z \rVert}) \le 2 \frac{M}{N_0}\Big)\\
  =& \frac{N_0}{M} \P \Big(W \le \nu_0(B_{Z,\lVert \cX_{(M)}(Z) - Z \rVert}), \nu_0(B_{Z,\lVert \cX_{(M)}(Z) - Z \rVert}) \le 2 \frac{M}{N_0}, \lVert Z - x \rVert \le \delta_N\Big)\\
  \ge& \frac{N_0}{M} \P \Big(\frac{f_0(x)+2L\delta_N}{f_1(x)-L\delta_N} \nu_1(B_{x,\lVert x-Z \rVert}) \le \nu_0(B_{Z,\lVert \cX_{(M)}(Z) - Z \rVert}), \nu_0(B_{Z,\lVert \cX_{(M)}(Z) - Z \rVert}) \le 2 \frac{M}{N_0}, \lVert Z - x \rVert \le \delta_N\Big)\\
  =& \frac{N_0}{M} \P \Big(\frac{f_0(x)+2L\delta_N}{f_1(x)-L\delta_N} \nu_1(B_{x,\lVert x-Z \rVert}) \le \nu_0(B_{Z,\lVert \cX_{(M)}(Z) - Z \rVert}), \nu_0(B_{Z,\lVert \cX_{(M)}(Z) - Z \rVert}) \le 2 \frac{M}{N_0}\Big)\\
  \ge& \frac{N_0}{M} \P \Big(\frac{f_0(x)+2L\delta_N}{f_1(x)-L\delta_N} U \le U_{(M)}\Big) - \frac{N_0}{M} \P\Big(U_{(M)} > 2 \frac{M}{N_0}\Big)\\
  =& \frac{f_1(x)-L\delta_N}{f_0(x)+2L\delta_N} \int_0^{\frac{f_0(x)+2L\delta_N}{f_1(x)-L\delta_N} \frac{N_0}{M}} \P \Big( U_{(M)} \ge \frac{M}{N_0}t \Big) \d t - \frac{N_0}{M} \P\Big(U_{(M)} > 2 \frac{M}{N_0}\Big).
\end{align*}

Consider the first term. If $\frac{f_0(x)+2L\delta_N}{f_1(x)-L\delta_N} \ge 1$, then
\begin{align*}
  \frac{f_1(x)-L\delta_N}{f_0(x)+2L\delta_N} \int_0^{\frac{f_0(x)+2L\delta_N}{f_1(x)-L\delta_N} \frac{N_0}{M}} \P \Big( U_{(M)} \ge \frac{M}{N_0}t \Big) \d t = \frac{f_1(x)-L\delta_N}{f_0(x)+2L\delta_N} \frac{N_0}{N_0+1}.
\end{align*}
If $\frac{f_0(x)+2L\delta_N}{f_1(x)-L\delta_N} < 1$, using the Chernoff bound, for any $\gamma>0$,
\begin{align*}
  & \int_{\frac{f_0(x)+2L\delta_N}{f_1(x)-L\delta_N} \frac{N_0}{M}}^{\frac{N_0}{M}} \P \Big( U_{(M)} \ge \frac{M}{N_0}t \Big) \d t \le \int_{\frac{f_0(x)}{f_1(x)} \frac{N_0}{M}}^{\frac{N_0}{M}} \P \Big( U_{(M)} \ge \frac{M}{N_0}t \Big) \d t \le \int_{\frac{f_L}{f_U} \frac{N_0}{M}}^{\frac{N_0}{M}} \P \Big( U_{(M)} \ge \frac{M}{N_0}t \Big) \d t \\
  \le& \Big[1-\frac{f_L}{f_U}\Big] \frac{N_0}{M} \P \Big( U_{(M)} \ge \frac{f_L}{f_U} \Big) \le  \Big[1-\frac{f_L}{f_U} \Big] \frac{N_0}{M} \exp \Big[ M - \frac{f_L}{f_U} N_0 - M\log M + M \log\Big(\frac{f_L}{f_U} N_0\Big)\Big] \\
  \prec & N_0^{-\gamma}.
\end{align*}
The last step is due to $M \log N_0/N_0 \to 0$. Recalling \eqref{eq:ratebias-new1}, we then obtain
\begin{align}\label{eq:ratebias3}
  \E\big[\hat{r}_M(x)\big] \ge \frac{f_1(x)-L\delta_N}{f_0(x)+2L\delta_N} \frac{N_0}{N_0+1} - o(N_0^{-\gamma}).
\end{align}

Combining \eqref{eq:ratebias2} and \eqref{eq:ratebias3}, and taking $N_0$ large enough so that $L\delta_N \le f_U \wedge (f_L/4)$, one obtains
\begin{align*}
  \Big\lvert \E\big[\hat{r}_M(x)\big] - r(x) \Big\rvert &\le \Big\lvert \frac{f_1(x)+L\delta_N}{f_0(x)-2L\delta_N} \frac{N_0}{N_0+1} - \frac{f_1(x)}{f_0(x)} \Big\rvert \bigvee \Big\lvert \frac{f_1(x)-L\delta_N}{f_0(x)+2L\delta_N} \frac{N_0}{N_0+1} - \frac{f_1(x)}{f_0(x)} \Big\rvert + o(N_0^{-\gamma})\\
  &\le \frac{f_0(x)L\delta_N + 2f_1(x)L\delta_N}{f_0(x)(f_0(x)-2L\delta_N)} + \frac{1}{N_0+1} \frac{f_1(x)+L\delta_N}{f_0(x)-2L\delta_N} + o(N_0^{-\gamma})\\
  &\le \Big( \frac{2}{f_L} + \frac{4f_U}{f_L^2}\Big) L \delta_N + \frac{4f_U}{f_L}\frac{1}{N_0+1} + o(N_0^{-\gamma}).
\end{align*}

By the selection of $\delta_N$ and that the right hand side does not depend on $x$, we complete the proof for this case.

{\bf Case II.} $f_1(x) \le L\delta_N$.

The upper bound \eqref{eq:ratebias2} in Case I still holds for this case. Accordingly, taking $N_0$ large enough so that $L\delta_N \le f_L/4$, one has
\begin{align*}
  &\Big\lvert \E\big[\hat{r}_M(x)\big] - r(x) \Big\rvert \le \E\big[\hat{r}_M(x)\big] + r(x)\\
  \le& \frac{f_1(x)+L\delta_N}{f_0(x)-2L\delta_N} \frac{N_0}{N_0+1} + \frac{f_1(x)}{f_0(x)} + o(N_0^{-\gamma}) \le \frac{4}{f_L} L\delta_N + \frac{1}{f_L} L\delta_N + o(N_0^{-\gamma}).
\end{align*}

We thus complete the whole proof.
\end{proof}

\begin{proof}[Proof of Theorem~\ref{thm:pw-rate}\ref{thm:pw-rate2}]

By the law of total variance,
\begin{align}\label{eq:ratevar1}
  \Var\big[\hat{r}_M(x)\big] = \E\Big[\Var\Big[\hat{r}_M(x) \Biggiven \mX\Big]\Big] + \Var\Big[\E\Big[\hat{r}_M(x) \Biggiven \mX\Big]\Big].
\end{align}

For the first term in \eqref{eq:ratevar1}, let $Z$ be a copy drawn from $\nu_1$ independently of the data. Then, since $[Z_j]_{j=1}^{N_1}$ are i.i.d, 
\begin{align*}
  &\E\Big[\Var\Big[\hat{r}_M(x) \Biggiven \mX\Big]\Big] = \E\Big[\Var\Big[\frac{N_0}{N_1M} K_M(x) \Biggiven \mX\Big]\Big]\\
  =& \Big(\frac{N_0}{N_1M}\Big)^2 \E\Big[\Var\Big[ \sum_{j=1}^{N_1} \ind\big(Z_j \in A_M(x)\big) \Biggiven \mX\Big]\Big] = \frac{N_0^2}{N_1M^2} \E\Big[\Var\Big[ \ind\big(Z \in A_M(x)\big) \Biggiven \mX\Big]\Big]\\
  =& \frac{N_0^2}{N_1M^2} \E\Big[ \nu_1\big(A_M(x)\big) - \nu_1^2\big(A_M(x)\big)\Big] \le \frac{N_0^2}{N_1M^2} \E\Big[ \nu_1\big(A_M(x)\big)\Big] \\
  =& \frac{N_0}{N_1M} \E\big[\hat{r}_M(x)\big] \lesssim C \frac{N_0}{N_1M},
  \yestag\label{eq:ratevar11}
\end{align*}
where $C>0$ is a constant only depending on $f_L,f_U$. The last step is due to \eqref{eq:ratebias2}.

For the second term in \eqref{eq:ratevar1}, notice that
\begin{align*}
  \Var\Big[\E\Big[\hat{r}_M(x) \Biggiven \mX\Big]\Big] = \Var\Big[\E\Big[\frac{N_0}{N_1M} K_M(x) \Biggiven \mX\Big]\Big] = \Big(\frac{N_0}{M}\Big)^2 \Var\Big[  \nu_1\big(A_M(x)\big) \Big].
\end{align*}
Recalling that $W = \nu_0(B_{Z,\lVert x-Z \rVert})$, we have the following lemma about the density of $W$ near $x$.

\begin{lemma}\label{lemma:z,density}
  Denote the density of $W$ by $f_W$. Then for any $(\nu_0,\nu_1) \in \cP_{x,\rm p}(f_L,f_U,L,d,\delta)$, 
\[  
f_W(0) = r(x). 
\]
Furthermore, for any $\epsilon>0$ and $N_0$ sufficiently large, we have for all $0 \le w \le 2M/N_0$,
  \[
    \sup_{(\nu_0,\nu_1) \in \cP_{\rm p}(f_L,f_U,\delta,L,d)} f_W(w) \le (1+\epsilon)\frac{f_U}{f_L}.
  \]
\end{lemma}

Due to Lemma~\ref{lemma:z,density}, we can take $N_0$ sufficiently large so that for any $0 \le w \le 2M/N_0$,
\[
  \sup_{(\nu_0,\nu_1) \in \cP_{\rm p}(f_L,f_U,\delta,L,d)} f_W(w) \le 2\frac{f_U}{f_L}.
\]

Let $Z,\tZ$ be two independent copies from $\nu_1$ that are further independent of the data. Let $W = \nu_0(B_{Z,\lVert x-Z \rVert})$ and $\tW = \nu_0(B_{\tZ,\lVert x-\tZ \rVert})$. Let $V = \nu_0(B_{Z,\lVert \cX_{(M)}(Z) -Z \rVert})$ and $\tV = \nu_0(B_{\tZ,\lVert \cX_{(M)}(\tZ) -\tZ \rVert})$. We then have
\begin{align*}
  \Var\Big[  \nu_1\big(A_M(x)\big) \Big] =& \E\Big[ \nu_1^2\big(A_M(x)\big) \Big] - \Big(\E\Big[ \nu_1\big(A_M(x)\big) \Big]\Big)^2\\
  =& \P \Big(Z \in A_M(x), \tZ \in A_M(x)\Big) - \P \Big(Z \in A_M(x)\Big) \P \Big(\tZ \in A_M(x)\Big)\\
  =& \P \Big(\lVert x - Z \rVert \le \lVert \cX_{(M)}(Z) - Z \rVert , \lVert x - \tZ \rVert \le \lVert \cX_{(M)}(\tZ) - \tZ \rVert \Big) \\
  & - \P \Big(\lVert x - Z \rVert \le \lVert \cX_{(M)}(Z) - Z \rVert\Big) \P \Big(\lVert x - \tZ \rVert \le \lVert \cX_{(M)}(\tZ) - \tZ \rVert\Big)\\
  =& \P \Big(W \le V, \tW \le \tV\Big) - \P \Big(W \le V\Big) \P \Big(\tW \le \tV\Big).
\end{align*}

Due to the independence between $Z$ and $\tZ$, $W$ and $\tW$ are independent. Notice that $V \given Z$ have the same distribution as $U_{(M)}$ for any $Z \in \bR^d$, then $V$ and $Z$ are independent, so are $\tV$ and $\tZ$.

Let's expand the variance further as 
\begin{align*}
  & \Var\Big[  \nu_1\big(A_M(x)\big) \Big]\\
  = & \Big[\P \Big(W \le V, \tW \le \tV, W \le 2\frac{M}{N_0}, \tW \le 2\frac{M}{N_0} \Big) - \P \Big(W \le V, W \le 2\frac{M}{N_0}\Big) \P \Big(\tW \le \tV, \tW \le 2\frac{M}{N_0}\Big) \Big]\\
  &+ \Big[\P \Big(W \le V, \tW \le \tV\Big) - \P \Big(W \le V, \tW \le \tV, W \le 2\frac{M}{N_0}, \tW \le 2\frac{M}{N_0} \Big) \Big]\\
  &- \Big[\P \Big(W \le V\Big) \P \Big(\tW \le \tV\Big) - \P \Big(W \le V, W \le 2\frac{M}{N_0}\Big) \P \Big(\tW \le \tV, \tW \le 2\frac{M}{N_0}\Big) \Big].
  \yestag\label{eq:ratevar2} 
\end{align*}

For the first term in \eqref{eq:ratevar2}, we have the following lemma.

\begin{lemma}\label{lemma:ratevar} We have
  \begin{align*}
    & \Big(\frac{N_0}{M}\Big)^2 \Big[\P \Big(W \le V, \tW \le \tV, W \le 2\frac{M}{N_0}, \tW \le 2\frac{M}{N_0} \Big) - \P \Big(W \le V, W \le 2\frac{M}{N_0}\Big) \P \Big(\tW \le \tV, \tW \le 2\frac{M}{N_0}\Big) \Big] \\
    & \leq C \frac{1}{M},
  \end{align*}
  where $C>0$ is a constant only depending on $f_L,f_U$.
\end{lemma}

For the second term in \eqref{eq:ratevar2},
\begin{align*}
  &\P \Big(W \le V, \tW \le \tV\Big) - \P \Big(W \le V, \tW \le \tV, W \le 2\frac{M}{N_0}, \tW \le 2\frac{M}{N_0} \Big)\\
  \le&  \P \Big(W \le V, \tW \le \tV, W > 2\frac{M}{N_0}\Big) + \P \Big(W \le V, \tW \le \tV, \tW > 2\frac{M}{N_0}\Big)\\
  \le & \P \Big(V > 2\frac{M}{N_0}\Big) + \P \Big(\tV > 2\frac{M}{N_0}\Big) \\
  =& 2 \P \Big(U_{(M)} > 2\frac{M}{N_0}\Big).
\end{align*}
Using the Chernoff bound and $M/\log N_0\to \infty$, for any $\gamma>0$,
\begin{align*}
  \Big(\frac{N_0}{M}\Big)^2 \P \Big(U_{(M)} > 2\frac{M}{N_0}\Big) \le \Big(\frac{N_0}{M}\Big)^2 \exp \Big[ -(1 - \log 2)M \Big] \prec N_0^{-\gamma},
\end{align*}
We then have
\begin{align*}
  & \Big(\frac{N_0}{M}\Big)^2 \Big[\P \Big(W \le V, \tW \le \tV\Big) - \P \Big(W \le V, \tW \le \tV, W \le 2\frac{M}{N_0}, \tW \le 2\frac{M}{N_0} \Big) \Big]\\
  &- \Big(\frac{N_0}{M}\Big)^2 \Big[\P \Big(W \le V\Big) \P \Big(\tW \le \tV\Big) - \P \Big(W \le V, W \le 2\frac{M}{N_0}\Big) \P \Big(\tW \le \tV, \tW \le 2\frac{M}{N_0}\Big) \Big]\\
  \le& 2 \Big(\frac{N_0}{M}\Big)^2 \P \Big(U_{(M)} > 2\frac{M}{N_0}\Big) \prec N_0^{-\gamma}.
  \yestag\label{eq:ratevar9}
\end{align*}

For the third term in  \eqref{eq:ratevar2}, it is easy to see
\[
  \Big[\P \Big(W \le V\Big) \P \Big(\tW \le \tV\Big) - \P \Big(W \le V, W \le 2\frac{M}{N_0}\Big) \P \Big(\tW \le \tV, \tW \le 2\frac{M}{N_0}\Big) \Big] \ge 0.
\]

Plugging Lemma~\ref{lemma:ratevar} and \eqref{eq:ratevar9} into \eqref{eq:ratevar2}, we obtain
\begin{align}\label{eq:ratevar12}
  \Big(\frac{N_0}{M}\Big)^2 \Var\Big[  \nu_1\big(A_M(x)\big) \Big] \lesssim C \frac{1}{M},
\end{align}
where $C>0$ is a constant only depending on $f_L,f_U$.

Plugging \eqref{eq:ratevar11} and \eqref{eq:ratevar12} into \eqref{eq:ratevar1} completes the proof.
\end{proof}

\subsection{Proof of Theorem~\ref{thm:rate,risk}}

\begin{proof}[Proof of Theorem~\ref{thm:rate,risk}]

We only have to prove the first claim as the second is trivial.

Take $\delta_N = (\frac{4}{f_L V_d})^{1/d} (\frac{M}{N_0})^{1/d}$ as in the proof of Theorem~\ref{thm:pw-rate}\ref{thm:pw-rate1}. Take $\delta_N' = (\frac{2}{a f_L V_d})^{1/d} (\frac{M}{N_0})^{1/d}$. For any $x \in \bR^d$, denote the distance of $x$ to the boundary of $S_1$ by $\Delta(x)$, i.e., $\Delta(x) = \inf_{z \in \partial S_1} \lVert z-x \rVert$. 

Depending on the position of $x$ and the value of $\Delta(x)$, we separate the proof into three cases .

{\bf Case I.} $x \in S_1$ and $\Delta(x) > 2\delta_N$.

In this case, since $\Delta(x) > 2\delta_N$, for any $\lVert z - x \rVert \le \delta_N$, we have $B_{z,\lVert z-x \rVert} \subset S_1$. From the smoothness conditions on $f_0$ and $f_1$, similar to the proof of Theorem~\ref{thm:pw-rate}, we have
\begin{align*}
  & \E \Big[ \int_{\bR^d} \Big\lvert \hat{r}_M(x) - r(x) \Big\rvert f_0(x) \ind(x \in S_1, \Delta(x) > 2\delta_N) \d x \Big] \\
  = & \int_{\bR^d} \E \Big[ \Big\lvert \hat{r}_M(x) - r(x) \Big\rvert \Big] f_0(x) \ind(x \in S_1, \Delta(x) > 2\delta_N) \d x \\
  \le & \int_{\bR^d} \Big(\E \Big[ \hat{r}_M(x) - r(x) \Big]^2 \Big)^{1/2} f_0(x) \ind(x \in S_1, \Delta(x) > 2\delta_N) \d x \\
  \le & C \Big[ \Big(\frac{M}{N_0}\Big)^{1/d} + \Big(\frac{1}{M}\Big)^{1/2} + \Big(\frac{N_0}{MN_1}\Big)^{1/2}\Big] \int_{\bR^d} f_0(x) \ind(x \in S_1, \Delta(x) > 2\delta_N) \d x\\
  \le & C \Big[ \Big(\frac{M}{N_0}\Big)^{1/d} + \Big(\frac{1}{M}\Big)^{1/2} + \Big(\frac{N_0}{MN_1}\Big)^{1/2}\Big] \int_{\bR^d} f_0(x) \d x\\
  = & C \Big[ \Big(\frac{M}{N_0}\Big)^{1/d} + \Big(\frac{1}{M}\Big)^{1/2} + \Big(\frac{N_0}{MN_1}\Big)^{1/2}\Big],
  \yestag\label{eq:raterisk1}
\end{align*}
where the constant $C>0$ only depends on $f_L,f_U,L,d$. 

{\bf Case II.} $x \in S_0 \setminus S_1$ and $\Delta(x) > \delta_N'$.

In this case, $r(x) = 0$ and for any $z \in S_1$,
\[
  \nu_0(B_{z,\lVert z-x \rVert}) \ge f_L \lambda(B_{z,\lVert z-x \rVert} \cap S_0) \ge af_L \lambda(B_{z,\lVert z-x \rVert}) > af_L V_d \delta_N'^d \ge 2\frac{M}{N_0}.
\]
Then for any $\gamma>0$,
\begin{align*}
  &\E \Big[ \Big\lvert \hat{r}_M(x) - r(x) \Big\rvert \Big] = \E \Big[ \hat{r}_M(x) \Big] = \frac{N_0}{M} \E\big[\nu_1\big(A_M(x)\big)\big]\\
  =& \frac{N_0}{M} \P \Big(W \le \nu_0(B_{Z,\lVert \cX_{(M)}(Z) - Z \rVert})\Big) \le \frac{N_0}{M} \P\Big(U_{(M)} > 2 \frac{M}{N_0}\Big) \prec N_0^{-\gamma}.
\end{align*}
One then obtains
\begin{align*}
  & \E \Big[ \int_{\bR^d} \Big\lvert \hat{r}_M(x) - r(x) \Big\rvert f_0(x) \ind(x \notin S_1, \Delta(x) > \delta_N') \d x \Big] \\
  \prec & N_0^{-\gamma} \int_{\bR^d} f_0(x) \ind(x \in S_0 \setminus S_1, \Delta(x) > \delta_N') \d x \le N_0^{-\gamma}.
  \yestag\label{eq:raterisk2}
\end{align*}

{\bf Case III.} $x \in S_0$ and $\Delta(x) \le (2\delta_N) \vee \delta_N'$.

In this case, for any $z \in S_1$,
\[
  \nu_0(B_{z,\lVert z-x \rVert}) \ge f_L \lambda(B_{z,\lVert z-x \rVert} \cap S_0) \ge af_L \lambda(B_{z,\lVert z-x \rVert}) \ge \frac{af_L}{f_U} \nu_1(B_{x,\lVert z-x \rVert}).
\]
Accordingly, 
\begin{align*}
  &\E \Big[ \Big\lvert \hat{r}_M(x) - r(x) \Big\rvert \Big] \le \E \Big[ \hat{r}_M(x) \Big] + r(x) = \frac{N_0}{M} \P \Big(W \le \nu_0(B_{Z,\lVert \cX_{(M)}(Z) - Z \rVert}) \Big) + r(x)\\
  \le& \frac{N_0}{M} \P \Big( \frac{af_L}{f_U} \nu_1(B_{x,\lVert x-Z \rVert}) \le \nu_0(B_{Z,\lVert \cX_{(M)}(Z) - Z \rVert}), \nu_0(B_{Z,\lVert \cX_{(M)}(Z) - Z \rVert}) \le 2 \frac{M}{N_0}\Big) + r(x)\\
  \le& \frac{N_0}{M} \P\Big(\frac{af_L}{f_U} U \le U_{(M)} \Big) + r(x) = \frac{f_U}{af_L} (1+o(1)) + \frac{f_U}{f_L}.
\end{align*}

From the definition of $\delta_N,\delta_N'$ and $M/N_0 \to 0$, we have $\delta_N,\delta_N' \to 0$ as $N_0 \to \infty$. Since the surface area of $S_1$ is bounded by $H$, we have
\[
  \lambda(\{x:\Delta(x) \le (2\delta_N) \vee \delta_N'\}) \lesssim H \{(2\delta_N) \vee \delta_N'\}.
\]
Then we obtain
\begin{align*}
  & \E \Big[ \int_{\bR^d} \Big\lvert \hat{r}_M(x) - r(x) \Big\rvert f_0(x) \ind(\Delta(x) \le (2\delta_N) \vee \delta_N') \d x \Big] \\
  = & \int_{\bR^d} \E \Big[ \Big\lvert \hat{r}_M(x) - r(x) \Big\rvert \Big] f_0(x) \ind(\Delta(x) \le (2\delta_N) \vee \delta_N') \d x \\
  \le & \Big(\frac{f_U}{af_L} (1+o(1)) + \frac{f_U}{f_L} \Big)  \int_{\bR^d} f_0(x) \ind(\Delta(x) \le (2\delta_N) \vee \delta_N') \d x\\
  \le & \Big(\frac{f_U}{af_L} (1+o(1)) + \frac{f_U}{f_L} \Big) f_U \lambda(\{x:\Delta(x) \le (2\delta_N) \vee \delta_N'\})\\
  \lesssim & \Big(\frac{f_U}{af_L} + \frac{f_U}{f_L} \Big) f_U H (\delta_N + \delta_N') \le C \Big(\frac{M}{N_0}\Big)^{1/d},
  \yestag\label{eq:raterisk3}
\end{align*}
where the constant $C>0$ only depends on $f_L,f_U,a,H,d$. 

Combining \eqref{eq:raterisk1}, \eqref{eq:raterisk2}, \eqref{eq:raterisk3} completes the proof.
\end{proof}

\appendix

\section{Proofs of the rest results}\label{sec:proof}

\subsection{Proofs of the results in Section~\ref{sec:computation}}

\subsubsection{Proof of Theorem~\ref{prop:complexity}}

\begin{proof}[Proof of Theorem~\ref{prop:complexity}]

We consider the complexities of two algorithms seperately.

{\bf Algorithm~\ref{alg:1}.}

The worst-case computation complexity of building a balanced $k$-d tree is $O(d N_0 \log N_0)$ (cf. \cite{brown2015building}) since the size of the k-d tree is $N_0$. 

The average computation complexity of searching a NN is $O(\log N_0)$ from \cite{10.1145/355744.355745}, and then the average computation complexity of search $M$-NNs in $\{X_i\}_{i=1}^{N_0}$ for all $\{Z_j\}_{j=1}^{N_1}$ is $O(M N_1 \log N_0)$.

Notice that $\lvert S_j \rvert = M$ for any $j \in \zahl{N_1}$ and then $\big\lvert \bigcup_{j=1}^{N_1} S_j \big\rvert \le N_1 M$. Since the elements of each $S_j$ are in $\zahl{N_0}$, the largest integer in $\bigcup_{j=1}^{N_1} S_j$ is $N_0$. Then the computation complexity of counting step is $O(N_1 M + N_0)$ due to the counting sort algorithm \cite[Section 8.2]{cormen2009introduction}.

Combining the above three steps completes the proof for Algorithm \ref{alg:1}.

{\bf Algorithm~\ref{alg:2}.}

The computation complexity of building a $k$-d tree is $O(d(N_0+n) \log (N_0+n))$ from Algorithm~\ref{alg:1} since the size of the $k$-d tree is $N_0 + n$.

For the searching step, for each $j \in \zahl{N_1}$, the number of NNs to be searched is 
\[
  M + \sum_{i=1}^n \ind\big(\lVert x_i - Z_j \rVert \le \lVert \cX_{(M)}(Z_j) - Z_j \rVert\big).
\]
Then from \eqref{eq:KM}, the total number of NNs searched for all $j \in \zahl{N_1}$ is 
\[
  \sum_{j=1}^{N_1} \Big(M + \sum_{i=1}^n \ind\big(\lVert x_i - Z_j \rVert \le \lVert \cX_{(M)}(Z_j) - Z_j \rVert\big) \Big) = N_1 M + \sum_{i=1}^n K_M(x_i).
\]
Let $X, Y$ be two independent copies from $\nu_0,\nu_1$, respectively. Since $[Z_j]_{j=1}^{N_1}$ are i.i.d. and $[X_i]_{i=1}^{N_0} \bigcup [x_i]_{i=1}^n$ are i.i.d, we have
\[
  \E \Big[ \sum_{i=1}^n K_M(x_i) \Big] = n \E [K_M(X)] = N_1 n \E [\nu_1(A_M(X))] = N_1 n \frac{M}{N_0+1},
\]
where the last equation is due to
\[
  \E [\nu_1(A_M(X))] = \P\Big(\lVert X - Y \rVert \le \lVert \cX_{(M)}(Y) - Y \rVert\Big) = \P\Big(U \le U_{(M)}\Big) = \frac{M}{N_0+1},
\]
by using the probability integral transform. Then the average computation complexity for the searching step is $O( \frac{N_1M}{N_0} (N_0+n) \log (N_0 + n))$

For the counting step, the computation complexity for counting $\bigcup_{j=1}^{N_1} S_j$ is $O(N_0 + N_1 M)$ since the cardinality of $\bigcup_{j=1}^{N_1} S_j$ is $N_1 M$ and the largest integer is $N_0$. The average computation complexity for counting $\bigcup_{j=1}^{N_1} S'_j$ is $O(\frac{N_1 M}{N_0} n + n)$ since the average cardinality of $\bigcup_{j=1}^{N_1} S'_j$ is $\frac{N_1 M}{N_0} n$ and the largest integer is $n$.

Combining the above three steps completes the proof for Algorithm \ref{alg:2}.
\end{proof}

\subsection{Proofs of the results in Section~\ref{sec:theory}}

\subsubsection{Proof of Corollary~\ref{crl:moment,catch,random}}

\begin{proof}[Proof of Corollary~\ref{crl:moment,catch,random}]

Corollary \ref{crl:moment,catch,random} can be established following the same way as that of Theorem~\ref{thm:risk,lp} but with less effort since we only have to show
\[
  \lim_{N_0\to\infty} \E \Big[ \int_{\bR^d} \Big\lvert \E\Big[\hat{r}_M(x)\Biggiven \mX \Big] - r(x) \Big\rvert^p f_0(x) \d x \Big] = 0.
\]

In detail, denote the Radon-Nikodym derivative of the probability measure of $W$ with respect to $\nu_0$ by $r_W$. We then have
\begin{align*}
  & \limsup_{N_0\to\infty} \E \Big[ \Big\lvert \frac{N_0}{M} \nu_1\big(A_M(W)\big) - r(W) \Big\rvert^p \Big]  =  \limsup_{N_0\to\infty} \E \Big[ \int_{\bR^d} \Big\lvert \frac{N_0}{M} \nu_1\big(A_M(x)\big) - r(x) \Big\rvert^p r_W(x) \d \nu_0(x) \Big]\\
  =& \limsup_{N_0\to\infty} \E \Big[ \int_{\bR^d} \Big\lvert \frac{N_0}{M} \nu_1\big(A_M(x)\big) - r(x) \Big\rvert^p r_W(x) f_0(x) \d x \Big] \\
  \lesssim& \limsup_{N_0\to\infty} \E \Big[ \int_{\bR^d} \Big\lvert \frac{N_0}{M} \nu_1\big(A_M(x)\big) - r(x) \Big\rvert^p f_0(x) \d x \Big] \\
  =& \limsup_{N_0\to\infty} \E \Big[ \int_{\bR^d} \Big\lvert \E\Big[\hat{r}_M(x)\Biggiven \mX \Big] - r(x) \Big\rvert^p f_0(x) \d x \Big] = 0.
\end{align*}
Here the last line has intrinsically been established in the proof of Theorem~\ref{thm:risk,lp}.

Noticing that $\E \big[r(W)\big]^p$ is bounded under Assumption~\ref{asp:risk}, the proof is thus complete.
\end{proof}

\subsubsection{Proof of Proposition~\ref{thm:minimax,mse}}

\begin{proof}[Proof of Proposition~\ref{thm:minimax,mse}]

We take $\nu_0$ and $\nu_1$ to share the same support, and assume $x$ to be the origin of $\bR^d$ without loss of generality.

When $N_1 \lesssim N_0$, we take $\nu_0$ to be the uniform distribution with density $f_L$ on $[-f_L^{-1/d}/2,f_L^{-1/d}/2]^d$. Then the MSE is lower bounded by the density estimation over Lipchitz class with $N_1$ samples.

When $N_0 \lesssim N_1$, we take $\nu_1$ to be the uniform distribution with density $f_U$ on $[-f_U^{-1/d}/2,f_U^{-1/d}/2]^d$. Notice that $1/f_0$ is also local Lipchitz from the lower boundness condition and local Lipchitz condition on $f_0$. Then the MSE is lower bounded by the density estimation over Lipchitz class with $N_0$ samples.

We then complete the proof by combining the above two lower bounds and then using the famous minimax lower bound in Lipschitz density estimation \citep[Exercise 2.8]{MR2724359}, 
\end{proof}

\subsubsection{Proof of Proposition~\ref{thm:minimax,risk}}

\begin{proof}[Proof of Proposition~\ref{thm:minimax,risk}]

We take $\nu_0$ and $\nu_1$ to be of the same support.

When $N_1 \lesssim N_0$, we take $\nu_0$ to be the uniform distribution with density $f_L$ on $[-f_L^{-1/d}/2,f_L^{-1/d}/2]^d$. Then the $L_1$ risk is lower bounded by the $L_1$ risk over support of density estimation over Lipchitz class with $N_1$ samples.

When $N_0 \lesssim N_1$, we take $\nu_1$ to be the uniform distribution with density $f_U$ on $[-f_U^{-1/d}/2,f_U^{-1/d}/2]^d$. Notice $1/f_0$ is also Lipchitz from the lower boundness condition and Lipchitz condition on $f_0$. From the lower boundness condition on $f_0$, the $L_1$ risk is lower bounded by the $L_1$ risk over support of density estimation over Lipchitz class with $N_0$ samples.

We then complete the proof by combining the above two lower bounds and then using then the minimax lower bound of $L_1$ risk for density estimation over Lipchitz class \citep[Theorem 1]{zhao2020analysis}.
\end{proof}

\subsection{Proofs of the results in Section~\ref{sec:matching}}

We first formally define $\tilde{\tau}_{M,K}^{\rm bc}$. Assume $n$ is divisible by $K$ for simplicity. Let $[I_k]_{k=1}^K$ be a $K$-fold random partition of $\zahl{n}$, with each of size equal to $n' = n/K$. For each $k \in \zahl{K}$ and $\omega \in \{0,1\}$, construct $K_{M,k}^\omega(\cdot), \hat{\mu}_{\omega,k}(\cdot)$ using data $[(X_i,D_i,Y_i)]_{i=1,i \notin I_k}^n$. Then define
\begin{align*}
  & \widecheck{\tau}^{\rm bc}_{M,k} = \frac{1}{n'} \sum_{i=1,i \in I_k}^n \Big[\hat{\mu}_{1,k}(X_i) - \hat{\mu}_{0,k}(X_i)\Big] \\
  & + \frac{1}{n'} \Big[ \sum_{i=1,i \in I_k, D_i = 1}^n \Big(1 + \frac{K^1_{M,k}(X_i)}{M}\Big) \Big(Y_i - \hat{\mu}_{1,k}(X_i)\Big) - \sum_{i=1,i \in I_k, D_i = 0}^n \Big(1 + \frac{K^0_{M,k}(X_i)}{M}\Big) \Big(Y_i - \hat{\mu}_{0,k}(X_i)\Big)\Big]
\end{align*}
and
\begin{align*}
  \tilde{\tau}_{M,K}^{\rm bc} := \frac{1}{K} \sum_{k=1}^K \widecheck{\tau}^{\rm bc}_{M,k}.
\end{align*}


\subsubsection{Proof of Lemma~\ref{lemma:mbc}}

\begin{proof}[Proof of Lemma~\ref{lemma:mbc}]

From Equ.(3) in \cite{abadie2011bias}, the bias-corrected estimator imputes the missing potential outcomes by
\begin{align*}
    \hat{Y}^{\rm bc}_i(0) := \begin{cases}
        Y_i, & \mbox{ if } D_i=0,\\
        \frac{1}{M} \sum_{j \in \mathcal{J}^0_M(i)} (Y_j + \hat{\mu}_0(X_i) - \hat{\mu}_0(X_j)), & \mbox{ if } D_i=1,     
    \end{cases}  
\end{align*}
and
\begin{align*}
    \hat{Y}^{\rm bc}_i(1) := \begin{cases}
      \frac{1}{M} \sum_{j \in \mathcal{J}^1_M(i)} (Y_j + \hat{\mu}_1(X_i) - \hat{\mu}_1(X_j)), & \mbox{ if } D_i=0,\\   
      Y_i, & \mbox{ if } D_i=1.
    \end{cases}    
\end{align*}

The bias-corrected estimator is defined as:
\begin{align*}
  \hat{\tau}^{\rm bc}_M := \frac{1}{n} \sum_{i=1}^n \Big[\hat{Y}^{\rm bc}_i(1) -\hat{Y}^{\rm bc}_i(0)\Big].
\end{align*}

Then by simple algebra, we have
\begin{align*}
  &\hat{\tau}^{\rm bc}_M = \frac{1}{n} \sum_{i=1}^n \Big[\hat{Y}^{\rm bc}_i(1) -\hat{Y}^{\rm bc}_i(0)\Big] \\
  =& \frac{1}{n} \sum_{i=1,D_i = 1}^n \Big[Y_i - \frac{1}{M} \sum_{j \in \mathcal{J}^0_M(i)} (Y_j + \hat{\mu}_0(X_i) - \hat{\mu}_0(X_j)) \Big] + \frac{1}{n} \sum_{i=1,D_i = 0}^n \Big[\frac{1}{M} \sum_{j \in \mathcal{J}^1_M(i)} (Y_j + \hat{\mu}_1(X_i) - \hat{\mu}_1(X_j)) - Y_i \Big]\\
  =& \frac{1}{n} \sum_{i=1,D_i = 1}^n \Big[\hat{R}_i + \hat{\mu}_1(X_i) - \hat{\mu}_0(X_i) - \frac{1}{M} \sum_{j \in \mathcal{J}^0_M(i)} \hat{R}_j \Big] + \frac{1}{n} \sum_{i=1,D_i = 0}^n \Big[\frac{1}{M} \sum_{j \in \mathcal{J}^1_M(i)} \hat{R}_j - \hat{R}_i + \hat{\mu}_1(X_i) - \hat{\mu}_0(X_i) \Big]\\
  =& \frac{1}{n} \sum_{i=1}^n \Big[\hat{\mu}_1(X_i) - \hat{\mu}_0(X_i)\Big] + \frac{1}{n} \Big[ \sum_{i=1,D_i = 1}^n \Big(1 + \frac{K^1_M(i)}{M}\Big) \hat{R}_i - \sum_{i=1,D_i = 0}^n \Big(1 + \frac{K^0_M(i)}{M}\Big) \hat{R}_i \Big].
\end{align*}
This completes the proof.
\end{proof}

\subsubsection{Proof of Theorem~\ref{thm:dml}}

\begin{proof}[Proof of Theorem~\ref{thm:dml}\ref{thm:dml1}]
{\bf Part I.} Suppose the density function is sufficiently smooth.
For any $i \in \zahl{n}$, let $\bar{R}_i := Y_i - \bar{\mu}_{D_i}(X_i)$. From \eqref{eq:mbc},
\begin{align*}
  \hat{\tau}_M^{\rm bc} =& \hat{\tau}^{\rm reg} + \frac{1}{n} \Big[ \sum_{i=1,D_i = 1}^n \Big(1 + \frac{K^1_M(i)}{M}\Big) \hat{R}_i - \sum_{i=1,D_i = 0}^n \Big(1 + \frac{K^0_M(i)}{M}\Big) \hat{R}_i \Big]\\
  = & \hat{\tau}^{\rm reg} + \frac{1}{n} \Big[ \sum_{i=1}^n D_i \Big(1 + \frac{K^1_M(i)}{M}\Big) \hat{R}_i - \sum_{i=1}^n (1-D_i)\Big(1 + \frac{K^0_M(i)}{M}\Big) \hat{R}_i \Big]\\
  = & \frac{1}{n} \sum_{i=1}^n \Big[\hat{\mu}_1(X_i) - \bar{\mu}_1(X_i)\Big] - \frac{1}{n} \sum_{i=1}^n \Big[\hat{\mu}_0(X_i) - \bar{\mu}_0(X_i)\Big]\\
  &+  \frac{1}{n} \Big[ \sum_{i=1}^n D_i \Big(1 + \frac{K^1_M(i)}{M}\Big) \Big(\bar{\mu}_1(X_i) - \hat{\mu}_1(X_i)\Big) - \sum_{i=1}^n (1-D_i)\Big(1 + \frac{K^0_M(i)}{M}\Big) \Big(\bar{\mu}_0(X_i) - \hat{\mu}_0(X_i)\Big) \Big]\\
  &+  \frac{1}{n} \Big[ \sum_{i=1}^n D_i \Big(1 + \frac{K^1_M(i)}{M} - \frac{1}{e(X_i)}\Big) \bar{R}_i - \sum_{i=1}^n (1-D_i)\Big(1 + \frac{K^0_M(i)}{M} - \frac{1}{1-e(X_i)}\Big) \bar{R}_i \Big]\\
  & + \frac{1}{n} \Big[ \sum_{i=1}^n \Big(1 - \frac{D_i}{e(X_i)}\Big) \bar{\mu}_1(X_i) - \sum_{i=1}^n \Big(1 - \frac{1-D_i}{1-e(X_i)}\Big) \bar{\mu}_0(X_i) \Big]\\
  & + \frac{1}{n} \Big[ \sum_{i=1}^n \frac{D_i}{e(X_i)} Y_i - \sum_{i=1}^n \frac{1-D_i}{1-e(X_i)} Y_i \Big].
  \yestag\label{eq:dml1}
\end{align*}

For each term, we only establish the first half part, and the second half can be established in the same way.

For the first term in \eqref{eq:dml1},
\[
  \Big\lvert \frac{1}{n} \sum_{i=1}^n \Big[\hat{\mu}_1(X_i) - \bar{\mu}_1(X_i)\Big] \Big\rvert \le \lVert \hat{\mu}_1 - \bar{\mu}_1 \rVert_\infty = o_\P(1).
\]

Then
\begin{align}\label{eq:dml11}
  \frac{1}{n} \sum_{i=1}^n \Big[\hat{\mu}_1(X_i) - \bar{\mu}_1(X_i)\Big] - \frac{1}{n} \sum_{i=1}^n \Big[\hat{\mu}_0(X_i) - \bar{\mu}_0(X_i)\Big] = o_\P(1).
\end{align}

For the second term in \eqref{eq:dml1},
\begin{align*}
  & \Big\lvert \frac{1}{n} \sum_{i=1}^n D_i \Big(1 + \frac{K^1_M(i)}{M}\Big) \Big(\bar{\mu}_1(X_i) - \hat{\mu}_1(X_i)\Big) \Big\rvert \\
  \le & \lVert \hat{\mu}_1 - \bar{\mu}_1 \rVert_\infty \frac{1}{n} \sum_{i=1}^n D_i \Big(1 + \frac{K^1_M(i)}{M}\Big) = \lVert \hat{\mu}_1 - \bar{\mu}_1 \rVert_\infty = o_\P(1),  
\end{align*}
where the last step is due to $\sum_{i=1}^n D_i K^1_M(i) = n_0 M$. We then have
\begin{align}\label{eq:dml12}
  \frac{1}{n} \Big[ \sum_{i=1}^n D_i \Big(1 + \frac{K^1_M(i)}{M}\Big) \Big(\bar{\mu}_1(X_i) - \hat{\mu}_1(X_i)\Big) - \sum_{i=1}^n (1-D_i)\Big(1 + \frac{K^0_M(i)}{M}\Big) \Big(\bar{\mu}_0(X_i) - \hat{\mu}_0(X_i)\Big) \Big] = o_\P(1).
\end{align}

For the third term in \eqref{eq:dml1}, by Theorem~\ref{thm:risk,lp},
\begin{align*}
  & \E \Big[\Big\lvert\frac{1}{n} \sum_{i=1}^n D_i \Big(1 + \frac{K^1_M(i)}{M} - \frac{1}{e(X_i)}\Big) \bar{R}_i\Big\rvert\Big] \le \E \Big[ \Big\lvert D_i \Big(1 + \frac{K^1_M(i)}{M} - \frac{1}{e(X_i)}\Big) \bar{R}_i \Big\rvert \Big]\\
  \le & \E \Big[1 + \frac{K^1_M(i)}{M} - \frac{1}{e(X_i)}\Big]^2 \E [D_i \bar{R}_i]^2 =  \E \Big[1 + \frac{K^1_M(i)}{M} - \frac{1}{e(X_i)}\Big]^2 \E [D_i (Y_i(1)-\bar{\mu}_1(X_i))^2]\\
  \le & \E \Big[1 + \frac{K^1_M(i)}{M} - \frac{1}{e(X_i)}\Big]^2 \E \Big[\sigma_1^2(X_i)+(\mu_1(X_i)-\bar{\mu}_1(X_i))^2\Big] = o(1).
\end{align*}
We then obtain
\begin{align}\label{eq:dml13}
  \frac{1}{n} \Big[ \sum_{i=1}^n D_i \Big(1 + \frac{K^1_M(i)}{M} - \frac{1}{e(X_i)}\Big) \bar{R}_i - \sum_{i=1}^n (1-D_i)\Big(1 + \frac{K^0_M(i)}{M} - \frac{1}{1-e(X_i)}\Big) \bar{R}_i \Big] = o_\P(1).
\end{align}

For the fourth term in \eqref{eq:dml1}, notice that
\[
  \E \Big[\frac{1}{n} \sum_{i=1}^n \Big(1 - \frac{D_i}{e(X_i)}\Big) \bar{\mu}_1(X_i) \Biggiven \mX \Big] = 0.
\]
Then
\begin{align*}
  & \Var\Big[\frac{1}{n} \sum_{i=1}^n \Big(1 - \frac{D_i}{e(X_i)}\Big) \bar{\mu}_1(X_i)\Big] = \E\Big[\Var \Big[\frac{1}{n} \sum_{i=1}^n \Big(1 - \frac{D_i}{e(X_i)}\Big) \bar{\mu}_1(X_i) \Biggiven \mX \Big]\Big]\\
  = & \frac{1}{n} \E \Big[\bar{\mu}^2_1(X_i) \Big(\frac{1}{e(X_i)} -1\Big)\Big] = O(n^{-1}).
\end{align*}
Then
\begin{align}\label{eq:dml14}
  \frac{1}{n} \Big[ \sum_{i=1}^n \Big(1 - \frac{D_i}{e(X_i)}\Big) \bar{\mu}_1(X_i) - \sum_{i=1}^n \Big(1 - \frac{1-D_i}{1-e(X_i)}\Big) \bar{\mu}_0(X_i) \Big] = o_\P(1).
\end{align}

For the fifth term in \eqref{eq:dml1}, notice that $\E [Y^2]$ are bounded and $[(X_i,D_i,Y_i)]_{i=1}^n$ are i.i.d.. Using the weak law of large numbers \cite[Theorem 2.2.3]{MR3930614} yields
\begin{align}\label{eq:dml15}
  \frac{1}{n} \Big[ \sum_{i=1}^n \frac{D_i}{e(X_i)} Y_i - \sum_{i=1}^n \frac{1-D_i}{1-e(X_i)} Y_i \Big] \stackrel{\sf p}{\longrightarrow} \E\Big[Y_i(1)-Y_i(0)\Big] = \tau.
\end{align}

Plugging \eqref{eq:dml11}, \eqref{eq:dml12}, \eqref{eq:dml13}, \eqref{eq:dml14}, \eqref{eq:dml15} into \eqref{eq:dml1} completes the proof.

{\bf Part II.} Suppose the outcome model is correct.
By \eqref{eq:mbc},
\begin{align*}
  \hat{\tau}_M^{\rm bc} =& \hat{\tau}^{\rm reg} + \frac{1}{n} \Big[ \sum_{i=1,D_i = 1}^n \Big(1 + \frac{K^1_M(i)}{M}\Big) \hat{R}_i - \sum_{i=1,D_i = 0}^n \Big(1 + \frac{K^0_M(i)}{M}\Big) \hat{R}_i \Big]\\
  = & \hat{\tau}^{\rm reg} + \frac{1}{n} \Big[ \sum_{i=1}^n D_i \Big(1 + \frac{K^1_M(i)}{M}\Big) \hat{R}_i - \sum_{i=1}^n (1-D_i)\Big(1 + \frac{K^0_M(i)}{M}\Big) \hat{R}_i \Big]\\
  = & \frac{1}{n} \sum_{i=1}^n \Big[\hat{\mu}_1(X_i) - \mu_1(X_i)\Big] - \frac{1}{n} \sum_{i=1}^n \Big[\hat{\mu}_0(X_i) - \mu_0(X_i)\Big]\\
  &+  \frac{1}{n} \Big[ \sum_{i=1}^n D_i \Big(1 + \frac{K^1_M(i)}{M}\Big) \Big(\mu_1(X_i) - \hat{\mu}_1(X_i)\Big) - \sum_{i=1}^n (1-D_i)\Big(1 + \frac{K^0_M(i)}{M}\Big) \Big(\mu_0(X_i) - \hat{\mu}_0(X_i)\Big) \Big]\\
  &+  \frac{1}{n} \Big[ \sum_{i=1}^n D_i \Big(1 + \frac{K^1_M(i)}{M}\Big) \Big(Y_i - \mu_1(X_i) \Big) - \sum_{i=1}^n (1-D_i)\Big(1 + \frac{K^0_M(i)}{M}\Big) \Big(Y_i - \mu_0(X_i) \Big) \Big]\\
  &+  \frac{1}{n} \sum_{i=1}^n \Big[\mu_1(X_i) - \mu_0(X_i)\Big].
  \yestag\label{eq:dml2}
\end{align*}

For the first term in \eqref{eq:dml2},
\[
  \Big\lvert \frac{1}{n} \sum_{i=1}^n \Big[\hat{\mu}_1(X_i) - \mu_1(X_i)\Big] \Big\rvert \le \lVert \hat{\mu}_1 - \mu_1 \rVert_\infty = o_\P(1).
\]
Then
\begin{align}\label{eq:dml21}
  \frac{1}{n} \sum_{i=1}^n \Big[\hat{\mu}_1(X_i) - \mu_1(X_i)\Big] - \frac{1}{n} \sum_{i=1}^n \Big[\hat{\mu}_0(X_i) - \mu_0(X_i)\Big] = o_\P(1).
\end{align}

For the second term in \eqref{eq:dml2},
\begin{align*}
  & \Big\lvert \frac{1}{n} \sum_{i=1}^n D_i \Big(1 + \frac{K^1_M(i)}{M}\Big) \Big(\mu_1(X_i) - \hat{\mu}_1(X_i)\Big) \Big\rvert \\
  \le & \lVert \hat{\mu}_1 - \mu_1 \rVert_\infty \frac{1}{n} \sum_{i=1}^n D_i \Big(1 + \frac{K^1_M(i)}{M}\Big) = \lVert \hat{\mu}_1 - \mu_1 \rVert_\infty = o_\P(1),  
\end{align*}
where the last step is due to $\sum_{i=1}^n D_i K^1_M(i) = n_0 M$. We then have
\begin{align}\label{eq:dml22}
  \frac{1}{n} \Big[ \sum_{i=1}^n D_i \Big(1 + \frac{K^1_M(i)}{M}\Big) \Big(\mu_1(X_i) - \hat{\mu}_1(X_i)\Big) - \sum_{i=1}^n (1-D_i)\Big(1 + \frac{K^0_M(i)}{M}\Big) \Big(\mu_0(X_i) - \hat{\mu}_0(X_i)\Big) \Big] = o_\P(1).
\end{align}

For the third term in \eqref{eq:dml2}, noticing that $K_M^1(\cdot)$ is a function of $\mX$ and $\mD$, one could obtain
\begin{align*}
  \E \Big[ \frac{1}{n} \sum_{i=1}^n D_i \Big(1 + \frac{K^1_M(i)}{M}\Big) \Big(Y_i - \mu_1(X_i) \Big) \Biggiven \mX, \mD \Big] = 0.
\end{align*}
Accordingly, we have
\begin{align*}
  & \Var\Big[\frac{1}{n} \sum_{i=1}^n D_i \Big(1 + \frac{K^1_M(i)}{M}\Big) \Big(Y_i - \mu_1(X_i) \Big) \Big]\\
  = & \E \Big[\Var\Big[\frac{1}{n} \sum_{i=1}^n D_i \Big(1 + \frac{K^1_M(i)}{M}\Big) \Big(Y_i - \mu_1(X_i) \Big) \Big] \Biggiven \mX,\mD \Big]\\
  = & \frac{1}{n} \E \Big[ D_i \Big(1 + \frac{K^1_M(i)}{M}\Big)^2 \sigma_1^2(X_i) \Big] \le \frac{1}{n} \E \Big[ D_i \Big(1 + \frac{K^1_M(i)}{M}\Big)^2 \Big] \lVert \sigma_1^2 \rVert_\infty,
\end{align*}
where $\sigma_1^2(x) = \E [U^2_1 \given X=x]$ for $x \in \bX$.

Conditional on $\mD$, for any $j,\ell \in \zahl{n}$ and $j \neq \ell$ such that $D_j = D_\ell = 0$, for $i \in \zahl{n}$ such that $D_i = 1$, it is easy to check $\E [K_M^1(i)\given \mD] = Mn_0/n_1$ and $\E [(K_M^1(i))^2 \given \mD] \asymp M^2 $. Then
\[
  \Var\Big[\frac{1}{n} \sum_{i=1}^n D_i \Big(1 + \frac{K^1_M(i)}{M}\Big) \Big(Y_i - \mu_1(X_i) \Big) \Big] = O(n^{-1}).
\]
We then obtain
\begin{align}\label{eq:dml23}
  \frac{1}{n} \Big[ \sum_{i=1}^n D_i \Big(1 + \frac{K^1_M(i)}{M}\Big) \Big(Y_i - \mu_1(X_i) \Big) - \sum_{i=1}^n (1-D_i)\Big(1 + \frac{K^0_M(i)}{M}\Big) \Big(Y_i - \mu_0(X_i) \Big) \Big] = o_\P(1).
\end{align}

For the fourth term in \eqref{eq:dml2}, notice that $\E[\mu_\omega^2]$ is bounded for $\omega \in \{0,1\}$. Using the weak law of large number, we obtain
\begin{align}\label{eq:dml24}
  \frac{1}{n} \sum_{i=1}^n \Big[\mu_1(X_i) - \mu_0(X_i)\Big] \stackrel{\sf p}{\longrightarrow} \E \Big[\mu_1(X_i) - \mu_0(X_i)\Big] = \tau.
\end{align}

Plugging \eqref{eq:dml21}, \eqref{eq:dml22}, \eqref{eq:dml23}, \eqref{eq:dml24} into \eqref{eq:dml2} completes the proof.

Analogous analysis can be performed on $\widecheck{\tau}^{\rm bc}_{M,k}$ for any $k \in \zahl{K}$. Then the results apply to $\tilde{\tau}_{M,K}^{\rm bc}$ automatically since $K$ is fixed.
\end{proof}

\begin{proof}[Proof of Theorem~\ref{thm:dml}\ref{thm:dml2}]

From Definition 3.1 in \cite{chernozhukov2018double}, $\tilde{\tau}_{M,K}^{\rm bc}$ is the double machine learning estimator. We then follow the proof of Theorem 5.1 (recalling Remark \ref{remark:dml}) and use the notations in \cite{chernozhukov2018double}, essentially checking Assumption 3.1 and 3.2 therein. In the following the notation in  \cite{chernozhukov2018double} is adopted.

For estimating the ATE, from Equ. (5.3) in \cite{chernozhukov2018double}, the score is
\[
  \psi(X,D,Y;\tilde{\tau},\tilde{\zeta}) := \tilde{\mu}_1(X) - \tilde{\mu}_0(X) + \frac{D(Y-\tilde{\mu}_1(X))}{\tilde{e}(X)} - \frac{(1-D)(Y-\tilde{\mu}_0(X))}{1-\tilde{e}(X)} - \tilde{\tau},
\]
where $\tilde{\zeta}(x) = (\tilde{\mu}_0(x),\tilde{\mu}_1(x),\tilde{\rho}_0(x),\tilde{\rho}_1(x))$ are the nuisance parameters by letting $\tilde{\rho}_0(x) = 1/(1-\tilde{e}(x))$ and $\tilde{\rho}_1(x) = 1/\tilde{e}(x)$. Let $\rho_0(x) = 1/(1-e(x))$ and $\rho_1(x) = 1/e(x)$. Then the true value is $\zeta(x) = (\mu_0(x),\mu_1(x),\rho_0(x),\rho_1(x))$.

We can then write the score as
\[
  \psi(X,D,Y;\tilde{\tau},\tilde{\zeta}) = \tilde{\mu}_1(X) - \tilde{\mu}_0(X) + D(Y-\tilde{\mu}_1(X))\tilde{\rho}_1(X) - (1-D)(Y-\tilde{\mu}_0(X))\tilde{\rho}_0(X) - \tilde{\tau}.
\]

For any $p>0$, let $\lVert f \rVert_p := \lVert f(X,D,Y) \rVert_p = (\int \lvert f(\omega) \rvert^p \d \P_{(X,D,Y)}(\omega))^{1/p}$. For the $\kappa$ in Assumption~\ref{asp:dml1}, let $q = 2+ \kappa/2$, $q_1 = 2 + \kappa$ and $q_2$ such that $q^{-1} = q_1^{-1} + q_2^{-1}$. Let $\cT_n$ be the set consisting of all $\tilde{\zeta}$ such that for $\omega \in \{0,1\}$,
\[
  \lVert \tilde{\mu}_\omega - \mu_\omega \rVert_\infty = o(n^{-d/(4+2d)}), ~~ \lVert \tilde{\rho}_\omega - \rho_\omega \rVert_1 = O(n^{-1/(d+2)}), ~~ \lVert \tilde{\rho}_\omega - \rho_\omega \rVert_{q_2} = o(1).
\]

Then the selection of $\cT_n$ satisfies Assumption 3.2(a) in \cite{chernozhukov2018double} from Assumption~\ref{asp:dml3}, Theorem~\ref{thm:rate,risk}, and Theorem~\ref{thm:risk,lp}, respectively.

For step 1, we verify the Neyman orthogonality. It is easy to see $\E \psi(X,D,Y;\tau,\zeta) = 0$. For any $\tilde{\zeta} \in \cT_n$, the Gateaux derivative in the direction $\tilde{\zeta} - \zeta$ is
\begin{align*}
  & \partial_{\tilde{\zeta}} \E \psi(X,D,Y;\tau,\zeta) [\tilde{\zeta} - \zeta] = \E [\tilde{\mu}_1(X) - \mu_1(X)] - \E [\tilde{\mu}_0(X) - \mu_0(X)]\\
  & - \E [D(\tilde{\mu}_1(X) - \mu_1(X))\rho_1(X )] + \E [(1-D)(\tilde{\mu}_0(X) - \mu_0(X))\rho_0(X)] \\
  & + \E [D(Y-\mu_1(X))(\tilde{\rho}_1(X) - \rho_1(X))] - \E [(1-D)(Y-\mu_0(X))(\tilde{\rho}_0(X)-\rho_0(X))].
\end{align*}
It is easy to check the above quantity is zero, which completes this step.

Step 2 and step 3 can be directly applied.

For step 4, we can establish in the same way that for $\omega \in \{0,1\}$, $\lVert \mu_\omega \rVert_{q_1} = O(1)$ from $\lVert Y \rVert_{q_1} = O(1)$, and $\tau = O(1)$. Then from Hölder's inequality and $\lVert \rho_\omega \rVert_\infty$ is bounded for $\omega \in \{0,1\}$,
\begin{align*}
  & \lVert \psi(X,D,Y;\tau,\tilde{\zeta}) \lVert_q = \lVert \tilde{\mu}_1(X) - \tilde{\mu}_0(X) + D(Y-\tilde{\mu}_1(X))\tilde{\rho}_1(X) - (1-D)(Y-\tilde{\mu}_0(X))\tilde{\rho}_0(X) - \tau \rVert_q\\
  \le & \lVert \tilde{\mu}_1(X) \rVert_q + \lVert \tilde{\mu}_0(X) \rVert_q + \lVert (Y-\tilde{\mu}_1(X))\tilde{\rho}_1(X) \rVert_q + \lVert (Y-\tilde{\mu}_0(X))\tilde{\rho}_0(X) \rVert_q + \tau\\
  \le & \lVert \mu_1 \rVert_q + \lVert \tilde{\mu}_1 - \mu_1 \rVert_\infty + \lVert \mu_0 \rVert_q + \lVert \tilde{\mu}_0 - \mu_0 \rVert_\infty + (\lVert Y \rVert_{q_1} + \lVert \mu_1 \rVert_{q_1} + \lVert \tilde{\mu}_1 - \mu_1 \rVert_\infty) \lVert \tilde{\rho}_1 \rVert_{q_2}\\
  & + (\lVert Y \rVert_{q_1} + \lVert \mu_0 \rVert_{q_1} + \lVert \tilde{\mu}_0 - \mu_0 \rVert_\infty) \lVert \tilde{\rho}_0 \rVert_{q_2} + \tau = O(1).
\end{align*}
The last step is from the definition of $\cT_n$. Then we completes this step.

For step 5, from Hölder's inequality,
\begin{align*}
  & \lVert \psi(X,D,Y;\tau,\tilde{\zeta}) - \psi(X,D,Y;\tau,\zeta) \rVert_2\\
  \le & \lVert \tilde{\mu}_1 - \mu_1 \rVert_2 + \lVert \tilde{\mu}_0- \mu_0 \rVert_2 + \lVert D(Y-\tilde{\mu}_1(X))\tilde{\rho}_1(X) - D(Y-\mu_1(X))\rho_1(X) \rVert_2 \\
  & + \lVert (1-D)(Y-\tilde{\mu}_0(X))\tilde{\rho}_0(X) - (1-D)(Y-\mu_0(X))\rho_0(X) \rVert_2\\
  \le & \lVert \tilde{\mu}_1 - \mu_1 \rVert_2 + \lVert \tilde{\mu}_0- \mu_0 \rVert_2 + \lVert (Y-\mu_1(X)) (\tilde{\rho}_1 - \rho_1) \rVert
  _2 + \lVert (\tilde{\mu}_1 - \mu_1) \tilde{\rho}_1 \rVert_2 \\
  & + \lVert (Y-\mu_0(X)) (\tilde{\rho}_0 - \rho_0) \rVert
  _2 + \lVert (\tilde{\mu}_0 - \mu_0) \tilde{\rho}_0 \rVert_2\\
  \le & \lVert \tilde{\mu}_1 - \mu_1 \rVert_\infty + \lVert \tilde{\mu}_0- \mu_0 \rVert_\infty + (\lVert Y \rVert_{q_1} + \lVert \mu_1(X) \rVert_{q_1}) \lVert \tilde{\rho}_1 - \rho_1 \rVert
  _{q_2} + \lVert \tilde{\mu}_1 - \mu_1 \rVert_\infty \lVert \tilde{\rho}_1 \rVert_2 \\
  & + (\lVert Y \rVert_{q_1} + \lVert \mu_0(X) \rVert_{q_1}) \lVert \tilde{\rho}_0 - \rho_0 \rVert
  _{q_2} + \lVert \tilde{\mu}_0 - \mu_0 \rVert_\infty \lVert \tilde{\rho}_0 \rVert_2 = o(1).
\end{align*}
The last step is due to the definition of $\cT_n$.

Notice that for any $t \in (0,1)$,
\begin{align*}
  &\partial_t^2 \E \psi(X,D,Y;\tau,\zeta + t (\tilde{\zeta} - \zeta)) \\
  = & -2\Big(\E [D(\tilde{\mu}_1(X) - \mu_1(X))(\tilde{\rho}_1(X)-\rho_1(X))] - \E [(1-D)(\tilde{\mu}_0(X) - \mu_0(X))(\tilde{\rho}_0(X)-\rho_0(X))] \Big).
\end{align*}
Then from the definition of $\cT_n$,
\[
  \lvert \partial_t^2 \E \psi(X,D,Y;\tau,\zeta + t (\tilde{\zeta} - \zeta)) \rvert \le 2 [\lVert \tilde{\mu}_1 - \mu_1 \rVert_\infty \lVert \tilde{\rho}_1 - \rho_1 \rVert_1 + \lVert \tilde{\mu}_0 - \mu_0 \rVert_\infty \lVert \tilde{\rho}_0 - \rho_0 \rVert_1] = o(n^{-1/2}).
\]
We then complete this step and thus complete the proof.

\end{proof}

\begin{proof}[Proof of Theorem~\ref{thm:dml}\ref{thm:dml2}, consistency of $\hat{\sigma}^2$]

We decompose $\hat{\sigma}^2$ as
\begin{align*}
  &\hat{\sigma}^2 - \sigma^2 = \frac{1}{n} \sum_{i=1}^n \Big[\hat{\mu}_1(X_i) - \hat{\mu}_0(X_i) + D_i\Big(1 + \frac{K^1_M(i)}{M}\Big)\hat R_i -
   (1-D_i)\Big(1 + \frac{K^0_M(i)}{M}\Big)\hat R_i - \hat{\tau}_M^{\rm bc} \Big]^2\\
   = & \Big\{\frac{1}{n} \sum_{i=1}^n \Big[\hat{\mu}_1(X_i) - \hat{\mu}_0(X_i) + D_i\Big(1 + \frac{K^1_M(i)}{M}\Big)\hat R_i -
   (1-D_i)\Big(1 + \frac{K^0_M(i)}{M}\Big)\hat R_i - \hat{\tau}_M^{\rm bc} \Big]^2\\
   & - \frac{1}{n} \sum_{i=1}^n \Big[\mu_1(X_i) - \mu_0(X_i) + D_i\Big(1 + \frac{K^1_M(i)}{M}\Big)(Y_i - \mu_1(X_i)) -
   (1-D_i)\Big(1 + \frac{K^0_M(i)}{M}\Big)(Y_i - \mu_0(X_i)) - \hat{\tau}_M^{\rm bc} \Big]^2 \Big\}\\
   +& \Big\{ \frac{1}{n} \sum_{i=1}^n \Big[\mu_1(X_i) - \mu_0(X_i) + D_i\Big(1 + \frac{K^1_M(i)}{M}\Big)(Y_i - \mu_1(X_i)) -
   (1-D_i)\Big(1 + \frac{K^0_M(i)}{M}\Big)(Y_i - \mu_0(X_i)) - \hat{\tau}_M^{\rm bc} \Big]^2 \\
   &- \frac{1}{n} \sum_{i=1}^n \Big[\mu_1(X_i) - \mu_0(X_i) + \frac{D_i}{e(X_i)}(Y_i - \mu_1(X_i)) -
   \frac{1-D_i}{1-e(X_i)}(Y_i - \mu_0(X_i)) - \hat{\tau}_M^{\rm bc} \Big]^2 \Big\}\\
   +& \Big\{ \frac{1}{n} \sum_{i=1}^n \Big[\mu_1(X_i) - \mu_0(X_i) + \frac{D_i}{e(X_i)}(Y_i - \mu_1(X_i)) -
   \frac{1-D_i}{1-e(X_i)}(Y_i - \mu_0(X_i)) - \hat{\tau}_M^{\rm bc} \Big]^2 \\
   &- \frac{1}{n} \sum_{i=1}^n \Big[\mu_1(X_i) - \mu_0(X_i) + \frac{D_i}{e(X_i)}(Y_i - \mu_1(X_i)) -
   \frac{1-D_i}{1-e(X_i)}(Y_i - \mu_0(X_i)) - \tau \Big]^2 \Big\}\\
   +& \Big\{ \frac{1}{n} \sum_{i=1}^n \Big[\mu_1(X_i) - \mu_0(X_i) + \frac{D_i}{e(X_i)}(Y_i - \mu_1(X_i)) -
   \frac{1-D_i}{1-e(X_i)}(Y_i - \mu_0(X_i)) - \tau \Big]^2 - \sigma^2 \Big\}\\
   =:& T_1 + T_2 + T_3 + T_4. 
\end{align*}

By Assumption~\ref{asp:dml3'}, Assumption~\ref{asp:dml1}, Theorem~\ref{thm:risk,lp}, as well as  the fact that $\hat{\tau}_M^{\rm bc} = O_\P(1)$, we have $T_1 = o_\P(1)$. By Assumption~\ref{asp:dml1}, Theorem~\ref{thm:risk,lp}, and $\hat{\tau}_M^{\rm bc} = O_\P(1)$, we have $T_2 = o_\P(1)$. By Assumption~\ref{asp:dml1} and $\hat{\tau}_M^{\rm bc} - \tau = o_\P(1)$, we have $T_3 = o_\P(1)$. By the fact that $[(X_i,D_i,Y_i)]_{i=1}^n$ are i.i.d., Assumption~\ref{asp:dml1} and the weak law of large numbers, we have $T_4 = o_\P(1)$. Combining the above four facts together then completes the proof.
\end{proof}

\subsubsection{Proof of Theorem~\ref{thm:mbc}}

\begin{proof}[Proof of Theorem~\ref{thm:mbc}]

For $\omega\in\{0,1\}$, let $j^\omega_M(i)$ represent the index of $M$th-NN of $X_i$ in $\{X_j:D_j=\omega\}_{j=1}^n$, i.e., the index $j \in \zahl{n}$ such that $D_j=\omega$ and 
\[
    \sum_{\ell=1, D_\ell=\omega}^n \ind\Big(\lVert X_\ell -X_i \rVert \le \lVert X_j - X_i \rVert\Big) = M.
\]

With a little abuse of notation, let $\epsilon_i = Y_i - \mu_{D_i}(X_i)$ for any $i \in \zahl{n}$. Checking the definition of $\hat{\tau}_M^{\rm bc}$ in \eqref{eq:mbc}, we decompose $\hat{\tau}_M^{\rm bc}$ to be
\begin{align*}
  \hat{\tau}_M^{\rm bc} =& \hat{\tau}^{\rm reg} + \frac{1}{n} \Big[ \sum_{i=1,D_i = 1}^n \Big(1 + \frac{K^1_M(i)}{M}\Big) \hat{R}_i - \sum_{i=1,D_i = 0}^n \Big(1 + \frac{K^0_M(i)}{M}\Big) \hat{R}_i \Big]\\
  = & \frac{1}{n} \sum_{i=1}^n \Big[\hat{\mu}_1(X_i) - \hat{\mu}_0(X_i)\Big] + \frac{1}{n} \sum_{i=1}^n (2D_i-1) \Big(1 + \frac{K^{D_i}_M(i)}{M}\Big) \Big(Y_i - \hat{\mu}_{D_i}(X_i)\Big)\\
  = & \frac{1}{n} \sum_{i=1}^n \Big[\mu_1(X_i) - \mu_0(X_i)\Big] + \frac{1}{n} \sum_{i=1}^n (2D_i-1) \Big(1 + \frac{K^{D_i}_M(i)}{M}\Big) \epsilon_i \\
  & + \frac{1}{n}\sum_{i=1}^n (2D_i-1) \Big[\frac{1}{M}\sum_{m=1}^M \Big(\mu_{1-D_i}(X_i)-\mu_{1-D_i}(X_{j^{1-D_i}_m(i)})\Big)\Big]\\
  & - \frac{1}{n}\sum_{i=1}^n (2D_i-1) \Big[\frac{1}{M}\sum_{m=1}^M \Big(\hat{\mu}_{1-D_i}(X_i)-\hat{\mu}_{1-D_i}(X_{j^{1-D_i}_m(i)})\Big)\Big]\\
  :=& \bar{\tau}(\mX) + E_M + B_M - \hat{B}_M.
\end{align*}

We have the following central limit theorem on $\bar{\tau}(\mX) + E_M$.

\begin{lemma}\label{lemma:mbc,clt}
  \begin{align*}
    \sqrt{n} \sigma^{-1/2} \Big(\bar{\tau}(\mX) + E_M - \tau \Big)\stackrel{\sf d}{\longrightarrow} N\Big(0,1\Big).
  \end{align*}
\end{lemma}

For the bias term $B_M-\hat B_M$, define
\[
U_{m,i}:= X_{j^{1-D_i}_m(i)} - X_i ~~\text{ for any }i \in \zahl{n} \text{ and } m \in \zahl{M}.
\]
We then have the following lemma bounding the moments of $U_{M,i}$.

\begin{lemma}\label{lemma:mbc,dist}
  Let $p$ be any positive integer. Then there exists a constant $C_p>0$ only depending on $p$ such that for any $i \in \zahl{n}$,
  \begin{align*}
    \E \Big[\Big\lVert U_{M,i} \Big\rVert^p \Biggiven \mD\Big] \le C_p \Big(\frac{M}{n_{1-D_i}}\Big)^{p/d}.
  \end{align*}
\end{lemma}

In light of the smoothness conditions on $\mu_\omega$ and approximation conditions on $\hat{\mu}_\omega$ for $\omega \in \{0,1\}$, we can establish the following lemma using Lemma~\ref{lemma:mbc,dist}.

\begin{lemma}\label{lemma:mbc,bias}
  \begin{align*}
    \sqrt{n} \Big(B_M - \hat{B}_M \Big)\stackrel{\sf p}{\longrightarrow} 0.
  \end{align*}
\end{lemma}

Combining Lemma~\ref{lemma:mbc,clt} and Lemma~\ref{lemma:mbc,bias} completes the proof. The proof of the consistency of the variance estimator is the same as Theorem~\ref{thm:dml}\ref{thm:dml2} since only Assumption~\ref{asp:dml1} and Assumption~\ref{asp:dml3'} are needed.
\end{proof}

\subsubsection{Proof of Lemma~\ref{lemma:mbc,clt}}

\begin{proof}[Proof of Lemma~\ref{lemma:mbc,clt}]
For any $x \in \bX$, define $\sigma_\omega^2(x) = \E [[Y(\omega) - \mu_\omega(X)]^2 \given X=x ]$ for $\omega \in \{0,1\}$. Let
\[
  V^\tau := \E \Big[ \mu_1(X) - \mu_0(X) - \tau \Big]^2~~~{\rm and}~~~ V^E := \frac{1}{n} \sum_{i=1}^n \Big(1+\frac{K^{D_i}_M(i)}{M}\Big)^2 \sigma_{D_i}^2(X_i).
\]

From the standard central limit theorem, we have
\begin{align}\label{eq:mbc,clt1}
  \sqrt{n} \Big(\bar{\tau}(\mX) - \tau\Big) \stackrel{\sf d}{\longrightarrow} N\Big(0,V^\tau\Big).
\end{align}

Let $E_{M,i} = (2D_i-1) (1 + K^{D_i}_M(i)/M) \epsilon_i$ for any $i \in \zahl{n}$. Conditional on $\mX,\mD$, $[E_{M,i}]_{i=1}^n$ are independent. Notice that $\E [E_{M,i} \given \mX, \mD] = 0$ and $\sum_{i=1}^n \Var[E_{M,i} \given \mX, \mD] = n V^E$. To apply the Lindeberg-Feller central limit theorem \cite[Theorem 27.2]{MR1324786}, it suffices to verify that: for a given $(\mX, \mD)$,
\[
    \frac{1}{nV^E} \sum_{i=1}^n \E\Big[\Big(E_{M,i}\Big)^2 \ind\Big(\lvert E_{M,i} \rvert > \delta \sqrt{nV^E}\Big) \Biggiven \mX, \mD \Big] \to 0,
\]
for all $\delta>0$.

Let $C_\sigma := \sup_{x \in \bX, \omega \in \{0,1\}} \{\E [\lvert U_\omega \rvert^{2+\kappa} \given X=x] \vee \E [U^2_\omega \given X=x]\} < \infty$. Let $p_1 := 1+\kappa/2$ and $p_2$ be the constant such that $p_1^{-1} + p_2^{-1} = 1$. From Hölder's inequality and Markov's inequality,
\begin{align*}
  & \frac{1}{nV^E} \sum_{i=1}^n \E\Big[\Big(E_{M,i}\Big)^2 \ind\Big(\lvert E_{M,i} \rvert > \delta \sqrt{nV^E}\Big) \Biggiven \mX, \mD \Big]\\
  \le& \frac{1}{nV^E} \sum_{i=1}^n \Big(\E\Big[\Big\lvert E_{M,i}\Big\rvert ^{2+\kappa} \Biggiven \mX, \mD \Big] \Big)^{1/p_1} \Big(\P\Big(\lvert E_{M,i} \rvert > \delta \sqrt{nV^E} \Biggiven \mX, \mD \Big)\Big)^{1/p_2}\\
  \le & \frac{1}{nV^E} \sum_{i=1}^n \Big(\E\Big[\Big\lvert E_{M,i}\Big\rvert ^{2+\kappa} \Biggiven \mX, \mD \Big] \Big)^{1/p_1} \Big(\frac{1}{\delta^2 nV^E} \E\Big[\Big(E_{M,i}\Big)^2 \Biggiven \mX, \mD \Big]  \Big)^{1/p_2}\\
  \le & \frac{C_\sigma}{nV^E} \Big(\frac{1}{\delta^2 nV^E} \Big)^{1/p_2}  \sum_{i=1}^n \Big(1 + \frac{K^{D_i}_M(i)}{M}\Big)^{2(1+1/p_2)}.
\end{align*}
Notice that $\E [1 + K^{D_i}_M(i)/M]^{2(1+1/p_2)} < \infty$ from Theorem~\ref{thm:risk,lp}. Let $c_\sigma = \inf_{x \in \bX, \omega \in \{0,1\}} \E [U^2_\omega \given X=x] >0$. From the definition of $V^E$, we have $V^E \ge c_\sigma$ for almost all $\mX, \mD$. Then
\[
  \E \Big[\frac{1}{nV^E} \sum_{i=1}^n \E\Big[\Big(E_{M,i}\Big)^2 \ind\Big(\lvert E_{M,i} \rvert > \delta \sqrt{nV^E}\Big) \Biggiven \mX, \mD \Big] \Big] = O(n^{-1/p_2}) = o(1).
\]
We thus obtain
\[
  \frac{1}{nV^E} \sum_{i=1}^n \E\Big[\Big(E_{M,i}\Big)^2 \ind\Big(\lvert E_{M,i} \rvert > \delta \sqrt{nV^E}\Big) \Biggiven \mX, \mD \Big] = o_\P(1).
\]
Employing the Lindeberg-Feller central limit theorem then yields
\begin{align}\label{eq:mbc,clt2}
  \sqrt{n} (V^E)^{-1/2} E_M = \Big(nV^E\Big)^{-1/2} \sum_{i=1}^n E_{M,i} \stackrel{\sf d}{\longrightarrow} N\Big(0,1\Big).
\end{align}

Noticing that $\sqrt{n} \Big(\bar{\tau}(\mX) - \tau\Big)$ and $\sqrt{n} (V^E)^{-1/2} E_M$ are asymptotically independent, leveraging the same argument as made in \citet[Proof of Theorem 4, Page 267]{abadie2006large} and then combining \eqref{eq:mbc,clt1} and \eqref{eq:mbc,clt2} reaches
\begin{align}\label{eq:mbc,clt3}
  \sqrt{n} \Big(V^\tau + V^E\Big)^{-1/2} \Big(\bar{\tau}(\mX) + E_M - \tau \Big)\stackrel{\sf d}{\longrightarrow} N\Big(0,1\Big).
\end{align}

We decompose $V^E$ as:
\begin{align*}
   V^E =& \frac{1}{n} \sum_{i=1, D_i = 1}^n \Big(1+\frac{K^1_M(i)}{M}\Big)^2 \sigma_1^2(X_i) + \frac{1}{n} \sum_{i=1, D_i = 0}^n \Big(1+\frac{K^0_M(i)}{M}\Big)^2 \sigma_0^2(X_i)\\
  =& \Big[\frac{1}{n} \sum_{i=1, D_i = 1}^n \Big(\frac{1}{e(X_i)}\Big)^2 \sigma_1^2(X_i) + \frac{1}{n} \sum_{i=1, D_i = 0}^n \Big(\frac{1}{1-e(X_i)}\Big)^2 \sigma_0^2(X_i) \Big]\\
  &+ \frac{1}{n} \sum_{i=1, D_i = 1}^n \Big[\Big(1+\frac{K^1_M(i)}{M}\Big)^2 - \Big(\frac{1}{e(X_i)}\Big)^2\Big] \sigma_1^2(X_i) \\
  &+ \frac{1}{n} \sum_{i=1, D_i = 0}^n \Big[\Big(\frac{1}{1-e(X_i)}\Big)^2 - \Big(1+\frac{K^0_M(i)}{M}\Big)^2 \Big] \sigma_0^2(X_i).
  \yestag\label{eq:mbc,clt4}
\end{align*}

For the first term in \eqref{eq:mbc,clt4}, notice that $[(X_i,D_i,Y_i)]_{i=1}^n$ are i.i.d. and $\E[D_i (e(X_i))^{-2} \sigma_1^2(X_i)], \E[(1-D_i) (1-e(X_i))^{-2} \sigma_0^2(X_i)] < \infty$. Using the weak law of large numbers \cite[Theorem 2.2.14]{MR3930614}, we have
\begin{align*}
  \frac{1}{n} \sum_{i=1, D_i = 1}^n \Big(\frac{1}{e(X_i)}\Big)^2 \sigma_1^2(X_i) + \frac{1}{n} \sum_{i=1, D_i = 0}^n \Big(\frac{1}{1-e(X_i)}\Big)^2 \sigma_0^2(X_i) \stackrel{\sf p}{\longrightarrow} \E \Big[\frac{\sigma_1^2(X)}{e(X)} + \frac{\sigma_0^2(X)}{1-e(X)}\Big].
\end{align*}

For the second term in \eqref{eq:mbc,clt4}, using the Cauchy–Schwarz inequality,
\begin{align*}
  & \E \Big\lvert\frac{1}{n} \sum_{i=1, D_i = 1}^n \Big[\Big(1+\frac{K^1_M(i)}{M}\Big)^2 - \Big(\frac{1}{e(X_i)}\Big)^2\Big] \sigma_1^2(X_i) \Big\rvert\\
  \le & C_\sigma \E \Big[D_i\Big\lvert\Big(1+\frac{K^1_M(i)}{M}\Big)^2 - \Big(\frac{1}{e(X_i)}\Big)^2\Big\rvert \Big] = C_\sigma \E \Big[D_i \E \Big[\Big\lvert\Big(1+\frac{K^1_M(i)}{M}\Big)^2 - \Big(\frac{1}{e(X_i)}\Big)^2 \Big\rvert \Biggiven \mD \Big] \Big]\\
  \le & C_\sigma \E \Big[D_i \Big(\E \Big[\Big(\frac{K^1_M(i)}{M} - \frac{1-e(X_i)}{e(X_i)}\Big)^2 \Biggiven \mD \Big] \E \Big[\Big(2 + \frac{K^1_M(i)}{M} + \frac{1-e(X_i)}{e(X_i)}\Big)^2 \Biggiven \mD \Big] \Big)^{1/2}  \Big] = o(1),
\end{align*}
where the last step is due to Theorem~\ref{thm:risk,lp}. Then we obtain
\begin{align*}
  \frac{1}{n} \sum_{i=1, D_i = 1}^n \Big[\Big(1+\frac{K^1_M(i)}{M}\Big)^2 - \Big(\frac{1}{e(X_i)}\Big)^2\Big] \sigma_1^2(X_i) \stackrel{\sf p}{\longrightarrow} 0.
\end{align*}

For the third term in \eqref{eq:mbc,clt4}, we can establish in the same way that
\begin{align*}
  \frac{1}{n} \sum_{i=1, D_i = 0}^n \Big[\Big(\frac{1}{1-e(X_i)}\Big)^2 - \Big(1+\frac{K^0_M(i)}{M}\Big)^2 \Big] \sigma_0^2(X_i) \stackrel{\sf p}{\longrightarrow} 0.
\end{align*}

Then from \eqref{eq:mbc,clt4},
\begin{align*}
  V^E \stackrel{\sf p}{\longrightarrow} \E \Big[\frac{\sigma_1^2(X)}{e(X)} + \frac{\sigma_0^2(X)}{1-e(X)}\Big].
\end{align*}

Using Slutsky's lemma \cite[Theorem 2.8]{MR1652247} and then the definition of $\sigma^2$, we complete the proof by \eqref{eq:mbc,clt3}.
\end{proof}

\subsubsection{Proof of Lemma~\ref{lemma:mbc,dist}}

\begin{proof}[Proof of Lemma~\ref{lemma:mbc,dist}]

From Assumption~\ref{asp:risk} and Assumption~\ref{asp:dml1}, let $R := {\rm diam} (\bX) < \infty$ and $f_L := \inf_{x \in \bX, \omega \in \{0,1\}} f_\omega(x) > 0$. For any $x \in \bX$, $\omega \in \{0,1\}$ and $u \le R$, from Assumption~\ref{asp:risk},
\[
  \nu_\omega\big(B_{x,u} \cap \bX \big) \ge f_L \lambda\big(B_{x,u} \cap \bX \big) \ge f_L a \lambda\big(B_{x,u} \big) = f_L a V_d u^d.
\]
Let $c_0 = f_L a V_d$. For any $i \in \zahl{n}$, $x \in \mathbb{X}$, $M \le n_{1-D_i}$, if $0 \le u \le R n_{1-D_i}^{1/d}$, we have
\begin{align*}
    & \P\Big(\Big\lVert X_j - X_i \Big\rVert \ge un_{1-D_i}^{-1/d} \Biggiven \mD , X_i = x, j = j_M^{1-D_i}(i)\Big)\\
    \le & \P\Big({\rm Bin}\Big(n_{1-D_i}, \nu_{1-D_i}\big( B_{x, un_{1-D_i}^{-1/d}}\cap \bX \big)\Big) \le M \Biggiven \mD \Big)\\
    \le & \P\Big({\rm Bin}\Big(n_{1-D_i}, c_0 \big(un_{1-D_i}^{-1/d}\big)^d \Big) \le M \Biggiven \mD  \Big)\\
    = & \P\Big({\rm Bin}\Big(n_{1-D_i}, c_0 u^d n_{1-D_i}^{-1} \Big) \le M \Biggiven \mD \Big).
\end{align*}
Using the Chernoff bound, if $M < c_0u^d$, then 
\[
  \P\Big({\rm Bin}\Big(n_{1-D_i}, c_0 u^d n_{1-D_i}^{-1} \Big) \le M \Biggiven \mD \Big) \le \exp\Big(M - c_0 u^d + M \log\Big(\frac{c_0 u^d}{M}\Big)\Big).
\]
Notice that the above upper bound does not depend on $x$. We then obtain
\begin{align*}
  & \P\Big(\Big\lVert X_j - X_i \Big\rVert \ge un_{1-D_i}^{-1/d} \Biggiven \mD, j = j_M^{1-D_i}(i)\Big) \\
  \le & \ind\Big(M < c_0 u^d \Big) \exp\Big(M - c_0 u^d + M \log\Big(\frac{c_0 u^d}{M}\Big)\Big) + \ind\Big(M \ge c_0 u^d \Big).
\end{align*}
On the other hand, if $u > R n_{1-D_i}^{1/d}$, then the probability is zero from the definition of $R$. Accordingly, the above bound holds for any $u\ge 0$.

For any $i \in \zahl{n}$, we thus have
\begin{align*}
  & n_{1-D_i}^{p/d} \E \Big[\Big\lVert U_{M,i} \Big\rVert^p \Biggiven \mD\Big] = p\int_{0}^\infty \P\Big(\Big\lVert X_j - X_i \Big\rVert \ge un_{1-D_i}^{-1/d} \Biggiven \mD, j = j_M^{1-D_i}(i)\Big) u^{p-1} \d u\\
  \le & p\int_{0}^\infty \Big[\ind\Big(M < c_0 u^d \Big) \exp\Big(M - c_0 u^d + M \log\Big(\frac{c_0 u^d}{M}\Big)\Big) + \ind\Big(M \ge c_0 u^d \Big) \Big] u^{p-1} \d u\\
  = & p c_0^{-p/d} d^{-1} \Big[\int_{M}^\infty \Big(\frac{e}{M}\Big)^M t^{M+\frac{p}{d}-1} e^{-t} \d t + \int_{0}^M  t^{\frac{p}{d}-1} \d t\Big],
  \yestag\label{eq:mbc,dist1}
\end{align*}
where the last step is through taking $t = c_0 u^d$.

For the first term in \eqref{eq:mbc,dist1}, from Stirling's formula and $M \to \infty$,
\begin{align*}
  & \int_{M}^\infty \Big(\frac{e}{M}\Big)^M t^{M+\frac{p}{d}-1} e^{-t} \d t \le \int_0^\infty \Big(\frac{e}{M}\Big)^M t^{M+\frac{p}{d}-1} e^{-t} \d t = \Big(\frac{e}{M}\Big)^M \Gamma\Big(M+\frac{p}{d}\Big)\\
  \sim & \Big(\frac{e}{M}\Big)^M (M+1)^{\frac{p}{d}-1} \Gamma(M+1) \sim \Big(\frac{e}{M}\Big)^M (M+1)^{\frac{p}{d}-1} \sqrt{2\pi M}\Big(\frac{M}{e}\Big)^M \sim \sqrt{2\pi} M^{\frac{p}{d} - \frac{1}{2}},
\end{align*}
where $\sim$ means asymptotic convergence. 

For the second term in \eqref{eq:mbc,dist1},
\[
  \int_{0}^M  t^{\frac{p}{d}-1} \d t = \frac{d}{p} M^{\frac{p}{d}}.
\]

Combining the above two bounds then complete the proof.
\end{proof}

\subsubsection{Proof of Lemma~\ref{lemma:mbc,bias}}

\begin{proof}[Proof of Lemma~\ref{lemma:mbc,bias}]

We decompose $B_M - \hat{B}_M$ as
\begin{align*}
  & \lvert B_M - \hat{B}_M \rvert \\
  =& \Big\lvert \frac{1}{n}\sum_{i=1}^n (2D_i-1) \Big[\frac{1}{M}\sum_{m=1}^M \Big(\mu_{1-D_i}(X_i)-\mu_{1-D_i}(X_{j^{1-D_i}_m(i)}) - \hat{\mu}_{1-D_i}(X_i) + \hat{\mu}_{1-D_i}(X_{j^{1-D_i}_m(i)})\Big)\Big] \Big\rvert\\
  \le & \max_{i \in \zahl{n}, m \in \zahl{M}} \Big\lvert \mu_{1-D_i}(X_i)-\mu_{1-D_i}(X_{j^{1-D_i}_m(i)}) - \hat{\mu}_{1-D_i}(X_i) + \hat{\mu}_{1-D_i}(X_{j^{1-D_i}_m(i)}) \Big\rvert\\
  \le &  \max_{i \in \zahl{n}, m \in \zahl{M}, \omega \in \{0,1\}} \Big\lvert \mu_\omega(X_i)-\mu_\omega(X_{j^{1-D_i}_m(i)}) - \hat{\mu}_\omega(X_i) + \hat{\mu}_\omega(X_{j^{1-D_i}_m(i)}) \Big\rvert.
  \yestag\label{eq:mbc,bias1}
\end{align*}
Let $k = \lfloor d/2 \rfloor + 1$. For any $\omega \in \{0,1\}$, from Taylor expansion to $k$th order,
\begin{align}\label{eq:mbc,bias2}
    \Big\lvert \mu_\omega (X_{j^{1-D_i}_m(i)}) - \mu_{\omega} (X_i) - \sum_{\ell = 1}^{k-1} \frac{1}{\ell!} \sum_{t \in \Lambda_\ell} \partial^t \mu_{\omega} (X_i) U^t_{m,i} \Big\rvert \le \max_{t \in \Lambda_k} \lVert \partial^t \mu_{\omega} \rVert_\infty \frac{1}{k!} \sum_{t \in \Lambda_k} \lVert U_{m,i} \rVert^k.
\end{align}
In the same way,
\begin{align}\label{eq:mbc,bias3}
  \Big\lvert \hat{\mu}_\omega (X_{j^{1-D_i}_m(i)}) - \hat{\mu}_{\omega} (X_i) - \sum_{\ell = 1}^{k-1} \frac{1}{\ell!} \sum_{t \in \Lambda_\ell} \partial^t \hat{\mu}_{\omega} (X_i) U^t_{m,i} \Big\rvert \le \max_{t \in \Lambda_k} \lVert \partial^t \hat{\mu}_{\omega} \rVert_\infty \frac{1}{k!} \sum_{t \in \Lambda_k} \lVert U_{m,i} \rVert^k.
\end{align}
We also have
\begin{align}\label{eq:mbc,bias4}
  \Big\lvert \sum_{\ell = 1}^{k-1} \frac{1}{\ell!} \sum_{t \in \Lambda_\ell} (\partial^t \hat{\mu}_{\omega} (X_i) - \partial^t \mu_{\omega} (X_i))  U^t_{m,i} \Big\rvert \le \sum_{\ell = 1}^{k-1} \max_{t \in \Lambda_\ell} \lVert \partial^t \hat{\mu}_{\omega} - \partial^t \mu_{\omega} \rVert_\infty \frac{1}{\ell!} \sum_{t \in \Lambda_\ell} \lVert U_{m,i} \rVert^\ell.
\end{align}

Notice that $\lVert U_{M,i} \rVert = \max_{m \in \zahl{M}} \lVert U_{m,i} \rVert$ for any $i \in \zahl{n}, \omega \in \{0,1\}$. Then for any $\omega \in \{0,1\}$, plugging \eqref{eq:mbc,bias2}, \eqref{eq:mbc,bias3}, \eqref{eq:mbc,bias4} into \eqref{eq:mbc,bias1}, we obtain
\begin{align*}
  \lvert B_M - \hat{B}_M \rvert \lesssim (1+O_\P(1))\max_{i \in \zahl{n}} \lVert U_{M,i} \rVert^k + \max_{\ell \in \zahl{k-1 }} \max_{t \in \Lambda_\ell} \lVert \partial^t \hat{\mu}_{\omega} - \partial^t \mu_{\omega} \rVert_\infty \max_{i \in \zahl{n}} \lVert U_{M,i} \rVert^\ell.
\end{align*}

From Lemma~\ref{lemma:mbc,dist}, all moments of $(n/M)^{p/d} \lVert U_{M,i} \rVert^p$ are bounded. Then for any positive integer $p$ and any $\varepsilon>0$, we have
\begin{align*}
  \max_{i \in \zahl{n}} \lVert U_{M,i} \rVert^p = o_\P\Big(n^\varepsilon \Big(\frac{M}{n}\Big)^{p/d}\Big).
\end{align*}

We then obtain
\[
  B_M - \hat{B}_M = o_\P\Big(n^\varepsilon \Big(\frac{M}{n}\Big)^{k/d} \Big) + \max_{\ell \in \zahl{k-1 }} o_\P\Big(n^{-\gamma_\ell}n^\varepsilon \Big(\frac{M}{n}\Big)^{\ell/d} \Big).
\]
The proof is thus complete by noticing the definition of $\gamma$ and $M \lesssim n^\gamma$.
\end{proof}

\subsection{Proofs of the results in Section~\ref{sec:discussion}}

Before starting the proof, we first introduce some necessary notation. Given two samples $\{X_i\}_{i=1}^{N_0}$ and $\{Z_j\}_{j=1}^{N_1}$, for any $M \in \zahl{N_0}$ and $\ell \in \zahl{N_1}$, define $j_M(\ell)$ to be the index of the $M$-th NN of $Z_\ell$ in $\{X_i\}_{i=1}^{N_0}$. In other words, $j_M(\ell)$ is the index of $k \in \zahl{N_0}$ such that
\[
  \sum_{i=1}^{N_0} \ind\big(\lVert X_i - Z_\ell \rVert \le \lVert X_k - Z_\ell \rVert\big) = M.
\]
Further define $\cJ_M(\ell)$ to be the set of indices of all $M$-NNs of $Z_\ell$ in $\{X_i\}_{i=1}^{N_0}$, i.e., 
\[
\cJ_M(\ell) := \{j_m(\ell): m \in \zahl{M}\}.
\]

\subsubsection{Proof of Theorem~\ref{thm:rate,kl}}

\begin{proof}[Proof of Theorem~\ref{thm:rate,kl}]

Before starting the proof, we first introduce the following lemma about the truncation level for $f_1$ to attain the desirable estimation accuracy.

\begin{lemma} \label{lemma:thr} 
There exists a constant $C_0 >0$ only depending on $f_L, f_U, L, d$ such that for any $x \in \bR^d$ with
 \begin{align}\label{eq:lemma-thr-1}
 f_1(x) \ge C_0(\frac{M}{N_0})^{1/d} ~~~{\rm and} ~~~\Delta(x) \ge C_0(\frac{M}{N_0})^{1/d}
 \end{align}
 we have, for all $N_0$ sufficiently large, 
 \begin{itemize}
\item[(i)] $
    \Big\lvert \E\big[\hat{r}_M(x)\big] - r(x) \Big\rvert \le \frac{1}{2} r(x);
    $
\item[(ii)] $
  \Var\Big[ \E\big[\hat{r}_M(x) \biggiven \mX \big] \Big] \le \frac{32}{M} [r(x)]^2;
  $
\item[(iii)] with probability one,
  \begin{align*}
    \frac{N_0}{M} \nu_1\big(A_M(x)\big) \ge C_1 r(x),
  \end{align*}
  where $C_1>0$ is a constant only depending on $f_L, f_U, d$.
  \end{itemize}
\end{lemma}

Let $X$ be drawn from $\nu_0$ independent of the data. We then have
\[
  D_\phi(\nu_1 \Vert \nu_0) = \int \phi(r(x)) f_0(x) \d x = \E [\phi(r(X))].
\]

Let $\cE$ be the set of all $x \in \bR^d$ satisfying \eqref{eq:lemma-thr-1} in Lemma~\ref{lemma:thr}. We then decompose $\hat{D}_\phi$ as
\begin{align*}
  \hat{D}_\phi - D_\phi(\nu_1 \Vert \nu_0) =& \frac{1}{N_0} \sum_{i=1}^{N_0} \phi(\hat{r}_M(x)) - \E [\phi(r(X))]\\
  =& \frac{1}{N_0} \sum_{i=1}^{N_0} \Big[ \phi\Big(\hat{r}_M(X_i)\Big) - \phi\Big(\E[\hat{r}_M(X_i) \given \mX]\Big)\Big]\\
  &+ \frac{1}{N_0} \sum_{i=1}^{N_0} \Big[ \phi\Big(\E[\hat{r}_M(X_i) \given \mX]\Big) - \phi\Big(\E[\hat{r}_M(X_i) \given X_i]\Big) \Big] \ind \Big( X_i \in \cE \Big)\\
  &+ \frac{1}{N_0} \sum_{i=1}^{N_0} \Big[ \phi\Big(\E[\hat{r}_M(X_i) \given X_i]\Big) - \phi\Big(r(X_i)\Big) \Big] \ind \Big( X_i \in \cE \Big)\\
  &+ \frac{1}{N_0} \sum_{i=1}^{N_0} \Big[ \phi\Big(\E[\hat{r}_M(X_i) \given \mX]\Big) - \phi\Big(r(X_i)\Big) \Big] \ind \Big( X_i \notin \cE \Big)\\
  &+ \frac{1}{N_0} \sum_{i=1}^{N_0} \phi\Big(r(X_i)\Big) - \E [\phi(r(X))]\\
  =:& T_1 + T_2 + T_3 + T_4 + T_5.
\end{align*}

By the Cauchy–Schwarz inequality,
\[
  \E\Big[\hat{D}_\phi - D_\phi(\nu_1 \Vert \nu_0) \Big]^2 \le 5 \Big( \E[T_1]^2 + \E[T_2]^2 + \E[T_3]^2 + \E[T_4]^2 + \E[T_5]^2 \Big).
\]

We then consider $\E[T_1]^2,\E[T_2]^2,\E[T_3]^2,\E[T_4]^2,\E[T_5]^2$ seperately. We have the following lemmas.

\begin{lemma}\label{lemma:T1}
  It holds that
  \[
    \E[T_1]^2 \le C \Big[ \frac{1}{N_1} + \Big(\frac{N_0}{N_1M}\Big)^2\Big],
  \]
  where $C>0$ is a constant only depending on $f_L,f_U$.
\end{lemma}

\begin{lemma}\label{lemma:T2}
  It holds that
  \[
    \E[T_2]^2 \le C \Big[ \frac{1}{N_0} + \Big(\frac{1}{M}\Big)^2 \Big],
  \]
  where $C>0$ is a constant only depending on $f_L,f_U,d,f_U'$.
\end{lemma}

\begin{lemma}\label{lemma:T3}
  It holds that
  \[
    \E[T_3]^2 \le C \Big(\frac{M}{N_0}\Big)^{2/d},
  \]
  where $C>0$ is a constant only depending on $f_L,f_U,L,d, f_U'$.
\end{lemma}

\begin{lemma}\label{lemma:T4}
  It holds that
  \[
    \E[T_4]^2 \le C \Big(\frac{M}{N_0}\Big)^{2/d},
  \]
  where $C>0$ is a constant only depending on $f_L, f_U, a, H, d, f_U'$.
\end{lemma}

Notice that $T_5$ is a linear combination of the function of i.i.d. random variables with mean zero. From the boundedness assumption on $r$, it is easy to establish
\begin{align}\label{eq:ratekl1}
  \E[T_5]^2 \le C \frac{1}{N_0},
\end{align}
where $C>0$ is a constant only depending on $f_L,f_U$.

Combining Lemma~\ref{lemma:T1}, Lemma~\ref{lemma:T2}, Lemma~\ref{lemma:T3}, Lemma~\ref{lemma:T4}, and \eqref{eq:ratekl1} completes the proof.
\end{proof}


\subsubsection{Proof of Lemma~\ref{lemma:thr}}

\begin{proof}[Proof of Lemma~\ref{lemma:thr}]

For Part (i), recall that in the proof of Theorem~\ref{thm:pw-rate}\ref{thm:pw-rate1}, we take $\delta_N = (\frac{4}{f_L V_d})^{1/d} (\frac{M}{N_0})^{1/d}$. Take $C_0 > L (\frac{4}{f_L V_d})^{1/d}$. Then $f_1(x) \ge C_0 (\frac{M}{N_0})^{1/d} > L\delta_N$. Take $C_0 > 2 (\frac{4}{f_L V_d})^{1/d}$. Then $\Delta(x) > 2\delta_N$. Then for any $z \in \bR^d$ such that $\lVert z - x \rVert \le \delta_N$, we have $B_{z, \lVert x-z \rVert} \subset S_1$. We take $N_0$ large enough so that $\delta_N < f_L/(4L)$. Then from the proof of Theorem~\ref{thm:pw-rate}\ref{thm:pw-rate1}, $W \le 2 \frac{M}{N_0}$ implies $\lVert Z-x \rVert \le \delta_N$. From Case I in the proof of Theorem~\ref{thm:pw-rate}\ref{thm:pw-rate1}, we use \eqref{eq:ratebias2} and \eqref{eq:ratebias3}, and take $N_0$ large enough so that $L\delta_N \le f_U \wedge (f_L/4)$. We then have
\begin{align*}
  \Big\lvert \E\big[\hat{r}_M(x)\big] - r(x) \Big\rvert &\le \Big\lvert \frac{f_1(x)+L\delta_N}{f_0(x)-2L\delta_N} \frac{N_0}{N_0+1} - \frac{f_1(x)}{f_0(x)} \Big\rvert \bigvee \Big\lvert \frac{f_1(x)-L\delta_N}{f_0(x)+2L\delta_N} \frac{N_0}{N_0+1} - \frac{f_1(x)}{f_0(x)} \Big\rvert + o(N_0^{-\gamma})\\
  &\le \frac{f_0(x)L\delta_N + 2f_1(x)L\delta_N}{f_0(x)(f_0(x)-2L\delta_N)} + \frac{1}{N_0+1} \frac{f_1(x)+L\delta_N}{f_0(x)-2L\delta_N} + o(N_0^{-\gamma})\\
  &\le \frac{2}{f_L}L \delta_N + r(x)\frac{4}{f_L} L \delta_N + \frac{4f_U}{f_L}\frac{1}{N_0+1} + o(N_0^{-\gamma}).
\end{align*}
Taking $C_0 > L \frac{4}{f_L} (\frac{4}{f_L V_d})^{1/d}$,  Part (i) then holds when $N_0$ is sufficiently large.

For Part (ii), we take $C_0$ large enough as in Part (i). Then $W \le 2 \frac{M}{N_0}$ implies $\lVert Z-x \rVert \le \delta_N$. For any $\epsilon>0$, we take $C_0$ and $N_0$ sufficiently large such that for any $\lVert z-x \rVert \le2\delta_N$, we have $\lvert f_1(z) - f_1(x) \rvert \le \epsilon C_0(\frac{M}{N_0})^{1/d}$ and $\lvert f_0(z) - f_0(x) \rvert \le \epsilon f_L$, which is achievable from the Lipschitz condition. Then $\lvert f_1(z) - f_1(x) \rvert \le \epsilon f_1(x)$ and $\lvert f_0(z) - f_0(x) \rvert \le \epsilon f_0(x)$. For any $x \in \bR^d$ such that $f_1(x) \ge C_0(\frac{M}{N_0})^{1/d}$ and $\Delta(x) \ge C_0(\frac{M}{N_0})^{1/d}$, for $N_0$ sufficiently large, we have for all $0 \le w \le w+t \le 2M/N_0$,
\begin{align*}
  &\P\Big( w \le W \le w+t \Big) = \nu_1\Big(\Big\{z \in \bR^d: \nu_0(B_{z,\lVert x-z \rVert}) \in [w,w+t]\Big\}\Big)\\
  \le & \frac{f_1(x)(1+\epsilon)}{f_0(x) (1- \epsilon)} \nu_0\Big(\Big\{z \in \bR^d: \nu_0(B_{z,\lVert x-z \rVert}) \in [w,w+t]\Big\}\Big),
\end{align*}
and then
\[
  \limsup_{t \to 0} t^{-1} \P\Big( w \le W \le w+t \Big) \le r(x) \frac{1+\epsilon}{1 - \epsilon} (1+\epsilon).
\]
Proceeding in the same way as the proof of Theorem~\ref{thm:pw-rate}\ref{thm:pw-rate2} and Lemma~\ref{lemma:ratevar}, we prove Part (ii).

On Part (iii), we cite a result on lower bounding the $M$-NN distance without proof.

\begin{lemma}\label{lemma:dist,lower}\citep[Theorem 4.1]{MR3445317}
Let $X_1,\ldots,X_{N_0} \in \bR^d$ be $N_0$ independent copies from a probability measure with density $f_0$. Assume that $f_0$ is uniformly bounded by a universal constant $f_U$. If $M/\log N_0 \to \infty$ as $N_0 \to \infty$, then with probability one, for all $N_0$ large enough,
  \[
    \inf_{x \in \bR^d} \Big\lVert \cX_{(M)}(x) - x \Big\rVert \ge C \Big(\frac{M}{N_0}\Big)^{1/d},
  \]
where $C>0$ is a constant only depending on $f_U$ and $d$.
\end{lemma}

Notice that for $z \in \bR^d$, if $\lVert z-x \rVert \le \frac{1}{2} \Big\lVert \cX_{(M)}(x) - x \Big\rVert$, then $z \in A_M(x)$. From Lemma~\ref{lemma:dist,lower}, with probability one, for all $N_0$ sufficiently large, we have
\[
  \frac{N_0}{M} \nu_1\big(A_M(x)\big) \ge \frac{N_0}{M} \nu_1\Big(B_{x, \lVert \cX_{(M)}(x) - x \rVert/2}\Big) \ge \frac{N_0}{M} \nu_1\Big(B_{x, C(M/N_0)^{1/d}/2}\Big),
\]
where $C>0$ is the constant in Lemma~\ref{lemma:dist,lower}. Taking $C_0$ large enough so that $C_0 > C/2$ and $C_0 \ge LC$, then $\Delta(x) > C(M/N_0)^{1/d}/2$ and for any $\lVert z-x \rVert \le C(M/N_0)^{1/d}/2$, we have $f_1(z) \ge \frac{1}{2}f_1(x)$. Then
\[
  \frac{N_0}{M} \nu_1\big(A_M(x)\big) \ge \frac{N_0}{M} \frac{f_1(x)}{2} V_d \Big(\frac{C}{2} \Big(\frac{M}{N_0}\Big)^{\frac{1}{d}}\Big)^d = 2^{-(d+1)} C^d V_d f_1(x) \ge 2^{-(d+1)} C^d V_d f_L r(x).
\]
We then complete the proof of Part (iii).
\end{proof}

\subsubsection{Proof of Lemma~\ref{lemma:T1}}

\begin{proof}[Proof of Lemma~\ref{lemma:T1}]

Notice that 
\[
\E[T_1]^2 = \E\Big[\E[T_1^2 \given \mX]\Big] = \E\Big[\Big(\E[T_1 \given \mX]\Big)^2 + \Var[T_1 \given \mX]\Big]. 
\]
We consider the above two terms separately.

{\bf Step I.} $\E[T_1 \given \mX]$.

Let $g(x)$ be a given twice continuously differentiable function on $[0,\infty)$. Equation 2.5 in \cite{totik1994approximation} gives
\begin{align}\label{eq:T114}
  \lvert g(t) - g(x) - g'(x) (t-x) \rvert \le 4 \frac{(t-x)^2}{x} \lVert x g''(x) \rVert_{\infty},
\end{align}
for all $x,t \in [0,\infty)$.

From \eqref{eq:T114},
\begin{align*}
  \lvert \E[T_1 \given \mX] \rvert &= \Big\lvert \frac{1}{N_0} \sum_{i=1}^{N_0} \E \Big[ \phi\Big(\hat{r}_M(X_i)\Big) - \phi\Big(\E[\hat{r}_M(X_i) \given \mX]\Big) \Biggiven \mX \Big] \Big\rvert\\
  &\le 2 \lVert \phi - g \rVert_\infty + \Big\lvert \frac{1}{N_0} \sum_{i=1}^{N_0} \E \Big[ g\Big(\hat{r}_M(X_i)\Big) - g\Big(\E[\hat{r}_M(X_i) \given \mX]\Big) \Biggiven \mX \Big] \Big\rvert\\
  &\le 2 \lVert \phi - g \rVert_\infty + 4 \lVert x g''(x) \rVert_{\infty} \frac{1}{N_0} \sum_{i=1}^{N_0} \E \Big[ \frac{1}{\E[\hat{r}_M(X_i) \given \mX]} \Big(\hat{r}_M(X_i) - \E[\hat{r}_M(X_i) \given \mX]\Big)^2 \Biggiven \mX \Big]\\
  &= 2 \lVert \phi - g \rVert_\infty + 4 \lVert x g''(x) \rVert_{\infty} \frac{1}{N_0} \sum_{i=1}^{N_0} \frac{\Var[\hat{r}_M(X_i) \given \mX]}{\E[\hat{r}_M(X_i) \given \mX]}.
\end{align*}
From the definition of $\hat{r}_M(X_i)$,
\[
  \hat{r}_M(X_i) = \frac{N_0}{N_1} \frac{K_M(i)}{M} = \frac{N_0}{N_1M} \sum_{j=1}^{N_1} \ind\big(Z_j \in A_M(i)\big).
\]
Then conditional on $\mX$,
\[
  \frac{N_1M}{N_0} \hat{r}_M(X_i) \Biggiven \mX \sim {\rm Bin}(N_1, A_M(i)).
\]
It is easy to calculate
\[
  \E[\hat{r}_M(X_i) \given \mX] = \frac{N_0}{M} \nu_1\big(A_M(i)\big), ~~ \Var[\hat{r}_M(X_i) \given \mX] = \frac{N_0^2}{N_1M^2} \nu_1\big(A_M(i)\big) \big[1-\nu_1\big(A_M(i)\big)\big].
\]
Then
\begin{align*}
  \lvert \E[T_1 \given \mX] \rvert &\le 2 \lVert \phi - g \rVert_\infty + 4 \lVert x g''(x) \rVert_{\infty} \frac{1}{N_0} \sum_{i=1}^{N_0}  \frac{N_0}{N_1M} \big[1-\nu_1\big(A_M(i)\big)\big]\\
  &\le 2 \lVert \phi - g \rVert_\infty + 4 \frac{N_0}{N_1M}\lVert x g''(x) \rVert_{\infty}.
\end{align*}

We then cite the equivalence theorem.
\begin{lemma}\label{lemma:equiv}\citep[Theorem 2.1.1]{ditzian2012moduli}
  For any real-valued function $\phi$ on $[0,\infty)$ and $t>0$, there exists a universal constant $C>0$ such that
  \[
    \inf_{g \in C^2[0,\infty)} \Big( \lVert \phi - g \rVert_\infty + t^2 \lVert x g''(x) \rVert_{\infty} \Big) \le C \omega_\varphi^2 (\phi,t)_\infty,
  \]
  where $\omega_\varphi^2 (\phi,t)_\infty$ is the second-order modulus of smoothness with $\varphi(x) = \sqrt{x}$, and $C^2[0,\infty)$ is the class of all twice continuously differentiable functions on $[0,\infty)$.
\end{lemma}

From Example 3.4.2 in \cite{ditzian2012moduli}, when $\phi(x) = x\log x$, we have $\omega_\varphi^2 (\phi,t)_\infty \asymp t^2$. Then from Lemma~\ref{lemma:equiv}, we obtain
\begin{align}\label{eq:T112}
  \lvert \E[T_1 \given \mX] \rvert \lesssim \frac{N_0}{N_1M}.
\end{align}

{\bf Step II.} $\Var[T_1 \given \mX]$.

Conditional on $\mX$,
\begin{align*}
  \Var[T_1 \given \mX] &= \Var\Big[\frac{1}{N_0} \sum_{i=1}^{N_0} \Big[ \phi\Big(\hat{r}_M(X_i)\Big) - \phi\Big(\E[\hat{r}_M(X_i) \given \mX]\Big)\Big] \Biggiven \mX\Big] \\
  &= \Var\Big[\frac{1}{N_0} \sum_{i=1}^{N_0} \phi\Big(\hat{r}_M(X_i)\Big) \Biggiven \mX\Big].
  \yestag\label{eq:T111}
\end{align*}

We then apply Efron-Stein inequality \cite[Theorem 3.1]{boucheron2013concentration}. Recall $\mZ = (Z_1,\ldots,Z_{N_1})$ and for any $\ell \in \zahl{N_1}$, define
\[
  \mZ_\ell := (Z_1,\ldots,Z_{\ell-1},\tZ_\ell,Z_{\ell+1},\ldots,Z_{N_1}),
\]
where $[\tZ_\ell]_{\ell=1}^{N_1}$ are independent copies of $[Z_\ell]_{\ell=1}^{N_1}$.

Fix $\ell \in \zahl{N_1}$. Recall that $\cJ_M(\cdot)$ are defined based on $\mX$ and $\mZ$. We then define $\cJ'_M(\cdot)$ in the same way but based on $\mX$ and $\mZ_\ell$.

Let $\hat{r}'_M$ be the density ratio estimator based on $\mX$ and $\mZ_\ell$. Then for any $i \in \zahl{N_0}$,
\[
  \hat{r}_M(X_i) = \frac{N_0}{N_1 M} \sum_{j=1}^{N_1} \ind\big(i \in \cJ_M(j)\big), ~~ \hat{r}'_M(X_i) = \frac{N_0}{N_1 M} \sum_{j=1}^{N_1} \ind\big(i \in \cJ'_M(j)\big).
\]

Notice that for $j \in \zahl{N_1}$ and $j \neq \ell$, $\cJ_M(j) = \cJ'_M(j)$ since both $\cJ_M(\cdot)$ and $\cJ'_M(\cdot)$ are based on $\mX$. Then if $i \in \cJ_M(\ell) \setminus \cJ'_M(\ell)$ or $i \in \cJ'_M(\ell) \setminus \cJ_M(\ell)$, we have $\lvert \hat{r}_M(X_i) - \hat{r}'_M(X_i) \rvert = \frac{N_0}{N_1M}$. If $i \in \cJ_M(\ell) \cap \cJ'_M(\ell)$ or $i \notin \cJ_M(\ell) \cup \cJ'_M(\ell)$, we have $\hat{r}_M(X_i) = \hat{r}'_M(X_i)$. As a result,
\begin{align*}
  & \E\Big[\Big[\frac{1}{N_0} \sum_{i=1}^{N_0} \Big(\phi\Big(\hat{r}_M(X_i)\Big) - \phi\Big(\hat{r}'_M(X_i)\Big) \Big) \Big]^2\Biggiven \mX\Big]\\
  \le & \E\Big[\Big[\frac{1}{N_0} \sum_{i=1}^{N_0} \Big\lvert\phi\Big(\hat{r}_M(X_i)\Big) - \phi\Big(\hat{r}'_M(X_i)\Big) \Big\rvert \Big]^2\Biggiven \mX\Big]\\
  = & \E\Big[\Big[\frac{1}{N_0} \sum_{\substack{i=1\\ i \in (\cJ_M(\ell) \setminus \cJ'_M(\ell)) \cup (\cJ'_M(\ell) \setminus \cJ_M(\ell))}}^{N_0} \Big\lvert\phi\Big(\hat{r}_M(X_i)\Big) - \phi\Big(\hat{r}'_M(X_i)\Big) \Big\rvert \Big]^2\Biggiven \mX\Big]\\
  \le & \frac{1}{N_0^2} \E\Big[ \Big\lvert (\cJ_M(\ell) \setminus \cJ'_M(\ell)) \cup (\cJ'_M(\ell) \setminus \cJ_M(\ell)) \Big\rvert \\
  & \sum_{\substack{i=1\\ i \in (\cJ_M(\ell) \setminus \cJ'_M(\ell)) \cup (\cJ'_M(\ell) \setminus \cJ_M(\ell)) }}^{N_0} \Big[\phi\Big(\hat{r}_M(X_i)\Big) - \phi\Big(\hat{r}'_M(X_i)\Big) \Big]^2\Biggiven \mX\Big]\\
  \le & \frac{2M}{N_0^2} \E\Big[ \sum_{\substack{i=1\\ i \in (\cJ_M(\ell) \setminus \cJ'_M(\ell)) \cup (\cJ'_M(\ell) \setminus \cJ_M(\ell)) }}^{N_0} \Big[\phi\Big(\hat{r}_M(X_i)\Big) - \phi\Big(\hat{r}'_M(X_i)\Big) \Big]^2\Biggiven \mX\Big],
\end{align*}
where for any set $A$, $\lvert A \rvert$ is the cardinality of $A$. The last two steps are due to the Cauchy–Schwarz inequality and the fact that $\lvert \cJ_M(\ell) \rvert \le M, \lvert \cJ'_M(\ell) \rvert \le M$. Thus
\[
  \lvert (\cJ_M(\ell) \setminus \cJ'_M(\ell)) \bigcup (\cJ'_M(\ell) \setminus \cJ_M(\ell)) \rvert \le 2M.
\]
Since $[X_i]_{i=1}^{N_0}$ i.i.d., we have
\begin{align*}
  & \E\Big[\frac{1}{N_0} \sum_{i=1}^{N_0} \Big(\phi\Big(\hat{r}_M(X_i)\Big) - \phi\Big(\hat{r}'_M(X_i)\Big) \Big) \Big]^2\\
  \le & \frac{2M}{N_0} \E\Big[\Big[\phi\Big(\hat{r}_M(X_i)\Big) - \phi\Big(\hat{r}'_M(X_i)\Big) \Big]^2 \ind\Big(i \in (\cJ_M(\ell) \setminus \cJ'_M(\ell)) \cup (\cJ'_M(\ell) \setminus \cJ_M(\ell)) \Big) \Big]\\
  \le & \frac{2M}{N_0} \E\Big[\Big[\phi\Big(\hat{r}_M(X_i)\Big) - \phi\Big(\hat{r}'_M(X_i)\Big) \Big]^2 \Big[\ind\Big(i \in \cJ_M(\ell) \setminus \cJ'_M(\ell)\Big) + \ind\Big(\cJ'_M(\ell) \setminus \cJ_M(\ell) \Big)\Big] \Big]\\
  =& \frac{4M}{N_0} \E\Big[\Big[\phi\Big(\hat{r}_M(X_i)\Big) - \phi\Big(\hat{r}'_M(X_i)\Big) \Big]^2 \ind\Big(i \in \cJ_M(\ell) \setminus \cJ'_M(\ell)\Big) \Big],
  \yestag\label{eq:T11}
\end{align*}
where the last step is from symmetry.

Let
\[
  \hat{r}^{-\ell}_M(X_i) = \frac{N_0}{N_1 M} \sum_{j=1, j \neq \ell}^{N_1} \ind\big(i \in \cJ_M(j)\big).
\]
Then $\hat{r}_M(X_i) = \hat{r}^{-\ell}_M(X_i) + \frac{N_0}{N_1M} \ind\big(i \in \cJ_M(\ell)\big), \hat{r}'_M(X_i) = \hat{r}^{-\ell}_M(X_i) + \frac{N_0}{N_1M} \ind\big(i \in \cJ'_M(\ell)\big)$. Notice that $\hat{r}^{-\ell}_M(X_i)$ is a function of $\mX$ and $[Z_j]_{j=1,j \neq \ell}^{N_0}$. Since $[Z_j]_{j=1}^{N_1}$ are i.i.d. and independent with $\tZ_\ell$,
\begin{align*}
  & \E\Big[\Big[\phi\Big(\hat{r}_M(X_i)\Big) - \phi\Big(\hat{r}'_M(X_i)\Big) \Big]^2 \ind\Big(i \in \cJ_M(\ell) \setminus \cJ'_M(\ell)\Big) \Biggiven \mX \Big]\\
  = & \E\Big[\Big[\phi\Big(\hat{r}^{-\ell}_M(X_i) + \frac{N_0}{N_1M}\Big) - \phi\Big(\hat{r}^{-\ell}_M(X_i)\Big) \Big]^2 \Biggiven \mX \Big] \cdot \E\Big[ \ind\Big(i \in \cJ_M(\ell) \setminus \cJ'_M(\ell)\Big) \Biggiven \mX \Big].
  \yestag\label{eq:T12}
\end{align*}

For the second term in \eqref{eq:T12},
\begin{align}\label{eq:T15}
  \E\Big[ \ind\Big(i \in \cJ_M(\ell) \setminus \cJ'_M(\ell)\Big) \Biggiven \mX \Big] \le \E\Big[ \ind\Big(i \in \cJ_M(\ell) \Big) \Biggiven \mX \Big] = \nu_1\Big(A_M(i)\Big).
\end{align}

For the first term in \eqref{eq:T12}, notice that $\frac{N_1M}{N_0} \hat{r}^{-\ell}_M(X_i) \sim {\rm Bin}\Big(N_1-1, \nu_1\big(A_M(i)\big)\Big)$. Then using the Chernoff bound, we obtain
\begin{align*}
  & \P\Big( \frac{N_1M}{N_0} \hat{r}^{-\ell}_M(X_i) \le \frac{1}{2} (N_1 - 1) \nu_1\big(A_M(i)\big) \Big) \le \exp\Big(-\frac{1}{2}(1-\log 2) (N_1 - 1) \nu_1\big(A_M(i)\big) \Big),\\
  & \P\Big( \frac{N_1M}{N_0} \hat{r}^{-\ell}_M(X_i) \ge 2 (N_1 - 1) \nu_1\big(A_M(i)\big) \Big) \le \exp\Big(-(2\log 2 - 1) (N_1 - 1) \nu_1\big(A_M(i)\big) \Big).
\end{align*}

If $\nu_1\big(A_M(i)\big) \ge C_\gamma \frac{\log N_1}{N_1-1}$, then for any $\gamma>0$, there exists a constant $C_\gamma>0$ only depending on $\gamma$ such that, with probability at least $1 - N_1^{-\gamma}$, we have
\[
  \frac{1}{2} \frac{N_0}{M} \Big( 1 - \frac{1}{N_1}\Big) \nu_1\big(A_M(i)\big) \le \hat{r}^{-\ell}_M(X_i) \le 2 \frac{N_0}{M} \Big( 1 - \frac{1}{N_1}\Big) \nu_1\big(A_M(i)\big).
\]
Notice that $\frac{N_0}{N_1M} \prec \frac{N_0}{M}  \nu_1\big(A_M(i)\big)$ when $\nu_1\big(A_M(i)\big) \ge C_\gamma \frac{\log N_1}{N_1-1}$. Then for $N_0$ sufficiently large, with probability at least $1 - N_1^{-\gamma}$, the following sandwich inequality holds:
\begin{align}\label{eq:T13}
  \frac{1}{4} \frac{N_0}{M} \nu_1\big(A_M(i)\big) \le \hat{r}^{-\ell}_M(X_i) \le \hat{r}^{-\ell}_M(X_i) + \frac{N_0}{N_1M} \le 4 \frac{N_0}{M} \nu_1\big(A_M(i)\big).
\end{align}

If $\nu_1\big(A_M(i)\big) \le C_\gamma \frac{\log N_1}{N_1-1}$, then by the Chernoff bound,
\begin{align*}
  & \P\Big( \frac{N_1M}{N_0} \hat{r}^{-\ell}_M(X_i) \ge 2 C_\gamma \log N_1 \Big) \\
  \le & \exp\Big(2 C_\gamma \log N_1 - (N_1 - 1) \nu_1\big(A_M(i)\big) - 2 C_\gamma \log N_1 \log\Big(\frac{2 C_\gamma \log N_1}{(N_1 - 1) \nu_1\big(A_M(i)\big)}\Big)\Big)\\
  \le & \exp\Big(-(2\log 2 - 1) C_\gamma \log N_1\Big) \le N_1^{-\gamma}.
\end{align*}
Then for $N_0$ sufficiently large, with probability at least $1 - N_1^{-\gamma}$,
\begin{align}\label{eq:T14}
  \hat{r}^{-\ell}_M(X_i) \le 2 C_\gamma \frac{N_0 \log N_1}{N_1M}.
\end{align}

Then for any $\epsilon \in (0,1)$,
\begin{align*}
  & \E\Big[\Big[\phi\Big(\hat{r}^{-\ell}_M(X_i) + \frac{N_0}{N_1M}\Big) - \phi\Big(\hat{r}^{-\ell}_M(X_i)\Big) \Big]^2 \Biggiven \mX \Big]\\
  \le & C_\epsilon \Big(\frac{N_0}{N_1M}\Big)^{2(1-\epsilon)} \ind \Big(\nu_1\big(A_M(i)\big) \le C_\gamma \frac{\log N_1}{N_1}\Big) + \frac{2}{N_1^\gamma} \Big(2\frac{N_0}{M} \log \frac{N_0}{M}\Big)^2\\
  &+ C \Big[1 + \Big(\log\Big(\frac{N_0}{M} \nu_1\big(A_M(i)\big)\Big)\Big)^2\Big] \Big(\frac{N_0}{N_1M}\Big)^2 \ind \Big(\nu_1\big(A_M(i)\big) \ge C_\gamma \frac{\log N_1}{N_1}\Big),
  \yestag\label{eq:T16}
\end{align*}
where $C_\epsilon>0$ is a constant only depending on $\epsilon$ and $C>0$ is a universal constant. The first term is due to \eqref{eq:T14} and the inequality 
\[
  \lvert x \log x - y \log y \lvert \le C_\epsilon \lvert x - y \rvert^{1-\epsilon},
\]
whenever $x,y \in [0,1]$ and $\lvert x - y \rvert \le \frac{1}{2}$. The second term is from $\phi(x) = x\log x$ and $\hat{r}^{-\ell}_M(X_i) + \frac{N_0}{N_1M} \le \frac{N_0}{M}$. The third term is from $\phi'(x) = 1 + \log x$ and the mean value theorem by using the sandwich inequality \eqref{eq:T13}.

Plugging \eqref{eq:T15} and \eqref{eq:T16} into \eqref{eq:T11} and tracing back to \eqref{eq:T12} yields
\begin{align*}
  & \E\Big[\frac{1}{N_0} \sum_{i=1}^{N_0} \Big(\phi\Big(\hat{r}_M(X_i)\Big) - \phi\Big(\hat{r}'_M(X_i)\Big) \Big) \Big]^2\\
  \le & \frac{4M}{N_0} \E \Big[ C_\epsilon \Big(\frac{N_0}{N_1M}\Big)^{2(1-\epsilon)} \nu_1\Big(A_M(i)\Big) \ind \Big(\nu_1\big(A_M(i)\big) \le C_\gamma \frac{\log N_1}{N_1}\Big) + \frac{2}{N_1^\gamma} \Big(2\frac{N_0}{M} \log \frac{N_0}{M}\Big)^2 \nu_1\Big(A_M(i)\Big)\\
  &+ C \Big[1 + \Big(\log\Big(\frac{N_0}{M} \nu_1\big(A_M(i)\big)\Big)\Big)^2\Big] \Big(\frac{N_0}{N_1M}\Big)^2 \nu_1\Big(A_M(i)\Big) \ind \Big(\nu_1\big(A_M(i)\big) \ge C_\gamma \frac{\log N_1}{N_1}\Big) \Big].
  \yestag\label{eq:T17}
\end{align*}

For the first term in \eqref{eq:T17}, since $MN_1/(N_0\log^2N_1) \to \infty$,
\begin{align*}
  & C_\epsilon \frac{M}{N_0} \Big(\frac{N_0}{N_1M}\Big)^{2(1-\epsilon)} \nu_1\Big(A_M(i)\Big) \ind \Big(\nu_1\big(A_M(i)\big) \le C_\gamma \frac{\log N_1}{N_1}\Big) \\
  \le & C_\epsilon C_\gamma \frac{M}{N_0} \Big(\frac{N_0}{N_1M}\Big)^{2(1-\epsilon)} \frac{\log N_1}{N_1} = C_\epsilon C_\gamma \frac{1}{N_1^2} \Big(\frac{N_0}{N_1M}\Big)^{1-2\epsilon} \log N_1 \prec \frac{1}{N_1^2},
  \yestag\label{eq:T18}
\end{align*}
where we take $\epsilon$ sufficiently small.

For the second term in \eqref{eq:T17}, using Corollary~\ref{crl:moment,catch,random}, $\E [r(X)] = 1$ and $N_1^{-\frac{d}{1+d}}\log N_0 \to 0$,
\begin{align*}
  \frac{M}{N_0} \frac{1}{N_1^\gamma} \Big(\frac{N_0}{M} \log \frac{N_0}{M}\Big)^2 \E\Big[ \nu_1\Big(A_M(i)\Big) \Big] & = \frac{1}{N_1^\gamma} \Big(\log \frac{N_0}{M}\Big)^2 \E\Big[ \frac{N_0}{M} \nu_1\Big(A_M(i)\Big) \Big]\\
  \lesssim & \frac{1}{N_1^\gamma} \Big(\log \frac{N_0}{M}\Big)^2 \prec \frac{1}{N_1^2},
  \yestag\label{eq:T19}
\end{align*}
where we take $\gamma$ sufficiently large.

For the third term in \eqref{eq:T17}, using Corollary~\ref{crl:moment,catch,random},
\begin{align*}
  & \frac{M}{N_0} \E \Big[ \Big[1 + \Big(\log\Big(\frac{N_0}{M} \nu_1\big(A_M(i)\big)\Big)\Big)^2\Big] \Big(\frac{N_0}{N_1M}\Big)^2 \nu_1\Big(A_M(i)\Big) \ind \Big(\nu_1\big(A_M(i)\big) \ge C_\gamma \frac{\log N_1}{N_1}\Big) \Big]\\
  \le & \frac{M}{N_0} \E \Big[ \Big[1 + \Big(\log\Big(\frac{N_0}{M} \nu_1\big(A_M(i)\big)\Big)\Big)^2\Big] \Big(\frac{N_0}{N_1M}\Big)^2 \nu_1\Big(A_M(i)\Big) \Big]\\
  = & \frac{1}{N_1^2} \E \Big[ \frac{N_0}{M}\nu_1\Big(A_M(i)\Big) + \Big(\log\Big(\frac{N_0}{M} \nu_1\big(A_M(i)\big)\Big)\Big)^2 \frac{N_0}{M}\nu_1\Big(A_M(i)\Big) \Big]\\
  \le & \frac{1}{N_1^2} \E \Big[ \frac{N_0}{M}\nu_1\Big(A_M(i)\Big) + \Big(\frac{N_0}{M}\nu_1\Big(A_M(i)\Big)\Big)^2 \vee 4e^{-2}\Big] \lesssim C \frac{1}{N_1^2},
  \yestag\label{eq:T110}
\end{align*}
where $C>0$ is a constant only depending on $f_L,f_U$. The last step is from $(\log x)^2 \le x$ when $x \ge 1$ and the maximum of $(\log x)^2 x$ is $4e^{-2}$ when $x \in [0,1]$.

Plugging \eqref{eq:T18}, \eqref{eq:T19}, \eqref{eq:T110} into \eqref{eq:T17} yields
\[
  \E\Big[\frac{1}{N_0} \sum_{i=1}^{N_0} \Big(\phi\Big(\hat{r}_M(X_i)\Big) - \phi\Big(\hat{r}'_M(X_i)\Big) \Big) \Big]^2 \lesssim C \frac{1}{N_1^2},
\]
where $C>0$ is a constant only depending on $f_L,f_U$.

By Efron-Stein inequality and \eqref{eq:T111},
\begin{align*}
  \E \Big[\Var[T_1 \given \mX] \Big] &= \E \Big[\Var\Big[\frac{1}{N_0} \sum_{i=1}^{N_0} \phi\Big(\hat{r}_M(X_i)\Big) \Biggiven \mX\Big] \Big]\\
  &\lesssim N_1 \E\Big[\frac{1}{N_0} \sum_{i=1}^{N_0} \Big(\phi\Big(\hat{r}_M(X_i)\Big) - \phi\Big(\hat{r}'_M(X_i)\Big) \Big) \Big]^2 \lesssim C \frac{1}{N_1},
  \yestag\label{eq:T113}
\end{align*}
where $C>0$ is a constant only depending on $f_L,f_U$.

Combining \eqref{eq:T112} and \eqref{eq:T113} completes the proof.
\end{proof}

\subsubsection{Proof of Lemma~\ref{lemma:T2}}

\begin{proof}[Proof of Lemma~\ref{lemma:T2}]

From \eqref{eq:T114},
\begin{align*}
  \lvert \E[T_2 \given \mX] \rvert &= \Big\lvert \E \Big[ \frac{1}{N_0} \sum_{i=1}^{N_0} \Big[ \phi\Big(\E[\hat{r}_M(X_i) \given \mX]\Big) - \phi\Big(\E[\hat{r}_M(X_i) \given X_i]\Big) \Big] \ind \Big( X_i \in \cE \Big) \Big] \Big\rvert\\
  &\le 2 \lVert \phi - g \rVert_\infty + \Big\lvert \E \Big[ \frac{1}{N_0} \sum_{i=1}^{N_0} \Big[ g\Big(\E[\hat{r}_M(X_i) \given \mX]\Big) - g\Big(\E[\hat{r}_M(X_i) \given X_i]\Big) \Big] \ind \Big( X_i \in \cE \Big) \Big] \Big\rvert\\
  &\le 2 \lVert \phi - g \rVert_\infty + 4 \lVert x g''(x) \rVert_{\infty} \frac{1}{N_0} \sum_{i=1}^{N_0} \E \Big[ \frac{(\E[\hat{r}_M(X_i) \given \mX] - \E[\hat{r}_M(X_i) \given X_i] )^2 }{\E[\hat{r}_M(X_i) \given X_i]} \ind \Big( X_i \in \cE \Big) \Big]\\
  &= 2 \lVert \phi - g \rVert_\infty + 4 \lVert x g''(x) \rVert_{\infty} \frac{1}{N_0} \sum_{i=1}^{N_0} \E \Big[\frac{\Var[\E[\hat{r}_M(X_i) \given \mX] \given X_i]}{\E[\hat{r}_M(X_i) \given X_i]} \ind \Big( X_i \in \cE \Big) \Big].
\end{align*}
Using Lemma~\ref{lemma:thr},
\begin{align*}
  \lvert \E[T_2 \given \mX] \rvert &\le 2 \lVert \phi - g \rVert_\infty + 4 \lVert x g''(x) \rVert_{\infty} \frac{1}{N_0} \sum_{i=1}^{N_0} \E \Big[\frac{\frac{32}{M} [r(X_i)]^2}{\frac{1}{2}r(X_i)} \ind \Big( X_i \in \cE \Big) \Big]\\
  &\le 2 \lVert \phi - g \rVert_\infty + 4 \lVert x g''(x) \rVert_{\infty} \frac{1}{N_0} \sum_{i=1}^{N_0} \frac{64}{M} \E [r(X_i)]\\
  &= 2 \lVert \phi - g \rVert_\infty +  \frac{256}{M} \lVert x g''(x) \rVert_{\infty}.
\end{align*}
Using Lemma~\ref{lemma:equiv}, we obtain
\begin{align}\label{eq:T21}
  \lvert \E[T_2 \given \mX] \rvert \lesssim C\frac{1}{M},
\end{align}
where $C>0$ is a constant only depending on $f_L,f_U$.

For the variance, we have
\[
  \Var[T_2] = \Var\Big[ \frac{1}{N_0} \sum_{i=1}^{N_0} \Big[ \phi\Big(\E[\hat{r}_M(X_i) \given \mX]\Big) - \phi\Big(\E[\hat{r}_M(X_i) \given X_i]\Big) \Big] \ind \Big( X_i \in \cE \Big) \Big].
\]
To apply Efron-Stein inequality, recall $\mX = (X_1,\ldots,X_{N_0})$ and for any $\ell \in \zahl{N_0}$,
\[
  \mX_\ell = (X_1,\ldots,X_{\ell-1},\tX_\ell,X_{\ell+1},\ldots,X_{N_0}),
\]
where $[\tX_\ell]_{\ell=1}^{N_0}$ are independent copies of $[X_\ell]_{\ell=1}^{N_0}$. 

Fix $\ell \in \zahl{N_0}$. Recall that $A_M(\cdot)$ are defined based on $\mX$. We then define $A'_M(\cdot)$ in the same way but based on $\mX_\ell$, and define $A^\ell_M(\cdot)$ based on $(X_1,\ldots,X_{\ell-1},X_{\ell+1},\ldots,X_{N_0})$.

(1) If $i= \ell$,
\begin{align*}
  & \Big\lvert \Big[ \phi\Big(\E[\hat{r}_M(X_\ell) \given \mX]\Big) - \phi\Big(\E[\hat{r}_M(X_\ell) \given X_\ell]\Big) \Big] \ind \Big( X_\ell \in \cE \Big) \\
  & - \Big[ \phi\Big(\E[\hat{r}_M(\tX_\ell) \given \mX_\ell]\Big) - \phi\Big(\E[\hat{r}_M(\tX_\ell) \given \tX_\ell]\Big) \Big] \ind \Big( \tX_\ell \in \cE \Big) \Big\rvert\\
  \le & \Big\lvert \phi\Big(\E[\hat{r}_M(X_\ell) \given \mX]\Big) - \phi\Big(\E[\hat{r}_M(X_\ell) \given X_\ell]\Big) \Big\rvert \ind \Big( X_\ell \in \cE \Big) \\
  & + \Big\lvert \phi\Big(\E[\hat{r}_M(\tX_\ell) \given \mX_\ell]\Big) - \phi\Big(\E[\hat{r}_M(\tX_\ell) \given \tX_\ell]\Big) \Big\rvert \ind \Big( \tX_\ell \in \cE \Big).
\end{align*}

(2) If $i \neq \ell$,
\begin{align*}
  & \Big\lvert \Big[ \phi\Big(\E[\hat{r}_M(X_i) \given \mX]\Big) - \phi\Big(\E[\hat{r}_M(X_i) \given X_i]\Big) \Big] \ind \Big( X_i \in \cE \Big) \\
  & - \Big[ \phi\Big(\E[\hat{r}_M(X_i) \given \mX_\ell]\Big) - \phi\Big(\E[\hat{r}_M(X_i) \given X_i]\Big) \Big] \ind \Big( X_i \in \cE \Big) \Big\rvert\\
  = & \Big\lvert \phi\Big(\E[\hat{r}_M(X_i) \given \mX]\Big) - \phi\Big(\E[\hat{r}_M(X_i) \given \mX_\ell]\Big) \Big\rvert \ind \Big( X_i \in \cE \Big)\\
  = & \Big\lvert \phi\Big(\frac{N_0}{M} \nu_1\big(A_M(i)\big) \Big) - \phi\Big(\frac{N_0}{M} \nu_1\big(A'_M(i)\big) \Big) \Big\rvert \ind \Big( X_i \in \cE \Big)\\
  \le & \Big\lvert \phi\Big(\frac{N_0}{M} \nu_1\big(A_M(i)\big) \Big) - \phi\Big(\frac{N_0}{M} \nu_1\big(A^\ell_M(i)\big) \Big) \Big\rvert \ind \Big( X_i \in \cE \Big)\\
  &+ \Big\lvert \phi\Big(\frac{N_0}{M} \nu_1\big(A_M^\ell(i)\big) \Big) - \phi\Big(\frac{N_0}{M} \nu_1\big(A'_M(i)\big) \Big) \Big\rvert \ind \Big( X_i \in \cE \Big).
\end{align*}

Since $[X_\ell]_{\ell=1}^{N_0}$ and $[\tX_\ell]_{\ell=1}^{N_0}$ are exchangeable, using Efron-Stein inequality, we have
\begin{align*}
  \Var[T_2] \lesssim& \frac{1}{N_0} \E \Big[ \Big\lvert \phi\Big(\E[\hat{r}_M(X_\ell) \given \mX]\Big) - \phi\Big(\E[\hat{r}_M(X_\ell) \given X_\ell]\Big) \Big\rvert \ind \Big( X_\ell \in \cE \Big) \\
  & + \sum_{i=1,i \neq \ell}^{N_0} \Big\lvert \phi\Big(\frac{N_0}{M} \nu_1\big(A_M(i)\big) \Big) - \phi\Big(\frac{N_0}{M} \nu_1\big(A^\ell_M(i)\big) \Big) \Big\rvert \ind \Big( X_i \in \cE \Big) \Big]^2\\
  \lesssim & \frac{1}{N_0} \E \Big[ \Big[ \phi^2\Big(\E[\hat{r}_M(X_\ell) \given \mX]\Big) + \phi^2\Big(\E[\hat{r}_M(X_\ell) \given X_\ell]\Big) \Big] \ind \Big( X_\ell \in \cE \Big) \Big]\\
  & + \frac{1}{N_0} \E \Big[ \sum_{i=1,i \neq \ell}^{N_0} \Big\lvert \phi\Big(\frac{N_0}{M} \nu_1\big(A_M(i)\big) \Big) - \phi\Big(\frac{N_0}{M} \nu_1\big(A^\ell_M(i)\big) \Big) \Big\rvert \ind \Big( X_i \in \cE \Big) \Big]^2.
\end{align*}

For the first term, from Corollary~\ref{crl:moment,catch,random}, $r(x)$ is bounded and $\phi^2(x) \le x^3$ when $x \ge 1$, we obtain
\[
  \E \Big[ \Big[ \phi^2\Big(\E[\hat{r}_M(X_\ell) \given \mX]\Big) + \phi^2\Big(\E[\hat{r}_M(X_\ell) \given X_\ell]\Big) \Big] \ind \Big( X_\ell \in \cE \Big) \Big] \lesssim C,
\]
where $C>0$ is a constant only depending on $f_L, f_U$.

For the second term, notice that $A_M(i) \subset A_M^\ell(i)$ and 
\[
  A_M^\ell(i) \setminus A_M(i) = \{z \in \bR^d: j_{M+1}(z) = i, \ell \in \cJ_M(z)\}.
\]

Then for an independent copy $Z \sim \nu_1$,
\[
  \frac{N_0}{M} \nu_1\big(A^\ell_M(i)\big) - \frac{N_0}{M} \nu_1\big(A_M(i)\big) = \frac{N_0}{M} \P\Big( j_{M+1}(Z) = i, \ell \in \cJ_M(Z) \Biggiven \mX \Big) \le \frac{N_0}{M} \nu_1\big(A_M(\ell)\big).
\]

By the mean value theorem, we have
\begin{align*}
  & \sum_{i=1,i \neq \ell}^{N_0} \Big\lvert \phi\Big(\frac{N_0}{M} \nu_1\big(A_M(i)\big) \Big) - \phi\Big(\frac{N_0}{M} \nu_1\big(A^\ell_M(i)\big) \Big) \Big\rvert \ind \Big( X_i \in \cE \Big)\\
  = &  \sum_{i=1,i \neq \ell}^{N_0} \Big\lvert 1 + \log \xi_i \Big\rvert \Big[\frac{N_0}{M} \P\Big( j_{M+1}(Z) = i, \ell \in \cJ_M(Z) \Biggiven \mX \Big) \Big] \ind \Big( X_i \in \cE \Big),
\end{align*}
where $\xi_i$ is between $\frac{N_0}{M} \nu_1\big(A_M(i)\big)$ and $\frac{N_0}{M} \nu_1\big(A^\ell_M(i)\big)$.

When $\xi_i \le 1$, $\lvert \log \xi_i \rvert \le -\log\Big(\frac{N_0}{M} \nu_1\big(A_M(i)\big) \Big) \ind\Big(\frac{N_0}{M} \nu_1\big(A_M(i)\big) \le 1\Big)$. When $\xi_i \ge 1$, $\lvert \log \xi_i \rvert \le \log \Big(\frac{N_0}{M} \nu_1\big(A^\ell_M(i)\big) \Big) \le \Big(\frac{N_0}{M} \nu_1\big(A_M(i)\big) + \frac{N_0}{M} \nu_1\big(A_M(\ell)\big) \Big)$. Then
\[
  \Big\lvert 1 + \log \xi_i \Big\rvert \le 1 -\log\Big(\frac{N_0}{M} \nu_1\big(A_M(i)\big) \Big) \ind\Big(\frac{N_0}{M} \nu_1\big(A_M(i)\big) \le 1\Big) + \Big(\frac{N_0}{M} \nu_1\big(A_M(i)\big) + \frac{N_0}{M} \nu_1\big(A_M(\ell)\big) \Big).
\]

Notice that events $\{j_{M+1}(Z) = i, \ell \in \cJ_M(Z)\}$ are disjoint for $i \neq \ell$. We then have
\begin{align*}
  & \sum_{i=1,i \neq \ell}^{N_0} \Big[1 + \frac{N_0}{M} \nu_1\big(A_M(\ell)\big) \Big] \Big[\frac{N_0}{M} \P\Big( j_{M+1}(Z) = i, \ell \in \cJ_M(Z) \Biggiven \mX \Big) \Big] \ind \Big( X_i \in \cE \Big)\\
  \le & \sum_{i=1,i \neq \ell}^{N_0} \Big[1 + \frac{N_0}{M} \nu_1\big(A_M(\ell)\big) \Big] \Big[\frac{N_0}{M} \P\Big( j_{M+1}(Z) = i, \ell \in \cJ_M(Z) \Biggiven \mX \Big) \Big]\\
  = & \Big[1 + \frac{N_0}{M} \nu_1\big(A_M(\ell)\big) \Big] \Big[\frac{N_0}{M} \P\Big( \ell \in \cJ_M(Z) \Biggiven \mX \Big) \Big]\\
  = & \Big[1 + \frac{N_0}{M} \nu_1\big(A_M(\ell)\big) \Big] \frac{N_0}{M} \nu_1\big(A_M(\ell)\big).
\end{align*}

From Corollary~\ref{crl:moment,catch,random},
\begin{align}
  \E \Big[\sum_{i=1,i \neq \ell}^{N_0} \Big[1 + \frac{N_0}{M} \nu_1\big(A_M(\ell)\big) \Big] \Big[\frac{N_0}{M} \P\Big( j_{M+1}(Z) = i, \ell \in \cJ_M(Z) \Biggiven \mX \Big) \Big] \ind \Big( X_i \in \cE \Big) \Big]^2 \lesssim C,
\end{align}
where $C>0$ is a constant only depending on $f_L, f_U$.

Let $\Xi_\ell \given \mX$ follows the same distribution as $Z \given \mX, \{Z \in A_M(\ell)\}$ for $Z \sim \nu_1$. Then
\begin{align*}
  & \sum_{i=1,i \neq \ell}^{N_0} \frac{N_0}{M} \nu_1\big(A_M(i)\big) \Big[\frac{N_0}{M} \P\Big( j_{M+1}(Z) = i, \ell \in \cJ_M(Z) \Biggiven \mX \Big) \Big] \ind \Big( X_i \in \cE \Big)\\
  \le & \sum_{i=1,i \neq \ell}^{N_0} \frac{N_0}{M} \nu_1\big(A_M(i)\big) \Big[\frac{N_0}{M} \P\Big( j_{M+1}(Z) = i, \ell \in \cJ_M(Z) \Biggiven \mX \Big) \Big]\\
  = & \sum_{i=1,i \neq \ell}^{N_0} \E \Big[ \Big(\frac{N_0}{M}\Big)^2 \nu_1\big(A_M(i)\big) \ind\Big( j_{M+1}(Z) = i, \ell \in \cJ_M(Z) \Big) \Biggiven \mX \Big]\\
  = & \Big(\frac{N_0}{M}\Big)^2 \E \Big[ \sum_{i=1,i \neq \ell}^{N_0} \nu_1\big(A_M(j_{M+1}(Z))\big) \ind\Big( j_{M+1}(Z) = i, \ell \in \cJ_M(Z) \Big) \Biggiven \mX \Big]\\
  = & \Big(\frac{N_0}{M}\Big)^2 \E \Big[ \nu_1\big(A_M(j_{M+1}(Z))\big) \ind\Big( \ell \in \cJ_M(Z) \Big) \Biggiven \mX \Big]\\
  = & \Big(\frac{N_0}{M}\Big)^2 \E \Big[ \nu_1\big(A_M(j_{M+1}(\Xi_\ell))\big) \Biggiven \mX \Big] \nu_1\big(A_M(\ell)\big) 
\end{align*}
Then
\begin{align*}
  & \E \Big[ \sum_{i=1,i \neq \ell}^{N_0} \frac{N_0}{M} \nu_1\big(A_M(i)\big) \Big[\frac{N_0}{M} \P\Big( j_{M+1}(Z) = i, \ell \in \cJ_M(Z) \Biggiven \mX \Big) \Big] \ind \Big( X_i \in \cE \Big) \Big]^2\\
  \le & \Big(\frac{N_0}{M}\Big)^4 \E \Big[ \E \Big[ \nu_1\big(A_M(j_{M+1}(Z))\big) \ind\Big( \ell \in \cJ_M(Z) \Big) \Biggiven \mX \Big] \Big]^2\\
  = & \Big(\frac{N_0}{M}\Big)^4 \E \Big[ \E \Big[ \nu_1\big(A_M(j_{M+1}(\Xi_\ell))\big) \Biggiven \mX \Big] \nu_1\big(A_M(\ell)\big) \Big]^2\\
  \le & \Big(\frac{N_0}{M}\Big)^4 \Big(\E\Big[\nu_1\big(A_M(j_{M+1}(\Xi_\ell))\big)\Big]^4 \E \Big[\nu_1\big(A_M(\ell)\big)\Big]^4 \Big)^{1/2}\\
  = & \Big(\frac{N_0}{M}\Big)^4 \Big(\E\Big[\nu_1\big(A_M(i)\big)\Big]^4 \E \Big[\nu_1\big(A_M(\ell)\big)\Big]^4 \Big)^{1/2}
\end{align*}
for some $i \neq \ell$. The last step is due to the fact that $[X_i]_{i=1,i \neq \ell}^{N_0}$ are exchangeable and $\ell \in \cJ_M(Z)$ implies $j_{M+1}(Z) \neq \ell$.

By Corollary~\ref{crl:moment,catch,random}, we obtain
\begin{align}
  \E \Big[ \sum_{i=1,i \neq \ell}^{N_0} \frac{N_0}{M} \nu_1\big(A_M(i)\big) \Big[\frac{N_0}{M} \P\Big( j_{M+1}(Z) = i, \ell \in \cJ_M(Z) \Biggiven \mX \Big) \Big] \ind \Big( X_i \in \cE \Big) \Big]^2 \lesssim C,
\end{align}
where $C>0$ is a constant only depending on $f_L, f_U$.

Using the same trick, by Lemma~\ref{lemma:thr} and Corollary~\ref{crl:moment,catch,random}, we have
\begin{align*}
  & \E \Big[ \sum_{i=1,i \neq \ell}^{N_0} \log\Big(\frac{N_0}{M} \nu_1\big(A_M(i)\big) \Big) \ind\Big(\frac{N_0}{M} \nu_1\big(A_M(i)\big) \le 1\Big) \Big[\frac{N_0}{M} \P\Big( j_{M+1}(Z) = i, \ell \in \cJ_M(Z) \Biggiven \mX \Big) \Big] \ind \Big( X_i \in \cE \Big) \Big]^2\\
  \le &  \Big(\frac{N_0}{M}\Big)^2 \Big(\E \Big[ \log^4\Big(\frac{N_0}{M} \nu_1\big(A_M(i)\big) \Big) \ind\Big(\frac{N_0}{M} \nu_1\big(A_M(i)\big) \le 1\Big) \ind \Big( X_i \in \cE \Big) \Big] \E \Big[\nu_1^4\big(A_M(\ell)\big) \Big] \Big)^{1/2}\\
  \le & \Big(\frac{N_0}{M}\Big)^2 \Big(\E \Big[ \log^4\Big(C_1 r(X_i) \Big) \ind\Big(\frac{N_0}{M} \nu_1\big(A_M(i)\big) \le 1\Big) \ind \Big( X_i \in \cE \Big)\Big] \E \Big[ \nu_1^4\big(A_M(\ell)\big) \Big]\Big)^{1/2}\\
  \le & \Big(\frac{N_0}{M}\Big)^2 \Big(\E \Big[ \log^4\Big(C_1 r(X_i) \Big)\Big] \E \Big[ \nu_1^4\big(A_M(\ell)\big) \Big]\Big)^{1/2}\\
  = & \Big[\E \Big[ \log^4\Big(C_1 r(X_i) \Big) \Big] \E \Big[ \Big(\frac{N_0}{M}\Big)^4 \nu_1^4\big(A_M(\ell)\big) \Big] \Big]^{1/2}
  \lesssim C,
\end{align*}
where $C>0$ is a constant only depending on $f_L, f_U, d, f_U'$.
\end{proof}


\subsubsection{Proof of Lemma~\ref{lemma:T3}}

\begin{proof}[Proof of Lemma~\ref{lemma:T3}]
We have
\begin{align*}
  \lvert T_3 \rvert &= \Big\lvert \frac{1}{N_0} \sum_{i=1}^{N_0} \Big[ \phi\Big(\E[\hat{r}_M(X_i) \given X_i]\Big) - \phi\Big(r(X_i)\Big) \Big] \ind \Big( X_i \in \cE \Big) \Big\rvert\\
  &\le \frac{1}{N_0} \sum_{i=1}^{N_0} \Big\lvert \phi\Big(\E[\hat{r}_M(X_i) \given X_i]\Big) - \phi\Big(r(X_i)\Big) \Big\rvert \ind \Big( X_i \in \cE \Big).
\end{align*}

From Lemma~\ref{lemma:thr}, if $X_i \in \cE$, then for $N_0$ sufficiently large,
\[
  \frac{1}{2} r(X_i) \le \E[\hat{r}_M(X_i) \given X_i] \le \frac{3}{2} r(X_i).
\]

From the mean value theorem and the proof of Theorem~\ref{thm:pw-rate}\ref{thm:pw-rate1},
\begin{align*}
  \lvert T_3 \rvert & \lesssim \frac{1}{N_0} \sum_{i=1}^{N_0} \Big[ 1 + \Big\lvert \log r(X_i) \Big\rvert \Big] \Big\lvert \E[\hat{r}_M(X_i) \given X_i] - r(X_i) \Big\rvert \ind \Big( X_i \in \cE \Big)\\
  & \lesssim C \Big(\frac{M}{N_0}\Big)^{1/d} \frac{1}{N_0} \sum_{i=1}^{N_0} \Big[ 1 + \Big\lvert \log r(X_i) \Big\rvert \Big]  \ind \Big( X_i \in \cE \Big),
\end{align*}
where $C>0$ is a constant only depending on $f_L, f_U, L, d, f_U'$.

Then we obtain
\begin{align}\label{eq:T31}
  \E [\lvert T_3 \rvert ] \le C \Big(\frac{M}{N_0}\Big)^{1/d},
\end{align}
where $C>0$ is a constant only depending on $f_L, f_U, L, d, f_U'$.

From the i.i.d.-ness of $[X_i]_{i=1}^{N_0}$,
\begin{align*}
  \Var[T_3] &= \Var \Big[\frac{1}{N_0} \sum_{i=1}^{N_0} \Big[ \phi\Big(\E[\hat{r}_M(X_i) \given X_i]\Big) - \phi\Big(r(X_i)\Big) \Big] \ind \Big( X_i \in \cE \Big) \Big]\\
  &= \frac{1}{N_0} \Var \Big[ \Big[ \phi\Big(\E[\hat{r}_M(X_i) \given X_i]\Big) - \phi\Big(r(X_i)\Big) \Big] \ind \Big( X_i \in \cE \Big) \Big]\\
  &\le \frac{1}{N_0} \E \Big[ \Big[ \phi\Big(\E[\hat{r}_M(X_i) \given X_i]\Big) - \phi\Big(r(X_i)\Big) \Big]^2 \ind \Big( X_i \in \cE \Big) \Big]\\
  & \le \frac{1}{N_0} C \Big(\frac{M}{N_0}\Big)^{2/d},
  \yestag\label{eq:T32}
\end{align*}
where $C>0$ is a constant only depending on $f_L, f_U, L, d, f_U'$. The last step is due to the mean value theorem.

Combining \eqref{eq:T31} and \eqref{eq:T32} completes the proof.
\end{proof}

\subsubsection{Proof of Lemma~\ref{lemma:T4}}

\begin{proof}[Proof of Lemma~\ref{lemma:T4}]

We seperate $S_0 \setminus \cE$ into three sets $\cE_1, \cE_2, \cE_3$, where $\cE_1 = \{x \in S_0: 0 < f_1(x) \le C_0(\frac{M}{N_0})^{1/d}\}$, $\cE_2 = \{x \in \bR^d: f_1(x) = 0, \Delta(x)> \delta_N'\}$, $\cE_3 = \{x \in S_0: f_1(x) = 0, \Delta(x) \le \delta_N'\} \cup \{x \in S_0: f_1(x) > C_0(\frac{M}{N_0})^{1/d}, \Delta(x) \le C_0(\frac{M}{N_0})^{1/d}\}$, where $C_0$ is the constant in Lemma~\ref{lemma:thr}, and $\delta_N'$ is defined in the proof of Theorem~\ref{thm:rate,risk}.

For $\cE_1$, from the Cauchy–Schwarz inquality,
\begin{align*}
  & \E \Big[\frac{1}{N_0} \sum_{i=1}^{N_0} \Big[ \phi\Big(\E[\hat{r}_M(X_i) \given \mX]\Big) - \phi\Big(r(X_i)\Big) \Big] \ind \Big( X_i \in \cE_1 \Big) \Big]^2\\
  \le & \E \Big[ \Big[ \phi\Big(\E[\hat{r}_M(X_i) \given \mX]\Big) - \phi\Big(r(X_i)\Big) \Big]^2 \ind \Big( X_i \in \cE_1 \Big) \Big]
\end{align*}

We take $N_0$ sufficiently large such that $C_0(\frac{M}{N_0})^{1/d} \le f_L e^{-1} \wedge 1$. Then for any $x \in \cE_1$, $r(x) \le e^{-1}$ and $f_1(x) \le 1$.

If $\E[\hat{r}_M(X_i) \given \mX] \le r(X_i)$, then
\begin{align*}
  & \E \Big[ \Big[ \phi\Big(\E[\hat{r}_M(X_i) \given \mX]\Big) - \phi\Big(r(X_i)\Big) \Big]^2 \ind \Big( X_i \in \cE_1 \Big) \Big]\\
  \le & \E \Big[ \phi^2\Big(r(X_i)\Big) \ind \Big( X_i \in \cE_1 \Big) \Big] \le \phi^2\Big(f_L^{-1} C_0 \Big(\frac{M}{N_0}\Big)^{1/d}\Big) \P\Big(0 < f_1(X_i) \le C_0(\frac{M}{N_0})^{1/d} \Big)\\
  \lesssim & C \Big(\frac{M}{N_0}\Big)^{2/d} \log^2\Big(\frac{N_0}{M}\Big) \P\Big(0 < f_1(X_i) \le C_0(\frac{M}{N_0})^{1/d} \Big)\\
  \le & C \Big(\frac{M}{N_0}\Big)^{2/d} \log^2\Big(\frac{N_0}{M}\Big) \P\Big(\Big\lvert \log f_1(X_i) \Big\rvert \ge \Big\lvert \frac{1}{d} \log \Big(\frac{N_0}{M}\Big) - \log C_0 \Big\rvert \Big)\\
  \le & C \Big(\frac{M}{N_0}\Big)^{2/d} \log^2\Big(\frac{N_0}{M}\Big)  \Big( \frac{1}{d} \log \Big(\frac{N_0}{M}\Big) - \log C_0 \Big)^{-4} \E \Big[ \log^4\Big(f_1(X_i) \Big) \Big] \\
  \prec& C \Big(\frac{M}{N_0}\Big)^{2/d},
\end{align*}
where $C>0$ is a constant only depending on $f_L, f_U, L, d$.

If $\E[\hat{r}_M(X_i) \given \mX] > r(X_i)$, then from mean value theorem,
\begin{align*}
  & \E \Big[ \Big[ \phi\Big(\E[\hat{r}_M(X_i) \given \mX]\Big) - \phi\Big(r(X_i)\Big) \Big]^2 \ind \Big( X_i \in \cE_1 \Big) \Big]\\
  \le & \E \Big[ \Big(1 + \log\Big(r(X_i)\Big)\Big)^2 \Big(\E[\hat{r}_M(X_i) \given \mX] - r(X_i)\Big)^2 \ind \Big( X_i \in \cE_1 \Big) \Big]\\
  \lesssim & C \Big(\frac{M}{N_0}\Big)^{2/d} \E \Big[ \Big(1 + \log\Big(r(X_i)\Big)\Big)^2 \ind \Big( X_i \in \cE_1 \Big) \Big]\\
  \lesssim & C \Big(\frac{M}{N_0}\Big)^{2/d} \Big(1 + \E \Big[ \log^4\Big(f_1(X_i) \Big) \Big] \Big) \lesssim C \Big(\frac{M}{N_0}\Big)^{2/d},
\end{align*}
where $C>0$ is a constant only depending on $f_L, f_U, L, d, f_U'$.

We then obtain
\begin{align*}
  \E \Big[\frac{1}{N_0} \sum_{i=1}^{N_0} \Big[ \phi\Big(\E[\hat{r}_M(X_i) \given \mX]\Big) - \phi\Big(r(X_i)\Big) \Big] \ind \Big( X_i \in \cE_1 \Big) \Big]^2 \lesssim C \Big(\frac{M}{N_0}\Big)^{2/d},
\end{align*}
where $C>0$ is a constant only depending on $f_L, f_U, L, d, f_U'$.

For $\cE_2$, notice that $(\log x)^2 < \frac{1}{x} + x$. In the same way as case II of the proof of Theorem~\ref{thm:rate,risk}, for any $\gamma>0$, we have
\begin{align*}
  & \E \Big[\frac{1}{N_0} \sum_{i=1}^{N_0} \Big[ \phi\Big(\E[\hat{r}_M(X_i) \given \mX]\Big) - \phi\Big(r(X_i)\Big) \Big] \ind \Big( X_i \in \cE_2 \Big) \Big]^2\\
  \le & \E \Big[ \Big[ \phi\Big(\E[\hat{r}_M(X_i) \given \mX]\Big) - \phi\Big(r(X_i)\Big) \Big]^2 \ind \Big( X_i \in \cE_2 \Big) \Big]\\
  = & \E \Big[ \phi^2 \Big(\E[\hat{r}_M(X_i) \given \mX]\Big) \ind \Big( X_i \in \cE_2 \Big) \Big] \prec N_0^{-\gamma}.
\end{align*}

For $\cE_3$,
\begin{align*}
  &\E \Big[ \frac{1}{N_0} \sum_{i=1}^{N_0} \Big[ \phi\Big(\E[\hat{r}_M(X_i) \given \mX]\Big) - \phi\Big(r(X_i)\Big) \Big] \ind \Big( X_i \in \cE_3 \Big) \Big]^2\\
  = & \frac{1}{N_0} \E \Big[\Big[\phi\Big(\E[\hat{r}_M(X_i) \given \mX]\Big) - \phi\Big(r(X_i)\Big)\Big]^2 \ind \Big( X_i \in \cE_3 \Big) \Big] + \Big(1-\frac{1}{N_0}\Big) \\
  & \E \Big[\Big[\phi\Big(\E[\hat{r}_M(X_i) \given \mX]\Big) - \phi\Big(r(X_i)\Big)\Big]\Big[\phi\Big(\E[\hat{r}_M(X_j) \given \mX]\Big) - \phi\Big(r(X_j)\Big)\Big] \ind \Big( X_i \in \cE_3 \Big) \ind \Big( X_j \in \cE_3 \Big) \Big].
\end{align*}

Notice that
\[
  \Big\lvert\phi\Big(\E[\hat{r}_M(X_i) \given \mX]\Big)\Big\rvert + \Big\lvert\phi\Big(r(X_i)\Big)\Big\rvert \le \Big(\E[\hat{r}_M(X_i) \given \mX]\Big)^2 + r(X_i)^2 + 2.
\]

Employing the same approach as in the proof of case III in Theorem~\ref{thm:rate,risk}, for any $X_i,X_j \in \cE_3$, one can establish
\[
  \E \Big[ \Big(\E[\hat{r}_M(X_i) \given \mX] \Big)^4 \Biggiven X_i,X_j\Big] \lesssim C,
\]
where $C>0$ is a constant only depending on $f_L, f_U, a$.

Then from the independence of $X_i$ and $X_j$, 
\[
  \E \Big[ \frac{1}{N_0} \sum_{i=1}^{N_0} \Big[ \phi\Big(\E[\hat{r}_M(X_i) \given \mX]\Big) - \phi\Big(r(X_i)\Big) \Big] \ind \Big( X_i \in \cE_3 \Big) \Big]^2 \lesssim C \Big(\frac{M}{N_0}\Big)^{2/d},
\]
where the constant $C>0$ only depending on $f_L,f_U,a,H,d$.
\end{proof}

\subsubsection{Proof of Proposition~\ref{thm:minimax,kl}}

\begin{proof}[Proof of Proposition~\ref{thm:minimax,kl}]

We take $\nu_0$ to be the uniform distribution of an arbitrary support such that the density is $f_L$. Notice that in this case, the KL divergence estimation is reduced to the estimation of differential entropy with $N_1$ samples. We then complete the proof using the minimax lower bound in bounded support Lipschitz entropy estimation without the assumption that the density is bounded away from zero \cite[Theorem 7]{han2020optimal}.
\end{proof}

\subsection{Proofs of results in Section~\ref{sec:main-proof}}

\subsubsection{Proof of Lemma~\ref{lemma:leb,p}}

\begin{proof}[Proof of Lemma~\ref{lemma:leb,p}]

The first inequality is directly from the definition of Lebesgue points. The second inequality follows by
\begin{align*}
  &\Big\lvert \frac{\nu(B_{z,\lVert z-x \rVert})}{\lambda(B_{z,\lVert z-x \rVert})} - f(x) \Big\rvert \le \frac{1}{\lambda(B_{z,\lVert z-x \rVert})} \int_{B_{z,\lVert z-x \rVert}} \lvert f(y) - f(x) \rvert \d y \\
  \le & \frac{1}{\lambda(B_{z,\lVert z-x \rVert})} \int_{B_{x,2\lVert z-x \rVert}} \lvert f(y) - f(x) \rvert \d y = \frac{\lambda(B_{x,2\lVert z-x \rVert})}{\lambda(B_{z,\lVert z-x \rVert})}  \frac{1}{\lambda(B_{x,2\lVert z-x \rVert})} \int_{B_{x,2\lVert z-x \rVert}} \lvert f(y) - f(x) \rvert \d y \\
  = & 2^d \frac{1}{\lambda(B_{x,2\lVert z-x \rVert})} \int_{B_{x,2\lVert z-x \rVert}} \lvert f(y) - f(x) \rvert \d y,
\end{align*}
and then the definition of Lebesgue points.
\end{proof}

\subsubsection{Proof of Lemma~\ref{lemma:z,density}}

\begin{proof}[Proof of Lemma~\ref{lemma:z,density}]

Fix any $(\nu_0,\nu_1) \in \cP_{x,\rm p}(f_L,f_U,L,d,\delta)$. 

We first prove the first claim. First consider $f_1(x)>0$. For any $\epsilon>0$, there exists $\delta'>0$ such that for any $z \in \bR^d$ satisfying $\lVert z-x \rVert \le \delta'$, we have $\lvert f_0(z) - f_0(x) \rvert \le \epsilon f_0(x)$ and $\lvert f_1(z) - f_1(x) \rvert \le \epsilon f_1(x)$ from the local Lipschitz assumption. We take $w>0$ sufficiently small such that $w < (1-\epsilon) f_0(x) \lambda(B_{0,\delta'})$. Then $W \le w$ implies $\lVert x-Z \rVert \le \delta'$. Then for $w>0$ sufficiently small,
\begin{align*}
  \P(W \le w) = \P\Big(W \le w, \lVert x-Z \rVert \le \delta'\Big) \le \P\Big(\frac{1-\epsilon}{1+\epsilon} \frac{f_0(x)}{f_1(x)}  \nu_1(B_{x,\lVert x-Z \rVert}) \le w\Big) = \frac{1+\epsilon}{1-\epsilon}\frac{f_1(x)}{f_0(x)} w,
\end{align*}
and
\begin{align*}
  &\P(W \le w) = \P\Big(W \le w, \lVert x-Z \rVert \le \delta'\Big) \ge \P\Big(\frac{1+\epsilon}{1-\epsilon} \frac{f_0(x)}{f_1(x)}  \nu_1(B_{x,\lVert x-Z \rVert}) \le w, \lVert x-Z \rVert \le \delta \Big) \\
  =& \P\Big(\frac{1+\epsilon}{1-\epsilon} \frac{f_0(x)}{f_1(x)}  \nu_1(B_{x,\lVert x-Z \rVert}) \le w\Big) = \frac{1-\epsilon}{1+\epsilon}\frac{f_1(x)}{f_0(x)} w.
\end{align*}
Then we have
\[
  \frac{1-\epsilon}{1+\epsilon}\frac{f_1(x)}{f_0(x)} \le \liminf_{w \to 0} w^{-1} \P(W \le w) \le \limsup_{w \to 0} w^{-1} \P(W \le w) \le \frac{1+\epsilon}{1-\epsilon}\frac{f_1(x)}{f_0(x)}.
\]
Since $\epsilon$ is arbitrary, we obtain
\[
  f_W(0) = \lim_{w \to 0} w^{-1} \P(W \le w) = \frac{f_1(x)}{f_0(x)} = r(x).
\]
The case for $f_1(x)=0$ can be established in the same way. This completes the proof of the first claim.

For the second claim, for any $0< \epsilon < f_L$, there exists $\delta'>0$ such that for any $z \in \bR^d$ satisfying $\lVert z-x \rVert \le \delta'$, we have $\lvert f_0(z) - f_0(x) \rvert \le \epsilon$ and $\lvert f_1(z) - f_1(x) \rvert \le \epsilon$ from the local Lipschitz assumption. We take $N_0$ sufficiently large such that $2\frac{M}{N_0} < (f_L-\epsilon) \lambda(B_{0,\delta'})$. Then for any $0 < w \le 2\frac{M}{N_0}$, we have $w < (f_L-\epsilon) \lambda(B_{0,\delta'})$. We take $t >0$ such that $w+t < (f_L - \epsilon) \lambda(B_{0,\delta'})$. Then for any $(\nu_0,\nu_1) \in \cP_{x,\rm p}(f_L,f_U,L,d,\delta)$,
\begin{align*}
  &\P\Big( w \le W \le w+t \Big) = \nu_1\Big(\Big\{z \in \bR^d: \nu_0(B_{z,\lVert x-z \rVert}) \in [w,w+t]\Big\}\Big)\\
  \le & \frac{f_1(x)+\epsilon}{f_0(x) - \epsilon} \nu_0\Big(\Big\{z \in \bR^d: \nu_0(B_{z,\lVert x-z \rVert}) \in [w,w+t]\Big\}\Big).
\end{align*}
Notice that $f_0$ is lower bounded by $f_L$. Then for $N_0$ sufficiently large,
\[
  \limsup_{t \to 0} t^{-1} \P\Big( w \le W \le w+t \Big) \le \frac{f_1(x)+\epsilon}{f_0(x) - \epsilon} (1+\epsilon).
\]
This then completes the proof.
\end{proof}

\subsubsection{Proof of Lemma~\ref{lemma:ratevar}}

\begin{proof}[Proof of Lemma~\ref{lemma:ratevar}]
Due to the i.i.d.-ness of $Z$ and $\tZ$,
\begin{align*}
  & \Big(\frac{N_0}{M}\Big)^2 \Big[\P \Big(W \le V, \tW \le \tV, W \le 2\frac{M}{N_0}, \tW \le 2\frac{M}{N_0} \Big) - \P \Big(W \le V, W \le 2\frac{M}{N_0}\Big) \P \Big(\tW \le \tV, \tW \le 2\frac{M}{N_0}\Big) \Big]\\
  =& \Big(\frac{N_0}{M}\Big)^2 \int_0^{2\frac{M}{N_0}} \int_0^{2\frac{M}{N_0}} \Big[\P \Big(V \ge w_1, \tV \ge w_2\Big) - \P \Big(V \ge w_1 \Big) \P \Big(\tV \ge w_2\Big) \Big] f_W(w_1)f_W(w_2) \d w_1 \d w_2\\
  \le& 4 \Big(\frac{f_U}{f_L}\Big)^2 \Big(\frac{N_0}{M}\Big)^2 \int_0^{2\frac{M}{N_0}} \int_0^{2\frac{M}{N_0}} \Big\lvert\P \Big(V \ge w_1, \tV \ge w_2\Big) - \P \Big(V \ge w_1 \Big) \P \Big(\tV \ge w_2\Big) \Big\rvert \d w_1 \d w_2\\
  =& 4 \Big(\frac{f_U}{f_L}\Big)^2 \int_{-1}^1 \int_{-1}^1 \Big\lvert\P \Big(V \ge \frac{M}{N_0}(1+t_1), \tV \ge \frac{M}{N_0}(1+t_2)\Big) - \P \Big(V \ge \frac{M}{N_0}(1+t_1) \Big) \P \Big(\tV \ge \frac{M}{N_0}(1+t_2)\Big) \Big\rvert \d t_1 \d t_2,
\end{align*}
where the last step is from taking $w_1 = \frac{M}{N_0}(1+t_1)$ and $w_2 = \frac{M}{N_0}(1+t_2)$.

Let
\[
  S(t_1,t_2):= \Big\lvert \P \Big(V \ge \frac{M}{N_0}(1+t_1), \tV \ge \frac{M}{N_0}(1+t_2)\Big) - \P \Big(V \ge \frac{M}{N_0}(1+t_1) \Big) \P \Big(\tV \ge \frac{M}{N_0}(1+t_2)\Big) \Big\rvert.
\]
If $t_1 \ge t_2 \ge 0$, 
\[
  S(t_1,t_2) \le \P \Big(V \ge \frac{M}{N_0}(1+t_1)\Big) = \P \Big(U_{(M)} \ge \frac{M}{N_0}(1+t_1)\Big).
\]
If $t_2 \ge t_1 \ge 0$, 
\[
  S(t_1,t_2) \le \P \Big(\tV \ge \frac{M}{N_0}(1+t_2)\Big) = \P \Big(U_{(M)} \ge \frac{M}{N_0}(1+t_2)\Big).
\]
Then for $t_1,t_2 \ge 0$,
\[
  S(t_1,t_2) \le \P \Big(U_{(M)} \ge \frac{M}{N_0}(1+t_1 \vee t_2)\Big).
\]
If $t_1 \le t_2 \le 0$ and $\P \Big(V \ge \frac{M}{N_0}(1+t_1), \tV \ge \frac{M}{N_0}(1+t_2)\Big) \ge \P \Big(V \ge \frac{M}{N_0}(1+t_1) \Big) \P \Big(\tV \ge \frac{M}{N_0}(1+t_2)\Big)$,
\begin{align*}
  & S(t_1,t_2) \le \P \Big(\tV \ge \frac{M}{N_0}(1+t_2)\Big) - \P \Big(V \ge \frac{M}{N_0}(1+t_1) \Big) \P \Big(\tV \ge \frac{M}{N_0}(1+t_2)\Big) \\
  =& \P \Big(V \le \frac{M}{N_0}(1+t_1) \Big) \P \Big(\tV \ge \frac{M}{N_0}(1+t_2)\Big) \le \P \Big(V \le \frac{M}{N_0}(1+t_1) \Big) = \P \Big(U_{(M)} \le \frac{M}{N_0}(1+t_1) \Big).
\end{align*}
If $t_1 \le t_2 \le 0$ and $\P \Big(V \ge \frac{M}{N_0}(1+t_1), \tV \ge \frac{M}{N_0}(1+t_2)\Big) \le \P \Big(V \ge \frac{M}{N_0}(1+t_1) \Big) \P \Big(\tV \ge \frac{M}{N_0}(1+t_2)\Big)$,
\begin{align*}
  &S(t_1,t_2) \le \P \Big(\tV \ge \frac{M}{N_0}(1+t_2)\Big) - \P \Big(V \ge \frac{M}{N_0}(1+t_1), \tV \ge \frac{M}{N_0}(1+t_2)\Big) \\
  = & \P \Big(V \le \frac{M}{N_0}(1+t_1), \tV \ge \frac{M}{N_0}(1+t_2)\Big) \le \P \Big(V \le \frac{M}{N_0}(1+t_1) \Big) = \P \Big(U_{(M)} \le \frac{M}{N_0}(1+t_1) \Big).
\end{align*}
If $t_2 \le t_1 \le 0$, we can establish in the same way that
\[
  S(t_1,t_2) \le \P \Big(U_{(M)} \le \frac{M}{N_0}(1+t_2) \Big).
\]
Then for $t_1,t_2 \le 0$,
\[
  S(t_1,t_2) \le \P \Big(U_{(M)} \le \frac{M}{N_0}(1+t_1 \wedge t_2)\Big).
\]
For $t_1 \ge 0 \ge t_2$, if $t_1+t_2 \ge 0$,
\[
  S(t_1,t_2) \le \P \Big(U_{(M)} \ge \frac{M}{N_0}(1+t_1)\Big),
\]
and if $t_1+t_2 \le 0$,
\[
  S(t_1,t_2) \le \P \Big(U_{(M)} \le \frac{M}{N_0}(1+t_2) \Big).
\]
Then
\begin{align*}
  & \Big(\frac{N_0}{M}\Big)^2 \Big[\P \Big(W \le V, \tW \le \tV, W \le 2\frac{M}{N_0}, \tW \le 2\frac{M}{N_0} \Big) - \P \Big(W \le V, W \le 2\frac{M}{N_0}\Big) \P \Big(\tW \le \tV, \tW \le 2\frac{M}{N_0}\Big) \Big]\\
  \le& 4 \Big(\frac{f_U}{f_L}\Big)^2 \int_{-1}^1 \int_{-1}^1 S(t_1,t_2) \d t_1 \d t_2\\
  =& 4 \Big(\frac{f_U}{f_L}\Big)^2 \Big[ \int_0^1 \int_0^1 S(t_1,t_2) \d t_1 \d t_2 + \int_{-1}^0 \int_{-1}^0 S(t_1,t_2) \d t_1 \d t_2 \\
  & + \int_{-1}^0 \int_0^1 S(t_1,t_2) \d t_1 \d t_2 + \int_0^1 \int_{-1}^0 S(t_1,t_2) \d t_1 \d t_2\Big]\\
  =& 4 \Big(\frac{f_U}{f_L}\Big)^2 \Big[ \int_0^1 \int_0^1 S(t_1,t_2) \d t_1 \d t_2 + \int_{-1}^0 \int_{-1}^0 S(t_1,t_2) \d t_1 \d t_2 + 2\int_0^1 \int_{-1}^0 S(t_1,t_2) \d t_1 \d t_2\Big],
  \yestag\label{eq:ratevar3} 
\end{align*}
where the last step is from the symmetry of $S(t_1,t_2)$.

For the first term in \eqref{eq:ratevar3}, by the symmetry of $S(t_1,t_2)$ and the Chernoff bound,
\begin{align*}
  & \int_0^1 \int_0^1 S(t_1,t_2) \d t_1 \d t_2 \le  \int_0^{\infty} \int_0^{\infty} S(t_1,t_2) \d t_1 \d t_2 =  2 \int_0^{\infty} \int_0^{\infty} S(t_1,t_2) \ind(t_1 \ge t_2) \d t_1 \d t_2\\
  \le & 2 \int_0^{\infty} \int_0^{\infty} \P \Big(U_{(M)} \ge \frac{M}{N_0}(1+t_1 \vee t_2)\Big) \ind(t_1 \ge t_2) \d t_1 \d t_2 = 2 \int_0^{\infty} t \P \Big(U_{(M)} \ge \frac{M}{N_0}(1+t)\Big) \d t\\
  \le & 2 \int_0^{\infty} t (1+t)^M e^{-Mt} \d t.
\end{align*}
Notice that 
\begin{align*}
  &\int_0^{\infty} t (1+t)^M e^{-Mt} \d t = \int_0^{\infty} (1+t)^{M+1} e^{-Mt} \d t - \int_0^{\infty} (1+t)^M e^{-Mt} \d t\\
  =& -\frac{1}{M} \Big( -1 - (M+1) \int_0^{\infty} (1+t)^M e^{-Mt} \d t \Big) - \int_0^{\infty} (1+t)^M e^{-Mt} \d t\\
  =& \frac{1}{M} + \frac{1}{M} \int_0^{\infty} (1+t)^M e^{-Mt} \d t = \frac{1}{M} + \frac{e^M}{M} \int_1^{\infty} t^M e^{-Mt} \d t \le \frac{1}{M} + \frac{e^M}{M} \int_0^{\infty} t^M e^{-Mt} \d t\\
  =& \frac{1}{M} + \frac{e^M}{M} \frac{1}{M^{M+1}} \int_0^{\infty} t^M e^{-t} \d t = \frac{1}{M} + \frac{e^M}{M^{M+2}} \Gamma(M+1) = \frac{1}{M} + \frac{e^M}{M^{M+2}} \sqrt{2\pi M}\Big(\frac{M}{e}\Big)^M(1+o(1))\\
  =& \frac{1}{M}(1+o(1)),
  \yestag\label{eq:ratevar4}
\end{align*}
where the second last inequality from Stirling's approximation since $M \to \infty$. We then obtain
\begin{align}\label{eq:ratevar6}
  \int_0^1 \int_0^1 S(t_1,t_2) \d t_1 \d t_2 \le \frac{2}{M}(1+o(1)).
\end{align}

For the second term in \eqref{eq:ratevar3},
\begin{align*}
  & \int_{-1}^0 \int_{-1}^0 S(t_1,t_2) \d t_1 \d t_2 = 2 \int_{-1}^0 \int_{-1}^0 S(t_1,t_2) \ind(t_1 \le t_2) \d t_1 \d t_2 \le 2 \int_{-1}^0 \int_{-1}^0 S(t_1,t_2) \ind(t_1 \le t_2) \d t_1 \d t_2\\
  \le & 2 \int_{-1}^0 \int_{-1}^0 \P \Big(U_{(M)} \le \frac{M}{N_0}(1+t_1 \wedge t_2)\Big) \ind(t_1 \le t_2) \d t_1 \d t_2 = 2 \int_{-1}^0 (-t) \P \Big(U_{(M)} \le \frac{M}{N_0}(1+t)\Big) \d t\\
  =& 2 \int_0^1 t \P \Big(U_{(M)} \le \frac{M}{N_0}(1-t)\Big) \d t \le 2\int_0^1 t(1-t)^M e^{Mt} \d t.
\end{align*}
Notice that
\begin{align*}
  &\int_0^1 t(1-t)^M e^{Mt} \d t =  -\int_0^1 (1-t)^{M+1} e^{Mt} \d t + \int_0^1 (1-t)^M e^{Mt} \d t\\
  =& -\frac{1}{M} \Big( -1 + (M+1) \int_0^1 (1-t)^M e^{Mt} \d t \Big) + \int_0^1 (1-t)^M e^{Mt} \d t \le \frac{1}{M}.
  \yestag\label{eq:ratevar5}
\end{align*}
We then obtain
\begin{align}\label{eq:ratevar7}
  \int_{-1}^0 \int_{-1}^0 S(t_1,t_2) \d t_1 \d t_2 \le \frac{2}{M}.
\end{align}

For the third term in \eqref{eq:ratevar3},
\begin{align*}
  & \int_0^1 \int_{-1}^0 S(t_1,t_2) \d t_1 \d t_2 \\
  =& \int_0^1 \int_{-t_1}^0 \P \Big(U_{(M)} \ge \frac{M}{N_0}(1+t_1)\Big) \d t_1 \d t_2 + \int_0^1 \int_{-1}^{-t_1} \P \Big(U_{(M)} \le \frac{M}{N_0}(1+t_2)\Big) \d t_1 \d t_2\\
  = & \int_0^1 t \P \Big(U_{(M)} \ge \frac{M}{N_0}(1+t)\Big) \d t + \int_{-1}^0 (-t) \P \Big(U_{(M)} \le \frac{M}{N_0}(1+t)\Big) \d t\\
  \le & \int_0^{\infty} t \P \Big(U_{(M)} \ge \frac{M}{N_0}(1+t)\Big) \d t + \int_{-1}^0 (-t) \P \Big(U_{(M)} \le \frac{M}{N_0}(1+t)\Big) \d t\\
  \le & \frac{1}{M}(1+o(1)) + \frac{1}{M} = \frac{2}{M}(1+o(1)),
\end{align*}
where the last step is from \eqref{eq:ratevar4} and \eqref{eq:ratevar5}.

We then obtain
\begin{align}\label{eq:ratevar8}
  \int_0^1 \int_{-1}^0 S(t_1,t_2) \d t_1 \d t_2 \le \frac{2}{M}(1+o(1)).
\end{align}

Plugging \eqref{eq:ratevar6}, \eqref{eq:ratevar7}, \eqref{eq:ratevar8} into \eqref{eq:ratevar3} yields
\begin{align*}
  & \Big(\frac{N_0}{M}\Big)^2 \Big[\P \Big(W \le V, \tW \le \tV, W \le 2\frac{M}{N_0}, \tW \le 2\frac{M}{N_0} \Big) - \P \Big(W \le V, W \le 2\frac{M}{N_0}\Big) \P \Big(\tW \le \tV, \tW \le 2\frac{M}{N_0}\Big) \Big]\\
  \le & 32 \Big(\frac{f_U}{f_L}\Big)^2 \frac{1}{M} (1+o(1)),
  \yestag\label{eq:ratevar10}
\end{align*}
and thus completes the proof.
\end{proof}

{
\bibliographystyle{apalike}
\bibliography{AMS}
}

\end{document}